\documentclass[11p,handout]{article}

\usepackage[utf8x]{inputenc}

\usepackage{amsmath}
\usepackage{amssymb}
\usepackage{tikz}
\usepackage[dvipsnames]{xcolor}
\usepackage{caption}
\usepackage{amsopn}
\usepackage{caption}
\usepackage{algorithm}
\usepackage{algpseudocode}

\usepackage{amsthm}
\usepackage{comment}
\theoremstyle{plain} 
\newtheorem{lemma}{Lemma}[section]

\usepackage{hyperref} 
\hypersetup{
	colorlinks=true,                          
	linkcolor=blue, 
	citecolor=red, 
	urlcolor=blue  } 

\date{}
\title{A structure-preserving semi-implicit finite volume scheme on vertex-staggered unstructured meshes}

\author{Elena Bernardelli\thanks{Department of Computer Science, University of Verona, Strada le Grazie 15, Verona, 37134, Italy 
		(elena.bernardelli@univr.it)}
	\and Elena Gaburro\thanks{Department of Computer Science, University of Verona, Strada le Grazie 15, Verona, 37134, Italy 
		(elena.gaburro@univr.it)}
	\and Michael Dumbser \thanks{Laboratory of Applied Mathematics, University of Trento, Via Mesiano 77, 38123 Trento, Italy (michael.dumbser@unitn.it)}
	\thanks{Department of Mathematics, Southern University of Science and Technology, Xueyuan Avenue 1088, 518055 Shenzhen, Guangdong, China}}

\begin{document} 

\maketitle	

\begin{abstract} 
				We present a novel structure-preserving semi-implicit finite volume method on vertex-based staggered meshes for the compatible discretization of first-order systems of time-dependent partial differential equations (PDE). 
				The method preserves divergence-free and curl-free vector fields exactly thanks to the  \textit{compatible vertex-staggered} discretization of the state variables on unstructured grids that are constituted by primal Delaunay triangles and their dual polygons. 
				The new scheme applies to a broad range of PDE, including the weakly compressible and incompressible Euler and Navier-Stokes equations, incompressible magnetohydrodynamics (MHD), and the incompressible hyperbolic Godunov-Peshkov-Romenski (GPR) model for continuum mechanics. 
				For the weakly compressible Euler equations, the proposed semi-implicit scheme is also asymptotic preserving (AP) in the low Mach number limit, i.e. it provides a compatible discretization of the incompressible Euler equations when the Mach number tends to zero. 
				
				The computational domain is covered by a primal triangular mesh and a dual tessellation made of so-called \textit{star} polygons constructed by linking the barycenters of adjacent triangles and the midpoints of the shared edges.
				The scalar pressure field, the density, and the viscous stress tensor are defined at the nodes of the triangles. In particular, the nodal pressure is evolved implicitly in a continuous finite element-type fashion, yielding a symmetric and positive definite pressure system. The velocity, the momentum, the magnetic field and the distortion field are instead stored at the barycenters of the triangles and are evolved explicitly with a compatible finite volume approach. 
				Thanks to the semi-implicit discretization, the CFL condition does not depend on the sound speed, thus enabling simulations also at \textit{low} Mach numbers.  
				The compatible semi-implicit vertex-staggered finite volume discretization ensures an exactly  \textit{divergence-free} velocity field in the incompressible limit, an exactly divergence-free magnetic field for the MHD equations and an exactly \textit{curl-free} inverse deformation gradient for solid mechanics.  For the incompressible systems under consideration here, we prove fully-discrete energy stability of the first order version of our method.
				Via the combination of a simple piecewise-linear polynomial reconstruction and a second order Runge--Kutta time integrator, the method has been extended to second order of accuracy in space and time for the incompressible Euler, MHD and GPR equations.
				The scheme is validated on a set of classical benchmark problems in order to illustrate its structure-preserving properties on the considered PDE systems. 
\end{abstract}
			
\textbf{Keywords:}
				unstructured vertex-staggered semi-implicit finite volume scheme, 
				structure-preserving and asymptotic preserving, fully compatible discretization, 
				divergence-free, 
				curl-free,  
				energy stability 

	\section{Introduction}
	%
	
	
	Fluids, solids and plasmas are involved in numerous natural phenomena as well as industrial and biological processes including, e.g., geophysical applications such as weather forecasting, pollution modeling and seismic wave propagation, blood flow in the human cardiovascular system, clean energy production via fusion or aerodynamic design of vehicles, aircraft and spacecraft. As a consequence, understanding their behavior is of great importance for the development of modern society. From the mathematical point of view, fluids are modeled with the Euler equations, which are directly derived from the fundamental principles of mass and momentum conservation. Then, under appropriate assumptions for the viscous stress tensor, the Navier-Stokes equations are derived, which constitute the most widely used model for the simulation of viscous flows. 
	Additionally, the magnetohydrodynamics (MHD) system generalizes the fluid dynamics framework to account for the interactions between electrically conducting fluids and magnetic fields occurring in plasma flows. This system couples the Euler equations with the Maxwell equations to describe the evolution of the magnetic field and plays a crucial role in astrophysics, plasma physics, and fusion energy research. 
	On the other hand, the Godunov-Peshkov-Romenski (GPR) model \cite{PeshRom2014,HPRmodel,HPRmodelMHD} provides a unified first-order hyperbolic formulation for continuum mechanics, encompassing both fluid and solid-like behavior of the continuum. The GPR model introduces additional evolution equations for the distortion field, allowing the system to capture elastodynamics and plasticity within a unified framework. 
	This makes it particularly well-suited for simulating complex multi-material flows, wave propagation in heterogeneous media and high-speed impact problems. 
	
	The behavior of a continuum is generally classified based on the dimensionless \textit{Mach number}, defined as $M = \|\mathbf{u}\|/ c_0$, where $\mathbf{u}$ is the flow velocity and $c_0$ is the speed of sound in the medium. In the limit $M \rightarrow 0$, the fluid behaves as an incompressible medium, satisfying the well-known divergence-free condition $\nabla \cdot \mathbf{u} = 0$. This asymptotic behavior was rigorously studied by Klainerman and Majda in~\cite{klainerman1981singular, klainerman1982compressible}.
	
	Traditionally, numerical methods for fluid flows have been divided into two main families: \textit{pressure-based} and \textit{density-based} solvers. The difference between these two kinds of numerical methods was initially motivated by the compressibility properties of the flow.
	Typically, the \textit{incompressible} Euler and Navier-Stokes equations are solved via \textit{semi-implicit pressure-based} schemes of finite difference type on staggered grids, see e.g.~\cite{harlow1965numerical, casulli1984pressure,chorin1997numerical, chorin1968numerical, patankar1983calculation, patankar2018numerical, van1986second, bell1989second, hirt1981volume} or with the aid of \textit{continuous finite elements}~\cite{taylor1973numerical, brooks1982streamline, hughes1986new, fortin1981old, verfurth1986finite, heywood1982finite, heywood1988finite}. 
	Conversely, \textit{compressible} flows at higher Mach numbers are more effectively handled by \textit{explicit density-based} finite volume schemes of the Godunov type~\cite{lax2005systems, godunov1959finite, roe1981approximate, osher1982upwind, harten1983upstream, einfeldt1991godunov, munz1994godunov, toro1994restoration, leveque2002finite, toro2013riemann}.
	A first attempt to generalize \textit{pressure-based} schemes to the more general case of compressible flows was made by Casulli and Greenspan~\cite{casulli1984pressure}, but the proposed method was not conservative. Semi-implicit schemes that explicitly make use of the low-Mach-number asymptotic limit of the governing partial differential equations can be found in~\cite{meister1999asymptotic, munz2003extension, klein1995semi, klein2001asymptotic}, while the first conservative staggered semi-implicit pressure-based scheme for compressible flows was introduced by Park and Munz in~\cite{park2005multiple}. The scheme~\cite{park2005multiple} can be considered as one of the first \textit{all Mach number} flow solvers ever proposed in the literature. The particular splitting of the explicit convective terms and implicit pressure term used in~\cite{park2005multiple} was later studied in more detail in~\cite{toro2012flux} in order to construct a novel flux-vector splitting method, which was the basis of the schemes proposed in~\cite{dumbser2016conservative,zampa2025asymptotic,dumbser2019divergence}.
	
	For the numerical solution of the MHD equations, \textit{exactly} divergence-free schemes usually employ \textit{staggered meshes} \cite{yee1966numerical,BalsaraSpicer1999,Balsara2004,balsarahlle2d,MUSIC1}, where the location of the discrete electric and magnetic field is chosen in a suitable manner in combination with compatible discrete differential operators that guarantee that the discrete magnetic field remains exactly divergence-free if it was initially divergence-free. 
	
	But staggered grids are not only used in the context of structure-preserving schemes. They are also widely used in the context of semi-implicit schemes.
	The need for structure-preserving schemes for hyperbolic conservation laws is very well known in the context of numerical methods for the solution of hyperbolic partial differential equations. 
	Common to almost all the exactly structure-preserving schemes is the fact that they require the use of a staggered grid in order to provide natural and compatible definitions of the discrete curl, gradient and divergence operators~\cite{abgrall2023simple, abgrall2025simple, boscheri2024structure, boscheri2025structure, rio2025exactly} which should respect the two basic vector calculus identities exactly also at the discrete level, namely that \textit{curl of gradient} and \textit{divergence of curl} are identically \textit{zero}~\cite{abgrall2025simple, boscheri2024new, peshkov2021simulation}. The two basic vector calculus identities
	\begin{equation}
		\nabla \times \nabla \phi = 0, \hspace{1cm} \nabla \cdot \nabla \times \mathbf{J} = 0,
		\label{basis vector calculus identities}
	\end{equation}
	with $\phi$ a scalar field and $\mathbf{J}$ a vector field are an immediate consequence of the Schwarz theorem on the symmetry of second partial derivatives
	\begin{equation}
		\partial_j\partial_k \phi = \partial_k\partial_j\phi, \hspace{1cm} \partial _j\partial_kJ_m = \partial_k \partial_jJ_m.
		\label{partial derivatives symmetry}
	\end{equation}
	This can be easily seen in Einstein index notation, assuming summation over two repeated indices, and using a classical fully antisymmetric Levi-Civita tensor $\epsilon_{ijk}$ for which it is well known that its contraction with a symmetric tensor is zero ($\epsilon_{ijk}T_{jk} = 0$ for $T_{jk} = T_{kj}$ since $\epsilon_{ijk} = - \epsilon_{ikj}$):
	\begin{equation}
		\nabla \times \nabla \phi = \epsilon_{ijk}\partial _j\partial_k\phi = 0, \hspace{1cm} \nabla \cdot \nabla \times \mathbf{J} = \partial_i(\epsilon_{ijk}\partial_jJ_k) = \epsilon_{ijk}\partial_i\partial_jJ_k = 0,
	\end{equation}
	due to the symmetry of second partial derivatives \eqref{partial derivatives symmetry} and because the contraction of the antisymmetric Levi-Civita tensor with a symmetric tensor is zero.
	
	Without claiming completeness, we refer the reader to the well-known mimetic finite differences~\cite{hyman1997natural, margolin2000discrete, lipnikov2014mimetic}, compatible finite volume schemes~\cite{park2005multiple, balsara2018computational, hazra2019globally} and compatible finite element~\cite{nedelec1980mixed, zhang2005new, hiptmair2002finite, arnold2006finite, monk2003finite, rodriguez2015finite, campos2016gauss, di2023arbitrary, bruno2022unisolvent,zampa2023using} schemes and references therein.
	Moreover, compatible schemes that preserve the curl-free property of a vector field exactly at the discrete level are still quite rare, but important steps forward in this direction have been made in~\cite{jeltsch2006curl,balsara2023curl,boscheri2021structure,chiocchetti2023exactly, jung2024curl, Barsukow2024}. Concerning compatible semi-implicit schemes, we mention \cite{dumbser2019divergence,boscheri2024structure,fambri2021novel} and \cite{boscheri2021structure,FourSplit}.  
	
	The aim of this paper is to present a novel fully compatible, asymptotic-preserving semi-implicit vertex-staggered finite volume method for the numerical solution of the Euler and Navier-Stokes equations, the incompressible MHD equations and the incompressible GPR model. The proposed scheme is both structure-preserving (SP) and asymptotic-preserving (AP), making it also suitable for low Mach number flows in the weakly compressible regime. Explicit density-based solvers suffer from inefficiencies and inaccuracies in the low Mach number regime, necessitating implicit time discretization. However, fully implicit schemes in general lead to highly nonlinear, non-symmetric algebraic systems with a large number of unknowns (density, velocity, pressure, distortion field, and magnetic field). Our semi-implicit approach circumvents these issues by employing: (i) an implicit discretization only for the pressure term and (ii) an explicit discretization for all other terms, especially of the nonlinear convective terms. 
	For weakly compressible flows, our discretization results in a mildly nonlinear system for the pressure as the only unknown, which is linear, symmetric and positive definite for the isothermal equation of state used in this paper. Moreover, the time step restriction of our method is dictated by the explicit subsystem rather than by the speed of sound, ensuring accuracy and robustness also in the incompressible limit when the Mach number tends to zero. Second order of accuracy in space and time is achieved with the aid of a piecewise linear reconstruction in space in combination with a second order Runge--Kutta time integrator. 
	
	The rest of the paper is organized as follows: in Section~\ref{sec.equation} we present the different systems of partial differential equations treated in this paper, in Section~\ref{sec.method} we introduce the domain discretization and the definition of our compatible discrete differential operators together with their properties. After that, we derive the discretization of the governing equations via the new structure-preserving semi-implicit finite volume scheme. In Section~\ref{sec.results}, an extensive validation through classical benchmark test cases is shown. The conclusions and an outlook on future work are given in Section~\ref{conclusion}.
	
	\section{Governing PDE systems}
	
	\label{sec.equation}
	
	As stated before, the new method introduced here applies to a rather broad spectrum of PDE. In the following, we briefly introduce the governing equations treated in this paper.
	
	%
	
	
	\subsection{Incompressible Euler equations}
	Let $\rho = \textnormal{const}$ be a constant fluid density, $\mathbf{u} = (u, v)$ the velocity vector and $p$ the fluid pressure. The incompressible Euler equations read 
	\begin{equation}
		\begin{aligned}
			\nabla \cdot \mathbf{u} & = 0, \\
			\frac{\partial (\rho \mathbf{u})}{\partial t} + \nabla \cdot (\rho \mathbf{u} \otimes \mathbf{u}) + \nabla p &= 0,
		\end{aligned}
		\label{eqn.inc.euler}
	\end{equation}
	and are of mixed hyperbolic-elliptic type due to the divergence-free condition of the velocity field. 
	
	\subsection{Weakly compressible Euler and Navier-Stokes equations}
	Let $\rho$ denote the density and $\mathbf{u} = (u, v)$ the velocity vector of the fluid. The weakly compressible Euler and Navier-Stokes equations expressing the conservation of mass and momentum are given in  conservation form by 
	\begin{equation}
		\begin{aligned}
			\frac{\partial \rho}{\partial t} + \nabla \cdot (\rho \mathbf{u}) &= 0, \\
			\frac{\partial (\rho \mathbf{u})}{\partial t} + \nabla \cdot (\rho \mathbf{u} \otimes \mathbf{u}) + \nabla p + \nabla \cdot \boldsymbol{\sigma} &= 0,
		\end{aligned}
	\end{equation}
	where $p$ represents the pressure, 
	$\boldsymbol
	{\sigma} = -\mu (\nabla \mathbf{u} + \nabla \mathbf{u}^T - \frac{2}{3} (\nabla \cdot \mathbf{u}) \mathbf{I})$ is the viscous stress tensor with dynamic viscosity $\mu$, and $\mathbf{I}$ is the identity matrix. 
	For simplicity, throughout this paper we assume the following relation between density and pressure:
	\begin{equation}
		p = p(\rho) = \rho c_0^2, 
	\end{equation}
	where $c_0 = \textnormal{const}$ is the isothermal sound speed. We furthermore introduce the Mach number
	\begin{equation}
		M = \frac{\left\| \mathbf{u} \right\|}{c_0}.
	\end{equation}
	In the limit of \textit{incompressible} flows when $M \to 0$, see \cite{klainerman1981singular,klainerman1982compressible,meister1999asymptotic,munz2003extension}, the velocity field becomes divergence-free, i.e. $\nabla \cdot \mathbf{u} \to 0$, and the stress tensor can be simplified to $\boldsymbol {\sigma} = -\mu \nabla \mathbf{u}$. The divergence-free property of the velocity in the limit $M \to 0$ will play a key role in validating our new numerical scheme.  In the particular case when the dynamic viscosity $\mu$ is equal to zero, we recover the \textit{isothermal} Euler equations. In this paper we  restrict ourselves to weakly compressible flows, since the main objective is to show the asymptotic-preserving (AP) behavior of the chosen compatible discretization in the limit $M \to 0$, while an extension to the full compressible Euler equations along the ideas outlined in \cite{park2005multiple,dumbser2016conservative} is left to future work.  
	
	\subsection{Incompressible MHD equations}
	The incompressible magnetohydrodynamics (MHD) equations describe the dynamics of electrically conducting fluids in the presence of magnetic fields. The governing equations read 
	\begin{equation}
		\begin{aligned}
			\frac{\partial \rho\mathbf{u}}{\partial t}
			+ \nabla \cdot \left(
			\rho \mathbf{u} \otimes \mathbf{u}		
			- \mathbf{B} \otimes \mathbf{B}
			+ \frac{1}{2} \mathbf{B}^2 \mathbf{I} 
			\right)
			+ \nabla p &= 0, \\
			\frac{\partial \mathbf{B}}{\partial t} + \nabla \times  \mathbf{E}
			&= 0,\\
			\nabla \cdot \mathbf{u} &= 0.
		\end{aligned}
		\label{eqn.inc.mhd}
	\end{equation}
	Here the density $\rho$ is assumed to be constant, $\mathbf{B}$ is the magnetic field obtained from a vector potential $\mathbf{A}$ as $\mathbf{B} = \nabla \times \mathbf{A}$ and $\mathbf{E}$ is the electric field given by $\mathbf{E} = - \mathbf{u} \times \mathbf{B}$. 
	In the case of the MHD system, the magnetic field must remain divergence-free, as no magnetic monopoles can exist, i.e. we have the following involution, 
	\begin{equation}
		\nabla \cdot \mathbf{B} = 0,
	\end{equation}
	which is automatically satisfied for all times if it is satisfied initially. 
	
	\subsection{Incompressible GPR system}
	The Godunov-Peshkov-Romenski (GPR) model \cite{PeshRom2014} of continuum mechanics in the incompressible limit reads 
	\begin{equation}
		\begin{aligned}
			\frac{\partial \rho \mathbf{u}}{\partial t} + \nabla \cdot (\rho \mathbf{u} \otimes \mathbf{u}  ) + \nabla p + \nabla \cdot \boldsymbol{\sigma} &= 0, \\
			\frac{\partial \mathbf{A}}{\partial t} + \nabla (\mathbf{A}\mathbf{u}) &= 0,\\
			\nabla \cdot \mathbf{u} &= 0,
		\end{aligned}
		\label{equation GPR}
	\end{equation}
	where $\mathbf{A}$ is the distortion field and $\boldsymbol{\sigma}$ is the shear stress tensor. 
	For elastic solids the distortion field $\mathbf{A}=A_{ik} = \partial X_i / \partial x_k$ represents the inverse deformation gradient, with $X_i$ being the Lagrangian coordinates in the reference configuration and $x_k$ being the Eulerian coordinates. Since $\mathbf{A}$ is a gradient, it must obviously be curl-free, hence the following condition (involution) on $\mathbf{A}$ holds for all times: 
	\begin{equation}
		\partial_m A_{ik} - \partial_{k} A_{im} = 0, \qquad \textnormal{or, equivalently}, \qquad 
		\nabla \times \mathbf{A}=0. 
	\end{equation}
	This expression represents the curl-free condition for the rows of the distortion field~$\mathbf{A}$, which we will use at the discrete level to validate our numerical results.
	From $\mathbf{A}$ the stress tensor can be computed as  
	\begin{equation}
		\boldsymbol{\sigma} = \rho c_s^2 \mathbf{A}^T\mathbf{A} \, \operatorname{dev}(\mathbf{G}),
	\end{equation}
	with $c_s$ the shear sound speed, $\mathbf{G} = \mathbf{A}^T\mathbf{A}$ being the Riemannian metric tensor induced by the coordinate transformation from Lagrangian to Eulerian coordinates and  $\operatorname{dev}(\mathbf{G}) = \mathbf{G} - \frac{1}{3} \operatorname{tr}(\mathbf{G}) \mathbf{I}$ its trace-free part.

	\section{Numerical method}
	\label{sec.method}
	
	In this section, we present how the computational domain is discretized with the aid of staggered unstructured meshes and we also introduce the definition of our compatible discrete differential operators, together with their mathematical properties. Moreover, we explain how to discretize the aforementioned governing equations both in space and time. 
	We then present the main properties of the obtained structure-preserving numerical schemes. 
	These include the full compatibility with the \textit{low} Mach number limit, the exact 
	preservation of the divergence-free property of the magnetic field and of the curl-free 
	property of the inverse deformation gradient, and the divergence-free property of the 
	velocity for the incompressible Euler and Navier-Stokes equations. Furthermore, the 
	asymptotic-preserving property of the scheme is shown for the weakly compressible 
	Euler equations as $M \to 0$. Under suitable assumptions we also prove fully-discrete energy stability of the proposed scheme applied to the incompressible Euler, Navier-Stokes, MHD and GPR equations.

	\subsection{Staggered mesh, definition of compatible discrete differential operators and main scheme properties}
	
	We discretize the domain $\Omega$ with boundary $\partial \Omega$ in $d = 2$ space dimensions with a set of non-overlapping triangular control volumes $\omega_c$ of the Delaunay type, which will represent the \textit{primal} mesh. 
	Here, each triangle $\omega_c$ is identified by the index $c$ and its barycenter has coordinates $\mathbf{x}_{c}$. 
	Next, we denote by 
	$p$ a generic vertex of the triangulation with coordinates $\mathbf{x}_p$.
	In particular, the set of vertices belonging to cell $\omega_c$ is denoted by $\mathcal{P}(c)$, while the set of triangles sharing the node $p$ is denoted by $\mathcal{C}(p)$. The notation and the definition of the discrete differential operators follow \cite{Maire2007,Maire2020}.  
	
	In order to create the \textit{dual} mesh, we first need to define the sub-cell $\omega_{pc}$ by connecting the vertex with position $\mathbf{x}_p$, the cell barycenter of coordinates $\mathbf{x}_{c} = 1/(d+1)\sum_{p \in \mathcal{P}(c)} \mathbf{x}_p$ and the left and right midpoints of the edges impinging on node $p$, namely $p^{-1/2}$ and $p^{+1/2}$ with positions $\mathbf{x}_{p^{-1/2}}$ and $\mathbf{x}_{p^{+1/2}}$, respectively.
	The half lengths of these edges are given by $l_{pc}^- = | \mathbf{x}_p - \mathbf{x}_{p^{-1/2}}|$ and $l_{pc}^+ = | \mathbf{x}_p - \mathbf{x}_{p^{+1/2}}|$ while $\mathbf{n}_{pc}^-$ and $\mathbf{n}_{pc}^+$ denote the outward unit normal vectors of the corresponding half-edges, namely of the segments linking $\mathbf{x}_p$ to $\mathbf{x}_{p^{-1/2}}$ and to $\mathbf{x}_{p^{+1/2}}$, respectively. Then, the corner normal vector is computed as \begin{equation}
		l_{pc}\mathbf{n}_{pc} = l_{pc}^+\mathbf{n}_{pc}^+ + l_{pc}^-\mathbf{n}_{pc}^-.
		\label{eqn.corner}
	\end{equation}
	By construction, the corner vectors satisfy the fundamental geometrical identity
	\begin{equation}
		\sum_{p \in \mathcal{P}(c)}l_{pc}\mathbf{n}_{pc} = 0,
		\label{eqn.gauss}
	\end{equation}
	which is a consequence of the Gauss theorem.
	\begin{figure}[!bp]
		\centering
		\begin{tikzpicture}[scale=1.5]
			\begin{scope}
				\draw[->, thick, purple, >=stealth] (2.6,3) -- (2.8,3.7);
				\node[purple] at (4, 3.8) {$l_{pc}\mathbf{n}_{pc} = l_{pc}^+\mathbf{n}_{pc}^+ + l_{pc}^-\mathbf{n}_{pc}^-$};
				\draw[->, thin, black, >=stealth] (1.95,2.25) -- (1.7,2.5);
				\draw[->, thin, black, >=stealth] (2.845,2.25) -- (3.2,2.4);
				
				\draw[black, thin] (0,0) -- (3.5,0.3) -- (2.6,3) -- cycle;
				\draw[cyan, thin] (1.3,1.5) -- (1.9,1) --(3.05,1.65);
				\draw[cyan, fill=cyan, fill opacity = 0.3] (1.3,1.5) -- (1.9,1) -- (3.05,1.65) -- (2.6,3) -- cycle;
				\draw[cyan, thin] (1.3,1.5) -- (1.9,1) --(3.05,1.65);
				\draw[purple!50!blue, fill=purple!50!blue, fill opacity = 0.3] (1.3,1.5) -- (1.9,1) -- (3.05,1.65) -- (2.6,3) -- cycle;
				\filldraw[fill=red, draw=black] (0,0) circle(1.5pt);
				\filldraw[fill=red, draw=black] (3.5,0.3) circle(1.5pt);
				\filldraw[fill=red, draw=black] (2.6,3) circle(1.5pt);
				\filldraw[fill=orange, draw=black] (1.3, 1.5) circle (1pt);
				\filldraw[fill=orange, draw=black] (3.05,1.65) 	circle (1pt);
				\begin{scope}[shift={(1.9,1)}, rotate=45]
					\filldraw[black] (-1pt,-1pt) rectangle (1pt,1pt); 
				\end{scope}

				\node[purple!50!blue] at (2.2, 1.7) {$\mathbf{\omega}_{pc}$};
				\node[red] at (2.8,3) {$p$};
				\node[black] at (1.5, 0.5) {$\mathbf{\omega}_c$};
				\node[black] at (1.9,0.8) {$\mathbf{x}_{c}$};
				\node[black] at (3.45, 1.65) {$\mathbf{x}_{p ^{-1/2}}$};
				\node[black] at (1.0, 1.65) {$\mathbf{x}_{p ^{+1/2}}$};
				\node[black] at (1.7, 2.67)	{$\mathbf{n}_{pc}^+$};
				\node[black] at (3.075, 2.65)	{$\mathbf{n}_{pc}^-$};
				\node[black] at (3.9, 2.05) {$l_{pc}^- = | \mathbf{x}_p - \mathbf{x}_{p^{-1/2}}|$};
				\node[black] at (0.8,2.15) {$l_{pc}^+ = | \mathbf{x}_p - \mathbf{x}_{p^{+1/2}}|$};
				
			\end{scope}
			\begin{scope}[shift={(1.6,0)}]
				
				\draw[->, thick, purple, >=stealth] (6.95,2.7) -- (7.4,3);
				\draw[->, thick, purple, >=stealth] (5.7,2) -- (5.1, 1.65);
				
				\draw[black, thin] (4,0) -- (6,0) -- (5.7,2) -- cycle;
				\draw[black, thin] (6,0) -- (7.5,1) -- (5.7,2) -- cycle;
				\draw[black, thin] (7.5,1) -- (8,2.4) -- (5.7,2) -- cycle;
				\draw[black, thin] (8,2.4) -- (7.3,3.8) -- (5.7,2) -- cycle;
				\draw[black, thin] (7.3,3.8) -- (5.5,3.5) -- (5.7,2) -- cycle;
				\draw[black, thin] (5.5,3.5) -- (3.8,2.2) -- (5.7,2) -- cycle;
				\draw[black, thin] (3.8,2.2) -- (4,0) -- (5.7,2) -- cycle;
				\draw[cyan, fill = cyan, fill opacity = 0.3] (5.3,0.7) -- (5.85,1) -- (6.5,1) -- (6.6,1.5) -- (7,1.8) -- (6.95,2.2) -- (6.95,2.7) -- (6.5,2.9) -- (6.1,3.1) -- (5.6,2.75) -- (5,2.6) -- (4.75,2.1) -- (4.4,1.4) -- (4.85,1) -- cycle;
				\draw[purple!50!blue, fill=purple!50!blue, fill opacity=0.30] (6.95,2.23) -- (6.95,2.7) -- (6.5,2.9) -- (5.7,2) -- cycle;
				\begin{scope}[shift={(5.3,0.7)}, rotate=45]
					\filldraw[black] (-1pt,-1pt) rectangle (0.6pt,0.6pt); 
				\end{scope}
				\filldraw[fill=orange, draw=black] (5.85,1) circle (1pt);
				\begin{scope}[shift={(6.5,1)}, rotate=45]
					\filldraw[black] (-1pt,-1pt) rectangle (0.5pt,0.5pt); 
				\end{scope}
				\filldraw[fill=orange, draw=black] (6.6,1.5) circle (1pt);
				\begin{scope}[shift={(7,1.8)}, rotate=45]
					\filldraw[black] (-1pt,-1pt) rectangle (0.5pt,0.5pt); 
				\end{scope}
				\filldraw[fill=orange, draw=black] (6.95,2.23) circle (1pt);
				\begin{scope}[shift={(6.95,2.7)}, rotate=45]
					\filldraw[black] (-1pt,-1pt) rectangle (0.5pt,0.5pt); 
				\end{scope}
				\filldraw[fill=orange, draw=black] (6.5,2.9) circle (1pt);
				\begin{scope}[shift={(6.1,3.1)}, rotate=45]
					\filldraw[black] (-1pt,-1pt) rectangle (0.5pt,0.5pt); 
				\end{scope}
				\filldraw[fill=orange, draw=black] (5.6,2.75) circle (1pt);
				\begin{scope}[shift={(5,2.6)}, rotate=45]
					\filldraw[black] (-1pt,-1pt) rectangle (0.5pt,0.5pt); 
				\end{scope}
				\filldraw[fill=orange, draw=black] (4.75,2.1) circle (1pt);
				\begin{scope}[shift={(4.4,1.4)}, rotate=45]
					\filldraw[black] (-1pt,-1pt) rectangle (0.5pt,0.5pt); 
				\end{scope}
				\filldraw[fill=orange, draw=black] (4.85,1) circle (1pt);
				
				\filldraw[fill=red, draw=black] (5.7,2) circle (1.5pt);
				
				\node[cyan] at (6.3,1.25) {$\mathbf{\omega}_p$};
				\node[red] at (5.6, 1.7) {$p$};
				\node[purple!50!blue] at (6.6,2.5) {$\mathbf{\omega}_{pc}$};
				\node[black] at (7.2,2.55) {$\mathbf{x}_{c}$};
				\node[purple] at (7.20,3.15) {$-l_{pc}\mathbf{n}_{pc}$};
				\node[purple] at (4.9,1.55) {$l_{pc}\mathbf{n}_{pc}$};
				
			\end{scope}
			\begin{scope}[shift={(-4.5,-4.3)}]
				
				\draw[->, thick, purple, >=stealth] (9,2) -- (9.05, 3.1);
				\node[purple] at (9.3, 3.4) {$l_{pb}\mathbf{n}_{pb} = l_{pb}^+\mathbf{n}_{pb}^+ + l_{pb}^-\mathbf{n}_{pb}^-$};
				\draw[->, thin, black, >=stealth] (8.5,1.82) -- (8.4,2.3);
				\draw[->, thin, black, >=stealth] (9.5,1.95) -- (9.6, 2.4);
				
				\draw[black, thin] (8,0) -- (9,2) -- (7,1.3) -- cycle;
				\draw[black, thin] (8,0) -- (10,0.3) -- (9,2) -- cycle;
				\draw[black, thin] (10,0.3) --(11,1.8) -- (9,2) -- cycle;
				
				\draw[cyan, fill = cyan, fill opacity = 0.3] (8,1.65) -- (8,1.1) -- (8.5,1) -- (9, 0.77) -- (9.5,1.15) -- (10, 1.37) -- (10,1.9) -- (9,2)  --cycle;
				
				\filldraw[fill=orange, draw=black] (8,1.65) circle (1pt);
				\begin{scope}[shift={(8,1.1)}, rotate=45]
					\filldraw[black] (-1pt,-1pt) rectangle (0.5pt,0.5pt); 
				\end{scope}
				\filldraw[fill=orange, draw=black] (8.5,1) circle (1pt);
				\begin{scope}[shift={(9, 0.77)}, rotate=45]
					\filldraw[black] (-1pt,-1pt) rectangle (0.5pt,0.5pt); 
				\end{scope}
				\filldraw[fill=orange, draw=black] (9.5,1.15) circle (1pt);
				\begin{scope}[shift={(10, 1.37)}, rotate=45]
					\filldraw[black] (-1pt,-1pt) rectangle (0.5pt,0.5pt); 
				\end{scope}
				\filldraw[fill=orange, draw=black] (10,1.9) circle (1pt);
				
				\filldraw[fill=red, draw=black] (9,2) circle (1.5pt);
				
				\node[red] at (9,1.7) {$p$};
				\node[cyan] at (9, 1) {$\mathbf{\omega}_p$};
				\node[black] at (9.55,1.75) {$\mathbf{x}_{p ^{-1/2}}$};
				\node[black] at (8.4,1.5) {$\mathbf{x}_{p ^{+1/2}}$};
				\node[black] at (8.4, 2.5)	{$\mathbf{n}_{pb}^+$};
				\node[black] at (9.6, 2.6)	{$\mathbf{n}_{pb}^-$};
				\node[black] at (10.8,2.35) {$l_{pb}^- = | \mathbf{x}_p - \mathbf{x}_{p^{-1/2}}|$};
				\node[black] at (7.55,1.92) {$l_{pb}^+ = | \mathbf{x}_p - \mathbf{x}_{p^{+1/2}}|$};
				
			\end{scope}
		\end{tikzpicture}
		\caption{Top left: a control volume $\omega_c$, with its barycenter and edge midpoints. In violet, the sub-cell $\omega_{pc}$ which results by connecting the vertex with position $\mathbf{x}_p$, the cell barycenter of coordinates $\mathbf{x}_{c} = 1/(d+1)\sum_{p \in \mathcal{P}(c)} \mathbf{x}_p$ and the left and right midpoints of the edges impinging on node $p$, namely $p^{-1/2}$ and $p^{+1/2}$. Top right: the dual cell $\omega_p$ (cyan) is obtained by the union of all sub-cells $\omega_{pc}$ sharing a generic mesh node $p$. Bottom: the notation for the boundary dual cell located on $\partial \Omega$.}
		\label{domain discretization}
	\end{figure}
	Then, we introduce the dual cell $\omega_p$ given by the union of all sub-cells $\omega_{pc}$ sharing the generic mesh node $p$:
	\begin{equation}
		\omega_p = \bigcup_{c \in \mathcal{C}(p)} \omega_{pc}.
	\end{equation}
	The outward pointing corner vectors of $\omega_p$ are simply the cell corner vectors $l_{pc}\mathbf{n}_{pc}$ with opposite sign for all $c \in \mathcal{C}(p)$ and thus the following identity holds:
	\begin{equation}
		\sum_{c\in \mathcal{C}(p)} - l_{pc}\mathbf{n}_{pc} = 0.
	\end{equation} 
	From now on, we define $l_{cp}\mathbf{n}_{cp} = - l_{pc}\mathbf{n}_{pc} $.
	
	With the above definitions we now can define our discrete compatible nabla operators. Let $\phi_p$ denote a scalar quantity defined at the mesh nodes and $\phi_c$ a scalar quantity defined at the barycenters of the primal mesh. The corresponding discrete gradient operators are defined as
	\begin{equation}
		\nabla_c^p \phi_p = \frac{1}{|\omega_c|} \sum_{p \in \mathcal{P}(c)} l_{pc} \mathbf{n}_{pc} \phi_p, \hspace{1cm} \nabla_p^c \phi_c = \frac{1}{|\omega_p|}\sum_{c \in \mathcal{C}(p)} l_{cp} \mathbf{n}_{cp} \phi_c.
		\label{operators}
	\end{equation}
	The first operator $\nabla_c^p$ acts on quantities defined at the mesh nodes $p$ and provides a discrete derivative at the barycenter of the primal cell $\omega_c$, while the second operator $\nabla_p^c$ acts on quantities defined at the cell barycenters and provides a derivative at the mesh node $p$.
	Accordingly, in the notation $\nabla_c^p$ and $\nabla_p^c$, the superscript identifies the location where the input quantity is defined, whereas the subscript identifies the location where the discrete derivative is evaluated and thus the location of the output quantity.
	Depending on the nature of the discrete fields and on the scalar or cross product involved, the same construction is subsequently used to define discrete gradients, divergences and curls. 
	
	Finally, the time coordinate $t$ is defined in the interval $[0, t_f]$ and is approximated by a sequence of discrete points $t^n$ such that
	\begin{equation}
		t^{n+1} = t^n + \Delta t,
	\end{equation}
	where the time step $\Delta t$ is computed according to the CFL stability condition as
	\begin{equation}
		\Delta t \le \text{CFL} \min_c\frac{h(c)}{\lambda(c)},
		\label{eqn.cfl}
	\end{equation}
	where $h(c)$ is the incircle diameter of control volume $\omega_c$ and $\lambda$ involves only the eigenvalues related to the explicit subsystem, while the pressure is always treated implicitly.

	\paragraph{Boundary conditions}
	Let $p$ be a node lying on the boundary of the domain $\partial \Omega$, as depicted at the bottom of Figure \ref{domain discretization}. The corner vectors on the boundary are given by
	\begin{equation}
		l_{pb}\mathbf{n}_{pb} = l_{pb}^+\mathbf{n}_{pb}^+ + l_{pb}^-\mathbf{n}_{pb}^-,
	\end{equation}
	where the outward normals $l_{pb}^+\mathbf{n}_{pb}^+$ and $l_{pb}^-\mathbf{n}_{pb}^-$ refer to the boundary sides impinging on node $p$. These normals are essential to satisfy the discrete Gauss theorem also for a boundary node, and they are indeed linked to the corner vectors with the following relation:
	\begin{equation}
		\sum_{c \in \mathcal{C}(p)} l_{pc}\mathbf{n}_{pc} = l_{pb}\mathbf{n}_{pb}.
	\end{equation}
	
	\paragraph{Properties of our compatible discrete differential operators}
	For our discrete operators one can prove the two classical vector calculus identities \eqref{basis vector calculus identities} in their discrete form.
	Given a scalar field $\phi_p$ and a vector field $\mathbf{J}_p$, we prove that
	\begin{equation}
		\nabla_p^c\times \nabla_c^p \, \phi_p = 0, \hspace{1cm} \nabla_p^c \cdot (\nabla_c^p \times \mathbf{J}_p) = 0.
		\label{vector identity}
	\end{equation}
	In the following Lemma~\ref{identity}, we will prove the first discrete vector calculus identity, namely that the curl of the gradient of a scalar field is zero.
	We also recall that this Lemma~\ref{identity} has already been proved in a slightly different context in~\cite{abgrall2025simple,boscheri2025structure,sidilkover2025spurious}.
	\begin{lemma}
		In two space dimensions ($d = 2$) the discrete gradient operator applied to a scalar field $\phi_p$ defined at the mesh nodes $p$
		\begin{equation}
			\nabla_c^p \, \phi_p = \frac{1}{|\omega_c|}\sum_{p\in \mathcal{P}(c)}l_{pc}\mathbf{n}_{pc}\phi_p, 
		\end{equation}
		and its dual discrete curl operator applied to a vector field $\mathbf{J}_c$ defined at the cell barycenters
		\begin{equation}
			\nabla_p^c \times \mathbf{J}_c = \frac{1}{|\omega_p|}\sum_{c \in \mathcal{C}(p)}l_{pc}\mathbf{n}_{cp} \times \mathbf{J}_c,
		\end{equation}
		satisfy the first discrete vector calculus identity (\ref{vector identity}).
		\label{identity}
	\end{lemma}
	
	\begin{proof}
		Let us consider a generic control volume $\omega_c$ with indices of the three nodes $p_1, p_2$ and $p_3$ and Eulerian node coordinates $\mathbf{x}_{p1}, \mathbf{x}_{p2}$ and $\mathbf{x}_{p3}$, respectively. We assume that the nodes are ordered counter-clockwise. The indices of the edges of $\omega_c$ are denoted by $e_1, e_2$ and $e_3$, respectively, with edge $e_1$ composed of nodes $p_1$ and $p_2$, edge $e_2$ composed of nodes $p_2$ and $p_3$ and edge $e_3$ composed of nodes $p_3$ and $p_1$. 
		Since the problem is two-dimensional, all planar vectors are embedded in $\mathbb{R}^3$ by setting their third component equal to zero.
		Assuming a linear distribution of the scalar field $\phi_p$ inside the control volume $\omega_c$ yields the following averages at the edge midpoints:
		\begin{equation}
			\phi_{e_1} = \frac{1}{2}\left(\phi_{p1} + \phi_{p2}\right), \, \phi_{e_2} = \frac{1}{2}\left(\phi_{p2} + \phi_{p3}\right), \, \phi_{e_3} = \frac{1}{2}\left(\phi_{p3} + \phi_{p1}\right).
		\end{equation}
		The three corner normals are given by
		\begin{equation}
			\begin{aligned}
				l_{p1c}\mathbf{n}_{p1c} &= \frac{1}{2}\left((y_{p2} - y_{p3}), (x_{p3} - x_{p2}), 0\right)^T,\\
				l_{p2c}\mathbf{n}_{p2c} &= \frac{1}{2}\left((y_{p3} - y_{p1}), (x_{p1} - x_{p3}), 0\right)^T,\\
				l_{p3c}\mathbf{n}_{p3c} &= \frac{1}{2}\left((y_{p1} - y_{p2}), (x_{p2} - x_{p1}), 0\right)^T,\\
			\end{aligned}
		\end{equation}
		while the area of the element $\omega_c$ is 
		\begin{equation}
			|\omega_c| = \frac{1}{2}\left(x_{p1}y_{p2} - x_{p1}y_{p3} - x_{p2}y_{p1} + x_{p2}y_{p3} + x_{p3}y_{p1} - x_{p3}y_{p2}  \right).
		\end{equation}
		The discrete gradient $\nabla_c^p \,\phi_p = \frac{1}{|\omega_c|}\sum_pl_{pc}\mathbf{n}_{pc}\phi_p$ of the scalar field $\phi_p$ in the cell $c$ then reads
		\begin{equation}
			\nabla_c^p \, \phi_p = \frac{1}{2|\omega_c|} 
			\begin{pmatrix}
				(y_{p2} - y_{p3})\phi_{p1} + (y_{p3} - y_{p1})\phi_{p2} + (y_{p1} - y_{p2})\phi_{p3} \\
				(x_{p3} - x_{p2})\phi_{p1} + (x_{p1} - x_{p3})\phi_{p2} + (x_{p2} - x_{p1})\phi_{p3} \\
				0
			\end{pmatrix}.
		\end{equation}
		We now show that this gradient is identical to the discrete gradient obtained by using continuous $P_1$ Lagrange finite elements inside the control volume $\omega_c$. The associated basis functions in the universal reference element $T_0 = \{\xi_1, \xi_2 : 0 \le \xi_1 \le 1, \, 0\le \xi_2 \le 1-\xi_1 \}$ with nodes (0,0), (1,0) and (0,1) and reference coordinates $\boldsymbol{\xi} = (\xi_1, \xi_2)$ read
		\begin{equation}
			\psi_{p1} = \psi_{p1}(\boldsymbol{\xi}) = 1 - \xi_1 - \xi_2, \quad \psi_{p2} = \psi_{p2}(\boldsymbol{\xi}) = \xi_1, \quad \psi_{p3} = \psi_{p3}(\boldsymbol{\xi}) = \xi_2.
		\end{equation}
		The Jacobian of the mapping 
		\begin{equation}
			\mathbf{x} = \mathbf{x}(\boldsymbol{\xi}) = \psi_{p1}(\boldsymbol{\xi})\mathbf{x}_{p1} + \psi_{p2}(\boldsymbol{\xi})\mathbf{x}_{p2} + \psi_{p3}(\boldsymbol{\xi})\mathbf{x}_{p3} = \sum_{p\in \mathcal{P}(c)} \psi_p \mathbf{x}_p,
		\end{equation}
		reads
		\begin{equation}
			\mathbf{M}_c(\mathbf{x}) = \frac{\partial \mathbf{x}}{\partial \boldsymbol{\xi}} = 
			\begin{pmatrix}
				x_{p2} - x_{p1} \quad x_{p3} - x_{p1}\\
				y_{p2} - y_{p1} \quad y_{p3} - y_{p1}
			\end{pmatrix}
			= \sum _{p \in \mathcal{P}(c)}\frac{\partial \psi_p}{\partial \boldsymbol{\xi}}\mathbf{x}_p,
		\end{equation}
		with $|\mathbf{M}_c| = 2|\omega_c|$. The discrete gradient of a discrete scalar quantity $\phi_h = \phi_{p1}\psi_{p1} + \phi_{p2}\psi_{p2} + \phi_{p3}\psi_{p3}$ is then given by
		\begin{equation}
			\begin{split}
				\nabla_c^p\phi_p &= \mathbf{M}_c^{-T}\frac{\partial \phi_h}{\partial \boldsymbol{\xi}} = \mathbf{M}_c^{-T} \sum _{p\in \mathcal{P}(c)} \frac{\partial \psi_p}{\partial \boldsymbol{\xi}}\phi_p\\
				&= \frac{1}{2|\omega_c|}
				\begin{pmatrix}
					(y_{p2} - y_{p3})\phi_{p1} + (y_{p3} - y_{p1})\phi_{p2} + (y_{p1} - y_{p2})\phi_{p3} \\
					(x_{p3} - x_{p2})\phi_{p1} + (x_{p1} - x_{p3})\phi_{p2} + (x_{p2} - x_{p1})\phi_{p3} \\
					0
				\end{pmatrix}
				= \nabla_c^p\phi_p.
			\end{split}
		\end{equation}
		In the following, we denote a generic node of element $\omega_c$ by $p$, as usual, and the corresponding two nodes on the opposite edge are denoted in counter-clockwise order by $p^-$ and $p^+$, respectively. Furthermore, we denote by $e^-_{pc}$ the edge connecting $p$ with $p^-$, and by $e^+_{pc}$ the edge connecting $p$ with $p^+$. We now compute the single contribution of element $\omega_c$ and node $p$ to the curl on the dual mesh $|\omega_p|\nabla_p^c \times \nabla_c^p\phi_p = -\sum_c l_{pc} \mathbf{n}_{pc} \times \nabla_c^p \phi_p$, i.e. the contribution $-l_{pc}\mathbf{n}_{pc}\times \nabla_c^p\phi_p$, which after straightforward calculation reduces to
		\begin{equation}
			-l_{pc} \mathbf{n}_{pc} \times \nabla_c^p \phi_p = \frac{1}{2}\left(\phi_{p^+} - \phi_{p^-} \right) = \phi_{e^+_{pc}} - \phi_{e^-_{pc}}, 
		\end{equation} 
		i.e. the difference of values at the respective edge midpoints. Consider now two adjacent elements $\omega_c$ and $\omega_d$ which share a common edge $e$ composed of the nodes $p$ and $q$. We assume that the elements are ordered counter-clockwise, i.e. the previous element is $\omega_c$ and the subsequent element is $\omega_d$. The value at the edge midpoint given by the previous element $\omega_c$ is denoted by $\phi^-_e = \phi_{e^+_{pc}}$ and the value at the edge midpoint given by the subsequent element $\omega_d$ is denoted by $\phi^+_e = \phi_{e^-_{pc}}$. Since the values of the scalar field $\phi$ at the edge midpoint are the \textit{same} for two adjacent elements that share a common edge $e$, i.e. $\phi^+ = \phi^- = \frac{1}{2}(\phi_p + \phi_q)$, summing up over all elements $\omega_c$ around the node $p$ yields zero, due to the telescopic sum property:
		\begin{equation}
			|\omega_p| \nabla_p^c \times \nabla_c^p \, \phi_p = - \sum_{c \in \mathcal{C}(p)}l_{pc}\mathbf{n}_{pc} \times \nabla_c^p\, \phi_p = \sum_{e \in \mathcal{S}(p)} \phi_e^+ - \phi_e^- = 0.
		\end{equation}
		Here, $\mathcal{S}(p)$ is the set of edges connected to node $p$. The above identity completes the proof that the discrete curl on the dual mesh applied to the discrete gradient on the primal mesh is zero. 
	\end{proof}
	In a similar manner it is possible to prove that the discrete dual divergence of the discrete primal curl is zero, see \cite{abgrall2025simple}. 
	
	In addition, we prove that, for interior nodes, both the discrete gradient operators in~\eqref{operators} reproduce globally linear fields exactly on general unstructured meshes.
	\begin{lemma}
		Let $\phi(\mathbf{x}) = a + \mathbf{b} \cdot \mathbf{x}$ be a globally linear scalar function where $a$ and $\mathbf{b}$ are constant, let $\phi_p = \phi(\mathbf{x}_p)$ denote its values at the mesh nodes and $\phi_c = \phi(\mathbf{x}_c)$ its values at the cell barycenters. Then
		\begin{eqnarray}
			\nabla_c^p \, \phi_p & = & \mathbf{b} \qquad \textnormal{in every cell } \omega_c, \label{eqn.exact.primal} \\
			\nabla_p^c \, \phi_c & = & \mathbf{b} \qquad \textnormal{at every interior node } p. \label{eqn.exact.dual}
		\end{eqnarray}
		\label{lemma.exact}
	\end{lemma}
	\begin{proof}
		In the proof of Lemma~\ref{identity} the primal gradient has been identified with the gradient of the continuous $P_1$ interpolant $\phi_h = \sum_{p \in \mathcal{P}(c)} \phi_p \psi_p$ inside $\omega_c$, where the $\psi_p$ are the nodal basis functions of $\omega_c$. Since both expressions are linear in the nodal values $\phi_p$, the comparison with the definition \eqref{operators} yields the identity
		\begin{equation}
			l_{pc} \mathbf{n}_{pc} = |\omega_c| \, \nabla \psi_p |_{\omega_c},
			\label{eqn.corner.basis}
		\end{equation}
		valid for every node $p \in \mathcal{P}(c)$ of every cell $\omega_c$. Since the $P_1$ interpolant of a linear function is the function itself, $\phi_h = \phi$ and \eqref{eqn.exact.primal} follows immediately.
		
		Now let $p$ be an interior node and let $\psi_p$ denote its global hat function, which is continuous and supported in the union of the cells $\omega_c$ with $c \in \mathcal{C}(p)$ and vanishes on the boundary of this union. Using $l_{cp} \mathbf{n}_{cp} = - l_{pc} \mathbf{n}_{pc}$ and \eqref{eqn.corner.basis}, the definition of the dual gradient becomes
		\begin{equation}
			|\omega_p| \, \nabla_p^c \, \phi_c =  \sum_{c \in \mathcal{C}(p)} l_{cp} \mathbf{n}_{cp} \, \phi(\mathbf{x}_c) \; = \; - \sum_{c \in \mathcal{C}(p)} |\omega_c| \, \phi(\mathbf{x}_c) \, \nabla \psi_p |_{\omega_c} \; = \; - \sum_{c \in \mathcal{C}(p)} \int_{\omega_c} \phi \, \nabla \psi_p \, d\mathbf{x},
			\label{eqn.exact.dual.1}
		\end{equation}
		where in the last step we have used that $\nabla \psi_p$ is constant inside $\omega_c$ and that the barycenter rule is exact for linear functions, $|\omega_c| \, \phi(\mathbf{x}_c) = \int_{\omega_c} \phi \, d\mathbf{x}$. Integration by parts on the union of the cells around $p$ produces no boundary term, because $\psi_p$ vanishes there, hence
		\begin{equation}
			|\omega_p| \, \nabla_p^c \, \phi_c = \int_{\Omega} \psi_p \, \nabla \phi \, d\mathbf{x} \; = \; \mathbf{b} \int_{\Omega} \psi_p \, d\mathbf{x} \; = 
			\; \mathbf{b} \sum_{c \in \mathcal{C}(p)} \frac{|\omega_c|}{3} \; = 
			\; \mathbf{b} \sum_{c \in \mathcal{C}(p)} \frac{|\omega_{pc}|}{|\omega_c|} \, |\omega_{c}| \;  = \; \mathbf{b} \, |\omega_p|,
			\label{eqn.exact.dual.2}
		\end{equation}
		since the integral of the hat function over a triangle is one third of its area and since each sub-cell $\omega_{pc}$ also has the area $|\omega_c|/3$, so that $|\omega_p| = \sum_{c \in \mathcal{C}(p)} |\omega_{pc}|$ coincides with the integral of $\psi_p$ over $\Omega$. This proves~\eqref{eqn.exact.dual}.
	\end{proof}
	The computation \eqref{eqn.exact.dual.1}--\eqref{eqn.exact.dual.2} applies also to the discrete divergence and curl operators on the dual mesh, which are built from the same corner vectors with the scalar and the cross product in place of the ordinary product: for a globally linear vector field $\mathbf{J}$ one obtains $\nabla_p^c \cdot \mathbf{J}_c = \nabla \cdot \mathbf{J}$ and $\nabla_p^c \times \mathbf{J}_c = \nabla \times \mathbf{J}$ at every interior node $p$.

	\subsection{Semi-implicit discretization of the equations}
	\label{sec.semiimplicit}
	
	To discretize the governing partial differential equations, we first need to clarify where the relevant variables are defined. 
	From now on, we will use the subscript $c$ for variables defined inside a triangle $\omega_c$ and the subscript $p$ for variables defined at mesh nodes. 
	
	Sometimes a quantity discretized at the mesh nodes needs to be used within a discretization at the barycenters, and vice versa. For example, to define the velocity discretized at the mesh nodes $\mathbf{u}_p$ starting from values of $\mathbf{u}_c$ given at the barycenters of the control volumes $\omega_c$, a weighted average approach is used.
	Specifically, the velocity at the node $p$ is computed as: 
	\begin{equation}
		\mathbf{u}_p = \frac{1}{|\omega_p|}\sum_{c \in \mathcal{C}(p)}|\omega_{pc}| \, \mathbf{u}_c,
		\label{eqn.avg.node}
	\end{equation}
	where $|\omega_p|$ is the total area of the \textit{star} polygon associated with the node $p$, $|\omega_{pc}|$ is the area of the subvolume $\omega_{pc} = \omega _c \cap \omega_p$ and $\mathbf{u}_c$ is the velocity at the barycenter $c$. In this way, the velocity at the mesh node is obtained by averaging the velocity at the barycenters, weighted by the areas of the subvolumes surrounding the node~$p$. This weighted averaging procedure is exact for constant fields but is in general not exact for linear fields on arbitrary unstructured meshes. In the second-order extension of our method, the cell-centered quantities are therefore first reconstructed at the node location via piecewise linear polynomials before applying the weighted average, thereby recovering linear exactness by construction, see also \cite{Maire2011}.
	
	Vice versa, to define the density $\rho_c$ at the barycenters of the triangles $\omega_c$, starting from values of $\rho_p$ defined at the mesh nodes $p$, a similar weighted average approach can be employed. In this case, the density at the barycenter of  the control volume $\omega_c$ is computed by averaging the densities at the vertices belonging to triangle $\omega_c$, with each vertex contributing equally. Specifically:
	\begin{equation}
		\rho_c = \frac{1}{|\omega_c|}\sum_{p\in \mathcal{P}(c)} |\omega_{pc}| \, \rho_p = \frac{1}{3}\sum_{p\in \mathcal{P}(c)} \rho_p,
		\label{eqn.avg.cell}
	\end{equation}
	where $\mathcal{P}(c)$ is the set of vertices belonging to the control volume $\omega_c$.
	According to the adopted vertex-staggered discretization, the discrete momentum $\mathbf{m}_c$, the velocity $\mathbf{u}_c$, the magnetic field $\mathbf{B}_c$  and the distortion field $\mathbf{A}_c$ are defined at the barycenters of the primal cell $\omega_c$, while the pressure $p_p$, the density $\rho_p$, the electric field $\mathbf{E}_p$ and the stress tensor $\boldsymbol{\sigma}_p$ are defined at the mesh nodes $p$.
	\subsubsection{Incompressible Euler equations}
	Using the notation and the operators introduced above, and assuming constant density $\rho_c = \rho = \textnormal{const}$, we first discretize the incompressible Euler equations in a semi-implicit manner as they represent the essential core of the numerical discretization of all the other equations. We proceed as in classical projection schemes \cite{harlow1965numerical,patankar2018numerical,Casulli1999,Guer06}, i.e. we first compute an intermediate momentum and velocity field via an \textit{explicit} finite volume method as follows: 
	\begin{equation}
		\mathbf{m}_c^* = \mathbf{m}_c^n - \frac{\Delta t}{|\omega_c|}\sum_{a \in \mathcal{N}_c}|\partial \omega_{ac}|\mathbf{f}_{ac}, 
		\qquad \mathbf{u}_c^* = \frac{\mathbf{m}_c^*}{\rho},
		\label{star equations}
	\end{equation}
	where $\mathcal{N}_c$ is the set of neighboring control volumes $\omega_{ac}$ that share an edge with element $\omega_c$ and $\mathbf{f}_{ac}$ is the classical edge-based Rusanov flux applied to the nonlinear convective terms and computed as 
	\begin{equation}
		\mathbf{f}_{ac} = \frac{1}{2}  \left( \rho_a \mathbf{u}_a \otimes \mathbf{u}_a + \rho_c \mathbf{u}_c \otimes \mathbf{u}_c \right) \cdot \mathbf{n}_{ac} - \frac{1}{2}  s_{ac}   (\rho_a \mathbf{u}_a - \rho_c \mathbf{u}_c).
		\label{eqn.rusanov} 
	\end{equation}
	Here, $s_{ac}>0$ is a suitable numerical dissipation that is needed to obtain energy stability and the consistent part of the numerical flux contains the flux tensor related to the nonlinear convective terms of the Euler equations.  
	While convective terms are treated explicitly, the pressure term is handled implicitly. With the aid of the intermediate momentum $\mathbf{m}_c^*$, we can rewrite the discrete momentum equation as
	\begin{equation}
		\mathbf{m}_c^{n+1} = \mathbf{m}_c^* -\Delta t \, \nabla_c^p p_p^{n+1},
		\label{velocity update with u*}
	\end{equation}
	which is combined with the discrete divergence-free condition of the velocity field 
	\begin{equation}
		\nabla_p^c \cdot \mathbf{u}_c^{n+1} = 0.
		\label{eqn.divvv} 
	\end{equation} 
	Since $\rho = \textnormal{const}$, we have 
	\begin{equation}
		\nabla_p^c \cdot \mathbf{m}_c^{n+1} = 0.
		\label{eqn.divvm} 
	\end{equation} 
	Inserting \eqref{velocity update with u*} into \eqref{eqn.divvm} leads to the discrete pressure Poisson equation 
	\begin{equation}
		|\omega_p| \, \nabla_p^c \cdot \nabla_c^q \, p_q^{n+1} = \frac{|\omega_p|}{\Delta t} \,  \nabla_p^c \cdot \mathbf{m}_c^*.
		\label{pressure correction equation}
	\end{equation}
	
	The discrete Laplacian on the left-hand side of \eqref{pressure correction equation} coincides, up to the sign, with the stiffness matrix of the $P_1$ continuous Galerkin discretization of the Laplacian, since the corner vectors are the gradients of the hat functions multiplied by the cell areas, see the proof of Lemma~\ref{lemma.exact}.
	Once the new pressure is known, the momentum can be updated via \eqref{velocity update with u*}. 
	This results, up to the sign, in a symmetric and positive definite system (positive semi-definite on periodic domains, where the pressure is determined up to an additive constant that does not affect the pressure gradient) that is solved for the pressure at the next time step $p_p^{n+1}$ using an iterative method, such as the matrix-free Conjugate Gradient (CG) method \cite{cgmethod}. 
	
	\subsubsection{Weakly compressible Euler and Navier-Stokes equations}
	\label{sec.wce}
	
	In the case of the weakly compressible Euler and Navier-Stokes equations, the discrete momentum equation reads as follows: 
	\begin{equation}
		\mathbf{m}_c^* = \mathbf{m}_c^n - \frac{\Delta t}{|\omega_c|}\sum_{a \in \mathcal{N}_c}|\partial \omega_{ac}|\mathbf{f}_{ac}
		- \Delta t \, \nabla_c^p\cdot \boldsymbol{\sigma}_p, 
		\qquad \mathbf{u}_c^* = \frac{\mathbf{m}_c^*}{\rho_c^n}.
		\label{eqn.mom.ns}
	\end{equation}
	In the case of an explicit discretization of the viscous stress tensor we simply have 
	\begin{equation}
		\boldsymbol{\sigma}_p = -\mu \left( (\nabla_p^c \mathbf{u}_c^n) + (\nabla_p^c \mathbf{u}_c^n)^T - \frac{2}{3} (\nabla_p^c \cdot \mathbf{u}_c^n) \mathbf{I}  \right),  
	\end{equation}
	while an implicit discretization of the viscous stress tensor reads 
	\begin{equation}
		\boldsymbol{\sigma}_p = -\mu \left( (\nabla_p^c \mathbf{u}_c^*) + (\nabla_p^c \mathbf{u}_c^*)^T - \frac{2}{3} (\nabla_p^c \cdot \mathbf{u}_c^*) \mathbf{I}  \right),  
		\label{eqn.sigma.implicit}
	\end{equation}
	and requires the solution of a symmetric positive definite linear system for the star velocity $\mathbf{u}_c^*$. 
	The discrete momentum equation can still be written in terms of $ \mathbf{m}_c^*$ as 
	\begin{equation}
		\mathbf{m}_c^{n+1} = \mathbf{m}_c^* -\Delta t\nabla_c^p p_p^{n+1}.
		\label{eqn.mom}
	\end{equation}
	In the weakly compressible case, the discrete mass conservation equation reads
	\begin{equation}
		\rho(p_p^{n+1}) = \rho_p^{n} - \Delta t \nabla_p^c \cdot \mathbf{m}_c^{n+1}.
		\label{eqn.conti} 
	\end{equation} 
	Inserting the discrete momentum equation \eqref{eqn.mom} into the discrete mass conservation equation \eqref{eqn.conti} yields the following mildly nonlinear discrete wave equation for the pressure, which can be solved with the nested Newton method of Casulli and Zanolli, see \cite{CasulliZanolli2010,CasulliZanolli2012}:
	\begin{equation}
		\rho(p_p^{n+1}) - \Delta t^2 \, \nabla_p^c \cdot \nabla_c^q \, p_q^{n+1} 
		= \rho(p^n_p) - \Delta t \, \nabla_p^c \cdot \mathbf{m}_c^*.
		\label{eqn.pressure.nonlin}
	\end{equation}
	
	Making use of the isothermal equation of state $p = \rho c_0^2$ then leads to the following linear system 
	\begin{equation}
		\frac{p^{n+1}_p}{c_0^2} - \Delta t^2 \nabla_p^c \cdot \nabla_c^q \, p_q^{n+1} 
		= \frac{p^n_p}{c_0^2} - \Delta t \nabla_p^c \cdot \mathbf{m}_c^*.
		\label{pressure correction for the compressible case}
	\end{equation}
	
	Once the new pressure is known, the momentum is updated via \eqref{eqn.mom} and the density via \eqref{eqn.conti}. In practice, following the MPV approach \cite{klein2001asymptotic,munz2003extension,park2005multiple}, the pressure system \eqref{pressure correction for the compressible case} is solved for the deviation $p'_p = p_p~-~p_0$ from a constant background pressure $p_0$, which is subtracted from the initial pressure, and the same deviation is used in the discrete pressure gradient of the momentum update \eqref{eqn.mom}, so that the constant $p_0$, which is of the order $M^{-2}$, never enters a difference of pressures; the full pressure is only needed in the equation of state. This is essential for very small Mach numbers, see Section~\ref{sec.cylinder}. It is easy to see that for $c_0 \to \infty$, i.e. in the low Mach number limit $M \to 0$, the discrete pressure equation \eqref{pressure correction for the compressible case} reduces to the incompressible pressure Poisson equation \eqref{pressure correction equation}. Furthermore, since for $M \to 0$ we have $\nabla_p^c \cdot \mathbf{m}_c^{n+1} \to 0$, an immediate result is that $\rho_p^{n+1} \to \rho_p^n$, i.e. density will remain constant if it was initially constant. Hence, our new method is by construction  asymptotic-preserving (AP) in the low Mach number limit. 
	
	\subsubsection{Incompressible MHD equations}
	\label{sec.mhd}
	The numerical scheme for the incompressible MHD equations is a straightforward extension of the method presented initially for the incompressible Euler equations. For incompressible MHD, the compatible discretization of the magnetic field reads
	\begin{equation}
		\mathbf{B}_c^{n+1} = \mathbf{B}_c^{n} - \Delta t \, \nabla_c^p \times \mathbf{E}_p^n, 
		\qquad \textnormal{ with }  \qquad \mathbf{E}_p^n = - \frac{1}{|\omega_p|} \sum_{c \in \mathcal{C}(p)} |\omega_{pc}| \, \mathbf{u}^n_c \times \mathbf{B}^n_c + \eta_p \, \mathbf{j}_p^n,
		\label{eqn.mhd.induction}
	\end{equation}
	where the first term is the nodal average \eqref{eqn.avg.node} of the cell values $- \mathbf{u}_c^n \times \mathbf{B}_c^n$, $\eta_p \geq 0$ is a numerical resistivity at the node $p$ and $\mathbf{j}_p^n = \nabla_p^c \times \mathbf{B}_c^n$ is the discrete current density at the node $p$. The discretization \eqref{eqn.mhd.induction} preserves the divergence-free condition of the magnetic field $\nabla_p^c \cdot \mathbf{B}_c^n = 0$ for all times if it was satisfied at the initial time $t=0$, independently of the choice of the nodal electric field and hence also in the presence of the resistive term. Using the continuous identity between the divergence of the Maxwell stress tensor and the Lorentz force, which holds for $\nabla \cdot \mathbf{B} = 0$, 
	\begin{equation}
		\nabla \cdot \left( -\mathbf{B} \otimes \mathbf{B} + \frac{1}{2} \mathbf{B}^2 \, \mathbf{I} \right) 
		= 
		\mathbf{B} \times \left( \nabla \times \mathbf{B} \right), 
	\end{equation} 
	the complete scheme for the incompressible MHD equations reads 
	\begin{equation}
		\mathbf{m}_c^* = \mathbf{m}_c^n - \frac{\Delta t}{|\omega_c|}\sum_{a \in \mathcal{N}_c}|\partial \omega_{ac}|\mathbf{f}_{ac} - 
		\frac{\Delta t}{|\omega_c|} \sum_{p \in \mathcal{P}(c)} |\omega_{pc}| \, \mathbf{B}_c^n \times \mathbf{j}_p^n,  
		\label{eqn.mhd.mom}
	\end{equation}	
	together with the induction equation \eqref{eqn.mhd.induction}, the pressure correction \eqref{velocity update with u*} and the divergence-free condition 
	\begin{equation}
		\nabla_p^c \cdot \mathbf{u}_c^{n+1} = 0.
	\end{equation} 
	
	Note that $|\omega_{pc}| / |\omega_c| = 1/3$ for every node $p \in \mathcal{P}(c)$, see the proof of Lemma~\ref{lemma.exact}.
	The discrete current density $\mathbf{j}_p^n$, i.e. the discrete curl of the magnetic field, is therefore computed at the nodes, where the electric field is defined, and transferred to the barycenters with the same average \eqref{eqn.avg.cell} that is used for the density, while the Lorentz force itself is evaluated with the cell values $\mathbf{B}_c^n$. The momentum equation is discretized in this particular manner in order to ensure energy stability, see Lemma~\ref{lemma.energy.mhd} below. It is again inserted into the discrete divergence-free condition to obtain the discrete pressure equation, while the magnetic field satisfies $\nabla_p^c \cdot \mathbf{B}_c = 0$ by construction thanks to the use of the compatible nabla operators, see \cite{abgrall2025simple} for details.  
	
	In order to ensure an exactly divergence-free magnetic field for all times, at the initial time $t=0$ the magnetic field must be computed 
	as the discrete curl of a discrete vector potential $\mathbf{A}_p$ such that $\mathbf{B}_c = \nabla_c^p \times \mathbf{A}_p$. 
	
	\subsubsection{Incompressible GPR model}
	\label{sec.gpr}
	
	The discretization of the incompressible GPR model is also a straightforward extension of the scheme for the incompressible Euler equations. 
	Considering the variable definitions given in Section~\ref{sec.semiimplicit},
	the compatible discretization for the distortion field equation reads 
	\begin{equation}
		\mathbf{A}_c^{n+1} = \mathbf{A}_c^n - \Delta t \, \nabla_c^p \left( \mathbf{A}^n_p \mathbf{u}^n_p - \zeta_p \, \nabla_p^a \cdot \boldsymbol{\alpha}_a^n \right),
		\label{eqn.gpr.A}
	\end{equation}
	
	where $\mathbf{A}_p^n$ and $\mathbf{u}_p^n$ are the nodal averages \eqref{eqn.avg.node} of $\mathbf{A}_c^n$ and $\mathbf{u}_c^n$.
	Moreover, $\zeta_p \geq 0$ is a numerical dissipation coefficient at the node $p$ and $\boldsymbol{\alpha}_a^n$ denotes the derivative, in the cell $\omega_a$, of the elastic energy density $\mathcal{E}_s(\mathbf{A})$ w.r.t. $\mathbf{A}$ that is  
	\begin{equation}
		\boldsymbol{\alpha} = \frac{\partial \mathcal{E}_s}{\partial \mathbf{A}} = \rho c_s^2 \mathbf{A} \operatorname{dev}(\mathbf{G})
		\quad \text{with} \quad 
		\mathcal{E}_s(\mathbf{A}) =  
		\frac{1}{4} \rho c_s^2 \,  \operatorname{tr}\left( \operatorname{dev}(\mathbf{G})^T \operatorname{dev}(\mathbf{G}) \right) = \frac{1}{4} \rho c_s^2 \,  \operatorname{dev}(\mathbf{G}) : \operatorname{dev}(\mathbf{G}),
		\label{eqn.gpr.tau}
	\end{equation}
	so that the stress tensor is $\boldsymbol{\sigma} = \mathbf{A}^T \boldsymbol{\alpha}$. 
	In the above formula note that the divergence of the matrix $\boldsymbol{\alpha}_a$ is taken row by row and that we have introduced the notation $\mathbf{M}:\mathbf{N} = \operatorname{tr}(\mathbf{M}^T \mathbf{N})$ for matrices $\mathbf{M}$ and $\mathbf{N}$ of the same size.

	Since the update \eqref{eqn.gpr.A} is the discrete gradient of a nodal field, the curl-free condition of the distortion field $\nabla_p^c \times \mathbf{A}_c = 0$ holds for all times, as long as it holds at $t=0$, independently of this nodal field and hence also in the presence of the dissipative term. For a curl-free distortion field and constant density, the divergence of the stress tensor satisfies 
	\begin{equation}
		\nabla \cdot \boldsymbol{\sigma} = \mathbf{A}^T \left( \nabla \cdot \boldsymbol{\alpha} \right) + \nabla \mathcal{E}_s, 
		\label{eqn.gpr.identity}
	\end{equation}
	where the gradient of the elastic energy density is absorbed into the pressure, which in the following denotes the modified pressure $p + \mathcal{E}_s$, so that the complete scheme for the incompressible GPR model reads 
	\begin{equation}
		\mathbf{m}_c^* = \mathbf{m}_c^n - \frac{\Delta t}{|\omega_c|}\sum_{a \in \mathcal{N}_c}|\partial \omega_{ac}|\mathbf{f}_{ac} - 
		\frac{\Delta t}{|\omega_c|} \sum_{p \in \mathcal{P}(c)} |\omega_{pc}| \left( \mathbf{A}_p^n \right)^T \left( \nabla_p^a \cdot \boldsymbol{\alpha}_a^n \right), 
		\label{eqn.gpr.mom}
	\end{equation}
	together with the distortion equation \eqref{eqn.gpr.A}, the pressure correction \eqref{velocity update with u*} and the divergence-free condition 
	\begin{equation}
		\nabla_p^c \cdot \mathbf{u}_c^{n+1} = 0. 
	\end{equation} 
	
	The elastic force is therefore computed at the nodes, where the divergence of $\boldsymbol{\alpha}$ is defined, and it is transferred to the barycenters with the same average \eqref{eqn.avg.cell} that is used for the density. The momentum equation is discretized in this particular manner in order to ensure energy stability, see Lemma~\ref{lemma.energy.gpr}. Insertion of the discrete momentum equation into the discrete divergence-free condition of the velocity leads again to the pressure equation. In order to ensure an exactly curl-free distortion field for all times, one must compute the initial distortion field as $\mathbf{A}_c = \nabla_c^p \mathbf{X}_p$ with $\mathbf{X}_p$ a discrete vector potential. The velocity $\mathbf{u}_c$ and the distortion field $\mathbf{A}_c$ are stored at the cell barycenters, and whenever nodal values are required, $\mathbf{u}_p$ and $\mathbf{A}_p$ are obtained from the corresponding cell-centered quantities through the area-weighted average \eqref{eqn.avg.node}.

	\subsection{Fully-discrete energy stability of the incompressible Euler and Navier-Stokes equations}
	For the pressure subsystem the chosen compatible semi-implicit discretization leads to a kinetic energy stable scheme. 
	\begin{lemma}
		\label{lemma.energy.pressure}
		The implicit pressure subsystem in the incompressible equations is kinetic energy stable at the fully-discrete level, in the sense that for the subsystem 
		\begin{equation}
			\begin{aligned}
				\rho\frac{\mathbf{u}^{n+1}_c-\mathbf{u}^*_c}{\Delta t} + \nabla_c^p p_p^{n+1} &= 0, \\
				\nabla_p^c \cdot \mathbf{u}_c^{n+1} & = 0,
			\end{aligned}
		\end{equation}
		the total kinetic energy is non-increasing, i.e. summing over all cells in the domain we have 
		\begin{equation}
			\sum_c \frac{1}{2} \rho |\omega_c| \left( \mathbf{u}_c^{n+1} \right)^2 \leq 
			\sum_c \frac{1}{2} \rho |\omega_c| \left( \mathbf{u}_c^{*} \right)^2.
		\end{equation}  
	\end{lemma}
	
	\begin{proof}
		The total kinetic energy at time $t^n$ is defined as
		$ 	\mathcal{E}^n = \sum_{c} |\omega_c| \frac{1}{2} \rho (\mathbf{u}_c^n)^2$.
		We now multiply the momentum equation of the pressure subsystem by $\mathbf{u}_c^{n+1} |\omega_c|$ and the divergence-free condition by $p_p^{n+1} |\omega_p|$.
		We obtain
		\begin{equation}
			\begin{aligned}
				\rho \, \mathbf{u}_c^{n+1} |\omega_c| \cdot \frac{\mathbf{u}_c^{n+1}-\mathbf{u}_c^{*}}{\Delta t} + \mathbf{u}_c^{n+1} |\omega_c| \cdot \nabla_c^p p_p^{n+1} &= 0, \\
				p_p^{n+1} |\omega_p| \nabla_p^c \cdot \mathbf{u}_c^{n+1} & = 0,
			\end{aligned}
		\end{equation}
		and arranging the equations and substituting the discrete formulations of our compatible nabla operators \eqref{operators}, we obtain
		\begin{equation}
			\begin{aligned}
				\rho |\omega_c| \frac{\mathbf{u}_c^{n+1} \cdot \mathbf{u}_c^{n+1}-\mathbf{u}_c^{n+1} \cdot \mathbf{u}_c^{*}}{\Delta t} + \mathbf{u}_c^{n+1} \cdot \sum_{p\in\mathcal{P}(c)}l_{pc}\mathbf{n}_{pc} p_p^{n+1} &= 0, \\
				p_p^{n+1} \sum_{c\in \mathcal{C}(p)}l_{cp}\mathbf{n}_{cp} \cdot \mathbf{u}_c^{n+1} & = 0.
			\end{aligned}
			\label{discretized euler for kinetic energy}
		\end{equation} 
		The numerator in the first term of the first equation can be rewritten as
		\begin{eqnarray}		
			\mathbf{u}_c^{n+1} \cdot \mathbf{u}_c^{n+1}-\mathbf{u}_c^{n+1} \cdot \mathbf{u}_c^{*} & = & \mathbf{u}_c^{n+1} \cdot \mathbf{u}_c^{n+1}-\mathbf{u}_c^{n+1} \cdot \mathbf{u}_c^{*} + \mathbf{u}_c^{*} \cdot \mathbf{u}_c^{*} - \mathbf{u}_c^{*} \cdot \mathbf{u}_c^{*} \nonumber \\ 
			& = & \frac{1}{2} \left( \mathbf{u}_c^{n+1} - \mathbf{u}_c^{*} \right)^2 + \frac{1}{2} \left( \mathbf{u}_c^{n+1} \right)^2 - \frac{1}{2} \left( \mathbf{u}_c^{*} \right)^2.  
		\end{eqnarray}
		Since $\mathbf{u}_c^{n+1}$ does not depend on $p$ and $p_p^{n+1}$ does not depend on $c$ and with $l_{cp}=l_{pc}$ and $\mathbf{n}_{cp} = -\mathbf{n}_{pc}$, we can rewrite the equations as 
		\begin{equation}
			\begin{aligned}
				\rho |\omega_c| \frac{\frac{1}{2} \left( \mathbf{u}_c^{n+1} - \mathbf{u}_c^{*} \right)^2 + \frac{1}{2} \left( \mathbf{u}_c^{n+1} \right)^2 - \frac{1}{2} \left( \mathbf{u}_c^{*} \right)^2}{\Delta t} +  \sum_{p\in\mathcal{P}(c)}l_{pc}\mathbf{n}_{pc} \cdot \mathbf{u}_c^{n+1}  p_p^{n+1} &= 0, \\
				-\sum_{c\in \mathcal{C}(p)}l_{pc}\mathbf{n}_{pc} \cdot \mathbf{u}_c^{n+1}  p_p^{n+1} & = 0.
			\end{aligned}
			\label{discretized euler for kinetic energy2}
		\end{equation} 
		Summing the first equation over all cells $c$ of the domain, summing the second equation over all points $p$ of the domain and summing up both equations yields 	
		\begin{equation}
			\sum \limits_c \rho |\omega_c| \frac{\frac{1}{2} \left( \mathbf{u}_c^{n+1} - \mathbf{u}_c^{*} \right)^2 + \frac{1}{2} \left( \mathbf{u}_c^{n+1} \right)^2 - \frac{1}{2} \left( \mathbf{u}_c^{*} \right)^2}{\Delta t} 
			= 0.
		\end{equation} 
		Multiplication by $\Delta t$ leads to 	
		\begin{equation}
			\sum \limits_c \rho |\omega_c| \frac{1}{2} \left( \mathbf{u}_c^{n+1} - \mathbf{u}_c^{*} \right)^2 + \sum \limits_c \rho |\omega_c| \frac{1}{2} \left( \mathbf{u}_c^{n+1} \right)^2 - \sum \limits_c \rho |\omega_c| \frac{1}{2} \left( \mathbf{u}_c^{*} \right)^2 = 0.  
		\end{equation}
		Since $\left( \mathbf{u}_c^{n+1} - \mathbf{u}_c^{*} \right)^2 \geq 0$, we obtain the desired result 
		\begin{equation}
			\sum \limits_c \rho |\omega_c| \frac{1}{2} \left( \mathbf{u}_c^{n+1} \right)^2 - \sum \limits_c \rho |\omega_c| \frac{1}{2} \left( \mathbf{u}_c^{*} \right)^2 \leq 0.  
		\end{equation}
		\vspace{3mm} 
	\end{proof}
	%
	To establish energy stability of the complete scheme, we now consider the explicit convective step \eqref{star equations} with the Rusanov flux \eqref{eqn.rusanov}, followed by the pressure projection \eqref{velocity update with u*}--\eqref{pressure correction equation}. While the projection does not increase the kinetic energy, stability of the explicit step requires sufficient numerical dissipation together with a suitable CFL condition.
	
	For the incompressible Euler equations with constant density, the Rusanov flux across an edge $\partial \omega_{ac}$ reads
	\begin{equation}
		\mathbf{f}_{ac} = \frac{\rho}{2} \left( \mathbf{u}_a \left( \mathbf{u}_a \cdot \mathbf{n}_{ac} \right) + \mathbf{u}_c \left( \mathbf{u}_c \cdot \mathbf{n}_{ac} \right) \right) - \frac{\rho}{2} \, s_{ac} \, [\mathbf{u}]_{ac}, 
		\qquad [\mathbf{u}]_{ac} = \mathbf{u}_a - \mathbf{u}_c, 
		\label{eqn.rus.euler}
	\end{equation}
	where $[\mathbf{u}]_{ac}$ denotes the jump of the velocity across the edge and $s_{ac} \geq 0$ is the dissipation coefficient of the Rusanov flux on the edge $\partial \omega_{ac}$, which is the maximum signal speed in \eqref{eqn.rusanov} and is kept as a free parameter in the following. The numerical dissipation produced by the Rusanov flux at time $t^n$ is
	\begin{equation}
		D^n = \frac{1}{2} \rho \sum_{\partial \omega_{ac}} |\partial \omega_{ac}| \, s_{ac} \, |[\mathbf{u}^n]_{ac}|^2 \geq 0, 
		\label{eqn.energy.D}
	\end{equation}
	where the sum runs over all edges of the primal mesh, each counted once.
	
	For each node $p$, let $\mathcal{S}(p)$ denote the set of edges sharing that node. We define
	\begin{equation}
		U_p = \max_{c \in \mathcal{C}(p)} |\mathbf{u}_c^n|, 
		\qquad 
		J_p = \sum_{\partial \omega_{ac} \in \mathcal{S}(p)} |\partial \omega_{ac}| \, |[\mathbf{u}^n]_{ac}|^2, 	
		\qquad   
		H_p = \sum_{\partial \omega_{ac} \in \mathcal{S}(p)} |\partial \omega_{ac}| \left( \mathbf{u}_a^n \cdot \mathbf{u}_c^n \right) \left( [\mathbf{u}^n]_{ac} \cdot \mathbf{n}_{ac} \right), 
		\label{eqn.energy.UJH} 
	\end{equation}
	and the nodal dissipation coefficient 
	\begin{equation}
		s_p = 2 \max \left\{ U_p, \, \frac{(-H_p)_+}{J_p} \right\}, \qquad (x)_+ = \max(x, 0), 
		\label{eqn.energy.sp}
	\end{equation}
	where the quotient is set to zero if $J_p = 0$. In this case, all velocity jumps vanish and $H_p = 0$.
	\begin{lemma}
		\label{lemma.energy.full}
		Consider the incompressible Euler equations \eqref{eqn.inc.euler} with constant density $\rho$ on a periodic domain and assume that the dissipation coefficients of the Rusanov flux satisfy 
		\begin{equation}
			s_{ac} \geq \max \left( s_p, s_q \right) \qquad \textnormal{for every edge $\partial \omega_{ac}$ with endpoints $p$ and $q$}, 
			\label{eqn.energy.hyp1}
		\end{equation}
		and that the time step satisfies 
		\begin{equation}
			\Delta t \sum_{a \in \mathcal{N}_c} |\partial \omega_{ac}| \, s_{ac} \leq \beta \, |\omega_c| \qquad \textnormal{for all cells $c$}, \qquad \textnormal{with} \qquad 0 < \beta \leq \frac{1}{4}. 
			\label{eqn.energy.hyp2}
		\end{equation}
		Then the total kinetic energy $\mathcal{E}^n = \sum_c \frac{1}{2} \rho |\omega_c| |\mathbf{u}_c^n|^2$ of the complete scheme is non-increasing: 
		\begin{equation}
			\mathcal{E}^{n+1} \leq \mathcal{E}^n - \left( \frac{1}{2} - 2 \, \beta \right) \Delta t \, D^n \leq \mathcal{E}^n. 
			\label{eqn.energy.result}
		\end{equation}
	\end{lemma}
	\begin{proof}
		We first derive the energy balance of the explicit step. We then bound the convective term and the energy gain due to the explicit Euler discretization. Finally, we use Lemma~\ref{lemma.energy.pressure} for the pressure projection.
		
		\textit{Step 1: energy balance of the explicit step.} From \eqref{star equations} and \eqref{eqn.rus.euler}, the explicit momentum update reads 
		\begin{equation}
			\mathbf{u}_c^* = \mathbf{u}_c^n - \frac{\Delta t}{\rho \, |\omega_c|} \sum_{a \in \mathcal{N}_c} |\partial \omega_{ac}| \, \mathbf{f}_{ac}, 
			\label{eqn.energy.du}
		\end{equation}
		where all velocities in the flux are taken at time $t^n$. Let $\delta \mathbf{u}_c = \mathbf{u}_c^* - \mathbf{u}_c^n$.  
		Multiplying \eqref{eqn.energy.du} by $\rho |\omega_c| \mathbf{u}_c^n$, summing over all cells, and using
		$|\mathbf{u}_c^*|^2 = |\mathbf{u}_c^n|^2 + 2 \mathbf{u}_c^n \cdot \delta \mathbf{u}_c + |\delta \mathbf{u}_c|^2$, we obtain
		\begin{equation}
			\mathcal{E}^* - \mathcal{E}^n = \Gamma^n - \Delta t \sum_c \sum_{a \in \mathcal{N}_c} |\partial \omega_{ac}| \, \mathbf{u}_c^n \cdot \mathbf{f}_{ac}, \qquad 
			\Gamma^n = \sum_c \frac{1}{2} \rho |\omega_c| \, |\delta \mathbf{u}_c|^2 \geq 0, 
			\label{eqn.energy.step1}
		\end{equation}
		where $\Gamma^n$ is the energy gain of the explicit Euler time discretization. In the rest of the proof the superscript $n$ on the velocity field is omitted to ease notation. 
		In $\mathbf{u}_c \cdot \mathbf{f}_{ac}$, the term $\frac{1}{2} \rho \, |\mathbf{u}_c|^2 \left( \mathbf{u}_c \cdot \mathbf{n}_{ac} \right)$ vanishes after summation over $a \in \mathcal{N}_c$, since $\sum_{a \in \mathcal{N}_c} |\partial \omega_{ac}| \, \mathbf{n}_{ac} = 0$, see \eqref{eqn.gauss}. Collecting the remaining contributions of the two cells $\omega_a$ and $\omega_c$ which share the edge $\partial \omega_{ac}$, with $\mathbf{n}_{ca} = - \mathbf{n}_{ac}$ and $[\mathbf{u}]_{ca} = - [\mathbf{u}]_{ac}$, we obtain the energy balance of the explicit step 
		\begin{equation}
			\mathcal{E}^* - \mathcal{E}^n = \Gamma^n - \Delta t \left( T^n + D^n \right),  		 
			\label{eqn.energy.balance}
		\end{equation}
		with $D^n$ the Rusanov dissipation \eqref{eqn.energy.D}, while
		\begin{equation}
			T^n = \frac{1}{2} \rho \sum_{\partial \omega_{ac}} |\partial \omega_{ac}| \left( \mathbf{u}_a \cdot \mathbf{u}_c \right) \left( [\mathbf{u}]_{ac} \cdot \mathbf{n}_{ac} \right), 
		\end{equation}	
		represents the convective contribution.
		
		\textit{Step 2: bound of the convective term at each node.} Each edge has two endpoints. We split its contribution equally between them and write $T^n = \sum_p T_p$ and $D^n = \sum_p D_p$, where 
		\begin{equation}
			T_p = \frac{1}{4} \rho \hspace{-2mm} \sum_{\partial \omega_{ac} \in \mathcal{S}(p)}  \hspace{-2mm} |\partial \omega_{ac}| \left( \mathbf{u}_a \cdot \mathbf{u}_c \right) \left( [\mathbf{u}]_{ac} \cdot \mathbf{n}_{ac} \right) = \frac{1}{4} \rho \, H_p, \qquad 
			D_p = \frac{1}{4} \rho  \hspace{-2mm} \sum_{\partial \omega_{ac} \in \mathcal{S}(p)} |\partial \omega_{ac}| \, s_{ac} \, |[\mathbf{u}]_{ac}|^2, 
			\label{eqn.energy.Tp}
		\end{equation}
		where $H_p$ is given by \eqref{eqn.energy.UJH}. Since $s_{ac} \geq s_p$ for every edge of $\mathcal{S}(p)$ by \eqref{eqn.energy.hyp1}, the definitions \eqref{eqn.energy.UJH} and \eqref{eqn.energy.sp} give 
		\begin{equation}
			D_p \geq \frac{1}{4} \rho \, s_p \, J_p \geq \frac{1}{2} \rho \, (-H_p)_+ = 2 \, (-T_p)_+. 
			\label{eqn.energy.Tbound}
		\end{equation}
		It follows that $D_p + T_p \geq D_p - (-T_p)_+ \geq \frac{1}{2} D_p$. Summing over all nodes gives
		\begin{equation}
			T^n + D^n \geq \frac{1}{2} \, D^n. 
			\label{eqn.energy.theta}
		\end{equation}
		
		\textit{Step 3: bound of the explicit Euler gain.} 
		We now bound $\Gamma^n$. Using \eqref{eqn.gauss}, we can subtract $\rho \mathbf{u}_c (\mathbf{u}_c \cdot \mathbf{n}_{ac})$ from each flux without changing $\delta \mathbf{u}_c$. This gives 
		\begin{equation}
			\mathbf{f}_{ac} - \rho \, \mathbf{u}_c \left( \mathbf{u}_c \cdot \mathbf{n}_{ac} \right) = \frac{1}{2} \rho \, \mathbf{u}_a \left( [\mathbf{u}]_{ac} \cdot \mathbf{n}_{ac} \right) + \frac{1}{2} \rho \, [\mathbf{u}]_{ac} \left( \mathbf{u}_c \cdot \mathbf{n}_{ac} \right) - \frac{1}{2} \rho \, s_{ac} \, [\mathbf{u}]_{ac}, 
			\label{eqn.energy.fred}
		\end{equation}
		whose modulus is bounded by $\frac{1}{2} \rho \left( |\mathbf{u}_a| + |\mathbf{u}_c| + s_{ac} \right) |[\mathbf{u}]_{ac}| \leq \rho \, s_{ac} \, |[\mathbf{u}]_{ac}|$, because \eqref{eqn.energy.hyp1} and \eqref{eqn.energy.sp} imply $s_{ac} \geq s_p \geq 2 \, U_p \geq |\mathbf{u}_a| + |\mathbf{u}_c|$ for an endpoint $p$ of the edge $\partial \omega_{ac}$. Hence $|\delta \mathbf{u}_c| \leq \frac{\Delta t}{|\omega_c|} \sum_{a \in \mathcal{N}_c} |\partial \omega_{ac}| \, s_{ac} \, |[\mathbf{u}]_{ac}|$, and the Cauchy-Schwarz inequality together with the time step restriction \eqref{eqn.energy.hyp2} yields 
		\begin{equation}
			|\omega_c| \, |\delta \mathbf{u}_c|^2  \leq  \frac{\Delta t^2}{|\omega_c|} \left( \sum_{a \in \mathcal{N}_c} |\partial \omega_{ac}| \, s_{ac} \right) \left( \sum_{a \in \mathcal{N}_c} |\partial \omega_{ac}| \, s_{ac} \, |[\mathbf{u}]_{ac}|^2 \right) 
			\leq  \beta \, \Delta t \sum_{a \in \mathcal{N}_c} |\partial \omega_{ac}| \, s_{ac} \, |[\mathbf{u}]_{ac}|^2. 
			\label{eqn.energy.gaincell}
		\end{equation}
		Summing over all cells, every edge is counted twice, so that 
		\begin{equation}
			\Gamma^n \leq \frac{1}{2} \rho \, \beta \, \Delta t \sum_c \sum_{a \in \mathcal{N}_c} |\partial \omega_{ac}| \, s_{ac} \, |[\mathbf{u}]_{ac}|^2 = 2 \, \beta \, \Delta t \, D^n. 
			\label{eqn.energy.gain}
		\end{equation}
		
		Combining \eqref{eqn.energy.balance}, \eqref{eqn.energy.theta} and \eqref{eqn.energy.gain} we obtain 
		\begin{equation}
			\mathcal{E}^* - \mathcal{E}^n \leq 2 \, \beta \, \Delta t \, D^n - \frac{1}{2} \, \Delta t \, D^n = - \frac{1}{2} \left( 1 - 4 \, \beta \right) \Delta t \, D^n, 
			\label{eqn.energy.concl}
		\end{equation}
		Since $\beta \leq 1/4$, the right-hand side is non-positive. Moreover, Lemma~\ref{lemma.energy.pressure} gives $\mathcal{E}^{n+1} \leq \mathcal{E}^*$. This completes the proof.
		
	\end{proof}
	%
	The proof of Lemma~\ref{lemma.energy.full} uses the hypotheses \eqref{eqn.energy.hyp1}--\eqref{eqn.energy.hyp2} only through the two inequalities \eqref{eqn.energy.theta} and $\Gamma^n \leq \frac{1}{2} \Delta t \, D^n$, which may also hold for smaller dissipation coefficients and larger time steps. 
	%
	
	For the incompressible Navier-Stokes equations the implicit discretization \eqref{eqn.sigma.implicit} of the viscous stress tensor introduces an additional non-negative dissipation.
	\begin{lemma}
		\label{lemma.energy.ns}
		Consider the incompressible Navier-Stokes equations with constant density $\rho$ and viscosity $\mu \geq 0$ on a periodic domain, discretized by the first order scheme \eqref{eqn.mom.ns} with the implicit viscous stress tensor~\eqref{eqn.sigma.implicit}, followed by the projection \eqref{velocity update with u*}--\eqref{pressure correction equation}. Assume that the Rusanov flux \eqref{eqn.rus.euler} and the time step satisfy \eqref{eqn.energy.hyp1}--\eqref{eqn.energy.hyp2}. Then 
		\begin{equation}
			\mathcal{E}^{n+1} \leq \mathcal{E}^n - \frac{1}{2} \left( 1 - 4 \, \beta \right) \Delta t \, D^n - \Delta t \, D_\mu^* \leq \mathcal{E}^n, 
			\qquad 
			D_\mu^* = - \sum_p |\omega_p| \, \boldsymbol{\sigma}_p : \nabla_p^c \mathbf{u}_c^* \geq 0. 
			\label{eqn.energy.ns}
		\end{equation}
	\end{lemma}
	\begin{proof}
		\textit{Step 1: explicit convective part.} For constant density, \eqref{eqn.mom.ns} with \eqref{eqn.sigma.implicit} reads 
		\begin{equation}
			\rho \, \mathbf{u}_c^* + \Delta t \, \nabla_c^p \cdot \boldsymbol{\sigma}_p = \rho \, \tilde{\mathbf{u}}_c, 
			\label{eqn.energy.ns.visc}
		\end{equation}
		where $\tilde{\mathbf{u}}_c$ is the velocity obtained from the explicit convective step \eqref{eqn.energy.du}. 
		From the proof of Lemma~\ref{lemma.energy.full}, we already know that
		$\tilde{\mathcal{E}} \leq \mathcal{E}^n - \frac{1}{2} (1 - 4 \, \beta) \Delta t \, D^n$, where $\tilde{\mathcal{E}}$ is the kinetic energy associated with $\tilde{\mathbf{u}}_c$.
		
		\textit{Step 2: implicit viscous part.} On a periodic domain, the definitions \eqref{operators} and $l_{cp} \mathbf{n}_{cp} = - l_{pc} \mathbf{n}_{pc}$ give the discrete summation by parts 
		\begin{equation}
			\sum_c |\omega_c| \, \mathbf{u}_c \cdot \nabla_c^p \cdot \boldsymbol{\sigma}_p = \sum_c \sum_{p \in \mathcal{P}(c)} \mathbf{u}_c \cdot \boldsymbol{\sigma}_p \, l_{pc} \mathbf{n}_{pc} = - \sum_p |\omega_p| \, \boldsymbol{\sigma}_p : \nabla_p^c \mathbf{u}_c 
			\label{eqn.duality.visc}
		\end{equation}
		for any vector field $\mathbf{u}_c$ and symmetric tensor field $\boldsymbol{\sigma}_p$. 
		Multiplying \eqref{eqn.energy.ns.visc} by $|\omega_c| \, \mathbf{u}_c^*$, summing over all cells, and using
		$\mathbf{u}_c^* \cdot \tilde{\mathbf{u}}_c = \frac{1}{2} |\mathbf{u}_c^*|^2 + \frac{1}{2} |\tilde{\mathbf{u}}_c|^2 - \frac{1}{2} |\mathbf{u}_c^* - \tilde{\mathbf{u}}_c|^2$,
		we obtain
		\begin{equation}
			\mathcal{E}^* = \tilde{\mathcal{E}} - \frac{1}{2} \rho \sum_c |\omega_c| \, |\mathbf{u}_c^* - \tilde{\mathbf{u}}_c|^2 - \Delta t \, D_\mu^*. 
			\label{eqn.energy.ns.step2}
		\end{equation}
		We now show that $D_\mu^* \geq 0$. In two space dimensions the Cauchy-Schwarz inequality gives $(\nabla_p^c \cdot \mathbf{u}_c^*)^2 \leq \frac{1}{2} |\nabla_p^c \mathbf{u}_c^* + (\nabla_p^c \mathbf{u}_c^*)^T|^2$. Using \eqref{eqn.sigma.implicit}, we obtain 
		\begin{equation}
			- \boldsymbol{\sigma}_p : \nabla_p^c \mathbf{u}_c^*  =  \mu \left( \frac{1}{2} |\nabla_p^c \mathbf{u}_c^* + (\nabla_p^c \mathbf{u}_c^*)^T|^2 - \frac{2}{3} (\nabla_p^c \cdot \mathbf{u}_c^*)^2 \right) \nonumber \\
			\geq  \frac{1}{6} \, \mu \, |\nabla_p^c \mathbf{u}_c^* + (\nabla_p^c \mathbf{u}_c^*)^T|^2 \geq 0, 
			\label{eqn.energy.ns.pos}
		\end{equation}
		and therefore $D_\mu^* \geq 0$. 
		
		\textit{Step 3: projection.} Finally, Lemma~\ref{lemma.energy.pressure} gives $\mathcal{E}^{n+1} \leq \mathcal{E}^*$. Combining this with Steps 1 and 2 completes the proof of \eqref{eqn.energy.ns}. 
	\end{proof}
	The implicit viscous terms do not require any additional time step restriction. 
	Lemmas~\ref{lemma.energy.full} and \ref{lemma.energy.ns} concern the first order scheme for the incompressible Euler and Navier-Stokes equations. 
	
	The magnetic and elastic terms of the incompressible MHD and GPR systems are discussed in Sections~\ref{sec.energy.mhd} and \ref{sec.energy.gpr}. For the second-order scheme presented in Section~\ref{sec.second.order}, Lemma~\ref{lemma.energy.pressure} applies to each stage, but the piecewise linear reconstruction is not covered by the present analysis.

	\subsection{Energy stability of the compatible MHD discretization}
	\label{sec.energy.mhd}
	For the incompressible MHD equations the total energy is the sum of the kinetic and of the magnetic energy, whose discrete counterpart is 
	\begin{equation}
		\mathcal{E} = \sum_c |\omega_c| \left( \frac{1}{2} \rho \, |\mathbf{u}_c|^2 + \frac{1}{2} |\mathbf{B}_c|^2 \right). 
		\label{eqn.mhd.energy}
	\end{equation}
	In order to study the energy stability of the scheme presented in Section~\ref{sec.mhd}, we first introduce two discrete duality relations. On a periodic domain, for any vector field $\mathbf{w}_c$ defined at the barycenters and any vector field $\mathbf{a}_p$ defined at the nodes, the definitions of the discrete curl operators \eqref{operators}, together with Lemma~\ref{identity}, $l_{cp} \mathbf{n}_{cp} = - l_{pc} \mathbf{n}_{pc}$, and the cyclic property of the triple product, give
	\begin{eqnarray}
		\sum_c |\omega_c| \, \mathbf{w}_c \cdot \nabla_c^p \times \mathbf{a}_p 
		& = & \sum_c \sum_{p \in \mathcal{P}(c)} \mathbf{w}_c \cdot \left( l_{pc} \mathbf{n}_{pc} \times \mathbf{a}_p \right) 
		\; = \; \sum_p \sum_{c \in \mathcal{C}(p)} \mathbf{a}_p \cdot \left( l_{cp} \mathbf{n}_{cp} \times \mathbf{w}_c \right) \nonumber \\
		& = & \sum_p |\omega_p| \, \mathbf{a}_p \cdot \nabla_p^c \times \mathbf{w}_c, 
		\label{eqn.duality.curl}
	\end{eqnarray}
	while the nodal average \eqref{eqn.avg.node} yields 
	\begin{equation}
		\sum_c |\omega_c| \, \mathbf{w}_c \cdot \frac{1}{|\omega_c|} \sum_{p \in \mathcal{P}(c)} |\omega_{pc}| \, \mathbf{a}_p 
		= \sum_p \mathbf{a}_p \cdot \sum_{c \in \mathcal{C}(p)} |\omega_{pc}| \, \mathbf{w}_c 
		= \sum_p |\omega_p| \, \mathbf{a}_p \cdot \mathbf{w}_p. 
		\label{eqn.duality.avg}
	\end{equation}
	The average with the weights $|\omega_{pc}| / |\omega_c|$, which coincides with \eqref{eqn.avg.cell}, is thus the adjoint of the nodal average~\eqref{eqn.avg.node}, in the same way as the primal curl is the adjoint of the dual curl, and these two adjoint pairs are exactly the ones that appear in the induction equation \eqref{eqn.mhd.induction} and in the Lorentz force of \eqref{eqn.mhd.mom}. 
	For the fully discrete scheme, we define the current density $\mathbf{j}_p^n = \nabla_p^c \times \mathbf{B}_c^n$, the discrete Lorentz force $\mathbf{F}_c^n$, and the resistive dissipation $D_\eta^n$ as
	\begin{equation}
		\mathbf{F}_c^n = - \frac{1}{|\omega_c|} \sum_{p \in \mathcal{P}(c)} |\omega_{pc}| \, \mathbf{B}_c^n \times \mathbf{j}_p^n,
		\qquad
		D_\eta^n = \sum_p |\omega_p| \, \eta_p \, |\mathbf{j}_p^n|^2 \geq 0,
		\label{eqn.mhd.FD}
	\end{equation}
	and by 
	\begin{equation}
		G^n = \frac{1}{\rho} \sum_c |\omega_c| \, |\mathbf{F}_c^n|^2 + \frac{1}{2} \sum_c |\omega_c| \, |\nabla_c^p \times \mathbf{E}_p^n|^2 \geq 0 
		\label{eqn.mhd.G}
	\end{equation}
	which measures the energy gain due to the explicit Euler discretization of the magnetic terms. All these quantities are evaluated at time $t^n$.
	\begin{lemma}
		\label{lemma.energy.mhd}
		Consider the incompressible MHD equations \eqref{eqn.inc.mhd} with constant density $\rho$ on a periodic domain, discretized by the first order scheme \eqref{eqn.mhd.mom}, \eqref{eqn.mhd.induction}, \eqref{velocity update with u*} with the Rusanov flux \eqref{eqn.rus.euler} and with the divergence-free condition $\nabla_p^c \cdot \mathbf{u}_c^{n+1} = 0$. Assume that the dissipation coefficients of the Rusanov flux satisfy \eqref{eqn.energy.hyp1}, the time step satisfies \eqref{eqn.energy.hyp2} with $0 < \beta < \frac{1}{8}$, and  
		\begin{equation}
			\Delta t \, G^n \leq \frac{1}{2} \left( \left( \frac{1}{2} - 4 \, \beta \right) D^n + D_\eta^n \right). 
			\label{eqn.mhd.hyp}
		\end{equation}
		Then the total energy \eqref{eqn.mhd.energy} of the complete scheme is non-increasing: 
		\begin{equation}
			\mathcal{E}^{n+1} \leq \mathcal{E}^n - \frac{1}{2} \, \Delta t \left( \left( \frac{1}{2} - 4 \, \beta \right) D^n + D_\eta^n \right) \leq \mathcal{E}^n. 
			\label{eqn.mhd.result}
		\end{equation}
	\end{lemma}
	\begin{proof}
		\textit{Step 1: energy balance of the explicit step.} Let $\delta \mathbf{u}_c = \mathbf{u}_c^* - \mathbf{u}_c^n$ and $\delta \mathbf{B}_c = \mathbf{B}_c^{n+1} - \mathbf{B}_c^n$, and let $\mathcal{E}^*$ be the total energy \eqref{eqn.mhd.energy} evaluated with $\mathbf{u}_c^*$ and $\mathbf{B}_c^{n+1}$. As in Step~1 of the proof of Lemma~\ref{lemma.energy.full}, 
		\begin{equation}
			\mathcal{E}^* - \mathcal{E}^n = \Gamma^n + \sum_c |\omega_c| \left( \rho \, \mathbf{u}_c^n \cdot \delta \mathbf{u}_c + \mathbf{B}_c^n \cdot \delta \mathbf{B}_c \right), 
			\qquad 
			\Gamma^n = \frac{1}{2} \sum_c |\omega_c| \left( \rho \, |\delta \mathbf{u}_c|^2 + |\delta \mathbf{B}_c|^2 \right), 
			\label{eqn.mhd.step1}
		\end{equation}
		where, from \eqref{eqn.mhd.mom} and \eqref{eqn.mhd.induction},
		$\rho \, \delta \mathbf{u}_c = - \frac{\Delta t}{|\omega_c|} \sum_{a \in \mathcal{N}_c} |\partial \omega_{ac}| \, \mathbf{f}_{ac} + \Delta t \, \mathbf{F}_c^n$
		and $\delta \mathbf{B}_c = - \Delta t \, \nabla_c^p \times \mathbf{E}_p^n$. 
		
		The convective contribution is equal to $-\Delta t (T^n + D^n)$, as shown in Lemma~\ref{lemma.energy.full}. For the remaining terms, we use the duality relations \eqref{eqn.duality.curl} and \eqref{eqn.duality.avg}, together with the definitions of the Lorentz force and the electric field. This gives
		\begin{eqnarray}
			\Delta t \sum_c |\omega_c| \left( \mathbf{u}_c^n \cdot \mathbf{F}_c^n - \mathbf{B}_c^n \cdot \nabla_c^p \times \mathbf{E}_p^n \right) 
			& = & - \Delta t \sum_c \sum_{p \in \mathcal{P}(c)} |\omega_{pc}| \, \mathbf{u}_c^n \cdot \left( \mathbf{B}_c^n \times \mathbf{j}_p^n \right) - \Delta t \sum_p |\omega_p| \, \mathbf{E}_p^n \cdot \mathbf{j}_p^n \nonumber \\
			& = & - \Delta t \sum_c \sum_{p \in \mathcal{P}(c)} |\omega_{pc}| \left( \mathbf{u}_c^n \cdot \left( \mathbf{B}_c^n \times \mathbf{j}_p^n \right) - \left( \mathbf{u}_c^n \times \mathbf{B}_c^n \right) \cdot \mathbf{j}_p^n \right) - \Delta t \, D_\eta^n \nonumber \\
			& = & - \Delta t \, D_\eta^n, 
			\label{eqn.mhd.cancel}
		\end{eqnarray}
		since $(\mathbf{u}_c^n \times \mathbf{B}_c^n) \cdot \mathbf{j}_p^n = \mathbf{u}_c^n \cdot (\mathbf{B}_c^n \times \mathbf{j}_p^n)$.
		
		Thus, the contributions of the Lorentz force and the ideal induction equation cancel exactly, leaving only the resistive dissipation. The energy balance becomes
		\begin{equation}
			\mathcal{E}^* - \mathcal{E}^n = \Gamma^n - \Delta t \left( T^n + D^n \right) - \Delta t \, D_\eta^n.
			\label{eqn.mhd.balance}
		\end{equation}
		
		\textit{Step 2: bound of the convective term.} From Step~2 of Lemma~\ref{lemma.energy.full}, assumption \eqref{eqn.energy.hyp1} gives $T^n + D^n \geq \frac{1}{2} D^n$.
		
		\textit{Step 3: bound of the explicit Euler gain.} We now bound $\Gamma^n$. We write
		$\delta \mathbf{u}_c = \delta \mathbf{u}_c^{\textnormal{conv}} + \frac{\Delta t}{\rho} \mathbf{F}_c^n$,
		where $\delta \mathbf{u}_c^{\textnormal{conv}}$ is the explicit convective update \eqref{eqn.energy.du}. From \eqref{eqn.energy.gain}, its energy gain is bounded by $2 \, \beta \, \Delta t \, D^n$ under assumptions \eqref{eqn.energy.hyp1}--\eqref{eqn.energy.hyp2}.
		
		Using $|\mathbf{a} + \mathbf{b}|^2 \leq 2 |\mathbf{a}|^2 + 2 |\mathbf{b}|^2$ and the definition \eqref{eqn.mhd.G}, we obtain
		\begin{equation}
			\Gamma^n \leq \rho \sum_c |\omega_c| \, |\delta \mathbf{u}_c^{\textnormal{conv}}|^2 + \frac{\Delta t^2}{\rho} \sum_c |\omega_c| \, |\mathbf{F}_c^n|^2 + \frac{\Delta t^2}{2} \sum_c |\omega_c| \, |\nabla_c^p \times \mathbf{E}_p^n|^2 \leq 4 \, \beta \, \Delta t \, D^n + \Delta t^2 \, G^n. 
			\label{eqn.mhd.gain}
		\end{equation}
		
		Combining \eqref{eqn.mhd.balance}, the bound from Step~2, \eqref{eqn.mhd.gain}, and assumption \eqref{eqn.mhd.hyp}, we obtain
		\begin{equation}
			\mathcal{E}^* - \mathcal{E}^n \leq - \Delta t \left( \left( \frac{1}{2} - 4 \, \beta \right) D^n + D_\eta^n \right) + \Delta t^2 \, G^n \leq - \frac{1}{2} \, \Delta t \left( \left( \frac{1}{2} - 4 \, \beta \right) D^n + D_\eta^n \right). 
			\label{eqn.mhd.concl}
		\end{equation}
		Finally, the projection step does not change the magnetic field and, by Lemma~\ref{lemma.energy.pressure}, does not increase the kinetic energy. Therefore, $\mathcal{E}^{n+1} \leq \mathcal{E}^*$, which proves \eqref{eqn.mhd.result}. 
	\end{proof}
	Condition \eqref{eqn.mhd.hyp} requires the energy gain of the explicit magnetic update to be controlled by the Rusanov and resistive dissipation. By Jensen's inequality, $\frac{1}{\rho} \sum_c |\omega_c| \, |\mathbf{F}_c^n|^2 \leq \frac{1}{\rho} \sum_p \sum_{c \in \mathcal{C}(p)} |\omega_{pc}| \, |\mathbf{B}_c^n|^2 \, |\mathbf{j}_p^n|^2$, so that the contribution of the Lorentz force to $\Delta t \, G^n$ is bounded by one half of the resistive term in \eqref{eqn.mhd.hyp} as soon as $\rho \, |\omega_p| \, \eta_p \geq 4 \, \Delta t \sum_{c \in \mathcal{C}(p)} |\omega_{pc}| \, |\mathbf{B}_c^n|^2$, i.e. the numerical resistivity must be at least of the order of $\Delta t$ times the square of the Alfv\'en speed. Moreover, if $\eta_p > 0$, condition \eqref{eqn.mhd.hyp} holds for sufficiently small time steps. Indeed, if $G^n > 0$, then $\mathbf{B}_c^n \neq 0$, and since $\mathbf{B}_c^n = \nabla_c^p \times \mathbf{A}_p^n$ with the discrete vector potential $\mathbf{A}_p^{n+1} = \mathbf{A}_p^n - \Delta t \, \mathbf{E}_p^n$, \eqref{eqn.duality.curl} yields $\sum_c |\omega_c| \, |\mathbf{B}_c^n|^2 = \sum_p |\omega_p| \, \mathbf{A}_p^n \cdot \mathbf{j}_p^n$, so that $\mathbf{j}_p^n$ does not vanish at all nodes and $D_\eta^n > 0$.

	\subsection{Energy stability of the compatible GPR discretization}
	\label{sec.energy.gpr}
	For the incompressible GPR model, the total energy is the sum of the kinetic and of the elastic energy. Its discrete counterpart is
	\begin{equation}
		\mathcal{E} = \sum_c |\omega_c| \left( \frac{1}{2} \rho \, |\mathbf{u}_c|^2 + \mathcal{E}_s(\mathbf{A}_c) \right). 
		\label{eqn.gpr.energy}
	\end{equation}
	For any matrix field $\mathbf{W}_c$ defined at the barycenters and any vector field $\mathbf{z}_p$ defined at the nodes, the discrete operators \eqref{operators} give the following duality relation on a periodic domain:
	\begin{equation}
		\sum_c |\omega_c| \, \mathbf{W}_c : \nabla_c^p \mathbf{z}_p = - \sum_p |\omega_p| \, \mathbf{z}_p \cdot \nabla_p^c \cdot \mathbf{W}_c, 
		\label{eqn.duality.grad}
	\end{equation}
	where the gradient and divergence are applied row by row. This relation is used for the distortion equation \eqref{eqn.gpr.A}, while the elastic force in \eqref{eqn.gpr.mom} is treated using the averaging relation \eqref{eqn.duality.avg}. With $\boldsymbol{\alpha}_c^n = \boldsymbol{\alpha}(\mathbf{A}_c^n)$ we define the discrete elastic force and the dissipation associated with the distortion equation as
	\begin{equation}
		\mathbf{F}_c^n = - \frac{1}{|\omega_c|} \sum_{p \in \mathcal{P}(c)} |\omega_{pc}| \left( \mathbf{A}_p^n \right)^T \left( \nabla_p^a \cdot \boldsymbol{\alpha}_a^n \right), 
		\qquad 
		D_\zeta^n = \sum_p |\omega_p| \, \zeta_p \, |\nabla_p^c \cdot \boldsymbol{\alpha}_c^n|^2 \geq 0, 
		\label{eqn.gpr.FD}
	\end{equation}
	and introduce 
	\begin{equation}
		\Gamma_s^n = \sum_c |\omega_c| \left( \mathcal{E}_s(\mathbf{A}_c^{n+1}) - \mathcal{E}_s(\mathbf{A}_c^n) - \boldsymbol{\alpha}_c^n : \left( \mathbf{A}_c^{n+1} - \mathbf{A}_c^n \right) \right) 
		\label{eqn.gpr.Gamma}
	\end{equation}
	which represents the energy gain of the explicit Euler discretization of the distortion equation. This quantity can be computed after the update \eqref{eqn.gpr.A}. 
	\begin{lemma}
		\label{lemma.energy.gpr}
		Consider the incompressible GPR model with constant density $\rho$ on a periodic domain, discretized by the first order scheme \eqref{eqn.gpr.mom}, \eqref{eqn.gpr.A}, \eqref{velocity update with u*} with the Rusanov flux \eqref{eqn.rus.euler} and with the divergence-free condition $\nabla_p^c \cdot \mathbf{u}_c^{n+1} = 0$. Assume that the dissipation coefficients of the Rusanov flux satisfy \eqref{eqn.energy.hyp1}, the time step satisfies \eqref{eqn.energy.hyp2} with $0 < \beta < \frac{1}{8}$, and 
		\begin{equation}
			\frac{\Delta t^2}{\rho} \sum_c |\omega_c| \, |\mathbf{F}_c^n|^2 + \Gamma_s^n \leq \frac{1}{2} \, \Delta t \left( \left( \frac{1}{2} - 4 \, \beta \right) D^n + D_\zeta^n \right). 
			\label{eqn.gpr.hyp}
		\end{equation}
		Then the total energy \eqref{eqn.gpr.energy} of the complete scheme is non-increasing: 
		\begin{equation}
			\mathcal{E}^{n+1} \leq \mathcal{E}^n - \frac{1}{2} \, \Delta t \left( \left( \frac{1}{2} - 4 \, \beta \right) D^n + D_\zeta^n \right) \leq \mathcal{E}^n. 
			\label{eqn.gpr.result}
		\end{equation}
	\end{lemma}
	\begin{proof}
		\textit{Step 1: energy balance of the explicit step.} Let $\delta \mathbf{u}_c = \mathbf{u}_c^* - \mathbf{u}_c^n$ and $\delta \mathbf{A}_c = \mathbf{A}_c^{n+1} - \mathbf{A}_c^n$, and let $\mathcal{E}^*$ be the total energy \eqref{eqn.gpr.energy} evaluated with $\mathbf{u}_c^*$ and $\mathbf{A}_c^{n+1}$. By the definition \eqref{eqn.gpr.Gamma} of $\Gamma_s^n$, 
		\begin{equation}
			\mathcal{E}^* - \mathcal{E}^n = \frac{1}{2} \rho \sum_c |\omega_c| \, |\delta \mathbf{u}_c|^2 + \Gamma_s^n + \sum_c |\omega_c| \left( \rho \, \mathbf{u}_c^n \cdot \delta \mathbf{u}_c + \boldsymbol{\alpha}_c^n : \delta \mathbf{A}_c \right), 
			\label{eqn.gpr.step1}
		\end{equation}
		where, from \eqref{eqn.gpr.mom} and \eqref{eqn.gpr.A},
		$\rho \, \delta \mathbf{u}_c = - \frac{\Delta t}{|\omega_c|} \sum_{a \in \mathcal{N}_c} |\partial \omega_{ac}| \, \mathbf{f}_{ac} + \Delta t \, \mathbf{F}_c^n$
		and $\delta \mathbf{A}_c = - \Delta t \, \nabla_c^p \mathbf{z}_p^n$, with
		$\mathbf{z}_p^n = \mathbf{A}_p^n \mathbf{u}_p^n - \zeta_p \, \nabla_p^a \cdot \boldsymbol{\alpha}_a^n$.
		
		As shown in Lemma~\ref{lemma.energy.full}, the convective contribution is equal to $-\Delta t (T^n + D^n)$. For the remaining terms, we use the duality relations \eqref{eqn.duality.avg} and \eqref{eqn.duality.grad}, together with the definitions of the elastic force and $\mathbf{z}_p^n$. This gives
		\begin{equation}
			\Delta t \sum_c |\omega_c| \left( \mathbf{u}_c^n \cdot \mathbf{F}_c^n - \boldsymbol{\alpha}_c^n : \nabla_c^p \mathbf{z}_p^n \right) 
			= - \Delta t \sum_p |\omega_p| \left( \mathbf{u}_p^n \cdot \left( \mathbf{A}_p^n \right)^T \left( \nabla_p^c \cdot \boldsymbol{\alpha}_c^n \right) - \mathbf{z}_p^n \cdot \left( \nabla_p^c \cdot \boldsymbol{\alpha}_c^n \right) \right) 
			= - \Delta t \, D_\zeta^n, 
			\label{eqn.gpr.cancel}
		\end{equation}
		since $(\mathbf{A}_p^n \mathbf{u}_p^n) \cdot \mathbf{v} = \mathbf{u}_p^n \cdot (\mathbf{A}_p^n)^T \mathbf{v}$ for any vector $\mathbf{v}$. 
		
		Thus, the contributions of the elastic force and the transport term in the distortion equation cancel exactly, leaving only the dissipative term. The energy balance becomes
		\begin{equation}
			\mathcal{E}^* - \mathcal{E}^n = \frac{1}{2} \rho \sum_c |\omega_c| \, |\delta \mathbf{u}_c|^2 + \Gamma_s^n - \Delta t \left( T^n + D^n \right) - \Delta t \, D_\zeta^n.
			\label{eqn.gpr.balance}
		\end{equation}
		
		\textit{Step 2: bound of the convective term.} From Step~2 of Lemma~\ref{lemma.energy.full}, assumption \eqref{eqn.energy.hyp1} gives $T^n + D^n \geq \frac{1}{2} D^n$.
		
		\textit{Step 3: bound of the explicit Euler gain.} As in Step~3 of the proof of Lemma~\ref{lemma.energy.mhd}, we write $\delta \mathbf{u}_c = \delta \mathbf{u}_c^{\textnormal{conv}} + \frac{\Delta t}{\rho} \mathbf{F}_c^n$.
		Using the bound \eqref{eqn.energy.gain}, we obtain
		\begin{equation}
			\frac{1}{2} \rho \sum_c |\omega_c| \, |\delta \mathbf{u}_c|^2 \leq 4 \, \beta \, \Delta t \, D^n + \frac{\Delta t^2}{\rho} \sum_c |\omega_c| \, |\mathbf{F}_c^n|^2. 
			\label{eqn.gpr.gain}
		\end{equation}
		
		Combining \eqref{eqn.gpr.balance}, Step~2, \eqref{eqn.gpr.gain} and \eqref{eqn.gpr.hyp} we obtain 
		\begin{equation}
			\mathcal{E}^* - \mathcal{E}^n \leq - \Delta t \left( \left( \frac{1}{2} - 4 \, \beta \right) D^n + D_\zeta^n \right) + \frac{\Delta t^2}{\rho} \sum_c |\omega_c| \, |\mathbf{F}_c^n|^2 + \Gamma_s^n \leq - \frac{1}{2} \, \Delta t \left( \left( \frac{1}{2} - 4 \, \beta \right) D^n + D_\zeta^n \right). 
			\label{eqn.gpr.concl}
		\end{equation}
		Finally, the projection step leaves the distortion field unchanged and does not increase the kinetic energy by Lemma~\ref{lemma.energy.pressure}, hence $\mathcal{E}^{n+1} \leq \mathcal{E}^*$, which proves \eqref{eqn.gpr.result}. 
	\end{proof}
	As in the MHD case, condition \eqref{eqn.gpr.hyp} requires the energy gain of the explicit elastic update to be controlled by the Rusanov dissipation and the dissipative term in the distortion equation. By Jensen's inequality, $\frac{1}{\rho} \sum_c |\omega_c| \, |\mathbf{F}_c^n|^2 \leq \frac{1}{\rho} \sum_p |\omega_p| \, |\mathbf{A}_p^n|^2 \, |\nabla_p^c \cdot \boldsymbol{\alpha}_c^n|^2$, so that the contribution of the elastic force is compensated by one half of the dissipative term in \eqref{eqn.gpr.hyp} as soon as $\zeta_p \geq 4 \, \Delta t \, |\mathbf{A}_p^n|^2 / \rho$. Moreover, the left-hand side of \eqref{eqn.gpr.hyp} is of second order in $\Delta t$. Therefore, the condition holds for sufficiently small time steps whenever its right-hand side is positive.

	\subsection{Second-order extension}
	\label{sec.second.order}
	The proposed scheme is extended to second order of accuracy in space and time via the combination of a piecewise linear MUSCL-type reconstruction in space with a two-stage Runge--Kutta time discretization. 
	At time $t^n$, the numerical solution is represented by the cell averages $\mathbf{Q}_c^n$ of the conserved variables over each control volume $\omega_c$, as
	\begin{equation}
		\mathbf{Q}_c^n
		=
		\frac{1}{|\omega_c|}
		\int_{\omega_c}
		\mathbf{Q}(\mathbf{x},t^n)
		\,d\mathbf{x}.
	\end{equation}
	
	In the first-order formulation, they are assumed to be piecewise constant inside each control volume. 
	To achieve second-order spatial accuracy, a piecewise linear reconstruction of the numerical solution is introduced in each cell.
	More precisely, for each control volume $\omega_c$ we define
	\begin{equation}
		\mathbf{w}_h(\mathbf{x},t^n)|_{\omega_c}
		=
		\mathbf{w}_c(\mathbf{x},t^n)
		=
		\mathbf{Q}_c^n
		+
		\nabla \mathbf{Q}_c
		(\mathbf{x}-\mathbf{x}_c),
		\qquad
		\mathbf{x}\in\omega_c,
	\end{equation}
	where $\mathbf{x}_c$ denotes the barycenter of the cell $\omega_c$ and 
	$\nabla\mathbf{Q}_c$ is the reconstructed slope.
	
	The slope is determined by exploiting the cell averages in the neighboring control volumes. 
	The reconstruction stencil associated with the cell $\omega_c$ is the set $\mathcal{N}_c$ of the neighboring cells sharing an edge with $\omega_c$, already introduced in \eqref{star equations}. 
	The reconstructed polynomial is required to reproduce, in a least-squares sense, the known cell averages over the stencil:
	\begin{equation}
		\frac{1}{|\omega_a|}
		\int_{\omega_a}
		\mathbf{w}_h(\mathbf{x},t^n)
		\,d\mathbf{x}
		=
		\mathbf{Q}_a^n,
		\qquad
		\forall\,\omega_a\in\mathcal{N}_c.
		\label{integral conservation for reconstruction}
	\end{equation}
	Since the reconstruction is centered at the barycenter $\mathbf{x}_c$, conservation over the target cell $\omega_c$ is automatically satisfied. 
	Indeed,
	\begin{equation}
		\frac{1}{|\omega_c|}
		\int_{\omega_c}
		(\mathbf{x}-\mathbf{x}_c)\,d\mathbf{x}
		=0,
	\end{equation}
	so that the cell average of $\mathbf{w}_c$ over $\omega_c$ is exactly $\mathbf{Q}_c^n$.
	For the neighboring cells, Eq.~\eqref{integral conservation for reconstruction} can therefore be written as
	\begin{equation}
		\frac{1}{|\omega_a|}
		\int_{\omega_a}
		\nabla\mathbf{Q}_c
		(\mathbf{x}-\mathbf{x}_c)
		\,d\mathbf{x}
		=
		\mathbf{Q}_a^n-\mathbf{Q}_c^n,
		\qquad
		\forall\,\omega_a\in\mathcal{N}_c.
		\label{integral conservation for reconstruction, 2}
	\end{equation}
	Since the resulting system is generally overdetermined, the slope $\nabla\mathbf{Q}_c$ is obtained by solving the associated normal equations in the least-squares sense.
	The resulting slope is initially unlimited. 
	In order to avoid the creation of spurious local extrema in the reconstructed solution, it is subsequently modified by means of the classical Barth--Jespersen slope limiter~\cite{barth1989design}. 
	The final limited reconstruction therefore reads
	\begin{equation}
		\mathbf{w}_c(\mathbf{x},t^n)
		=
		\mathbf{Q}_c^n
		+
		\phi_c\,\nabla\mathbf{Q}_c
		(\mathbf{x}-\mathbf{x}_c),
	\end{equation}
	where $\phi_c\in[0,1]$ denotes the limiter coefficient.
	
	Finally, second-order accuracy in time is achieved by means of a two-stage Runge--Kutta method based on the explicit midpoint rule.
	Denoting by $\mathcal{L}$ the spatial semi-discrete operator, the first stage advances the solution over half a time step,
	\begin{equation}
		\mathbf{Q}_c^{n+\frac12}
		=
		\mathbf{Q}_c^n
		+
		\frac{\Delta t}{2}
		\mathcal{L}_c(\mathbf{Q}^n).
	\end{equation}
	The intermediate solution $\mathbf{Q}^{n+\frac12}$ is then used to
	evaluate the spatial operator in the second stage, yielding
	\begin{equation}
		\mathbf{Q}_c^{n+1}
		=
		\mathbf{Q}_c^n
		+
		\Delta t\,
		\mathcal{L}_c
		\left(\mathbf{Q}^{n+\frac12}\right).
	\end{equation}
	In the semi-implicit scheme, at each Runge--Kutta stage the pressure equation is solved and the pressure correction is applied using the stage time step, i.e.\ $\Delta t/2$ in the first stage and $\Delta t$ in the second.  
	Moreover, the spatial reconstruction is recomputed from the solution available at each stage, so that the numerical fluxes in the second stage are evaluated using the reconstructed intermediate solution.
	The combination of the piecewise linear reconstruction in space and the two-stage midpoint Runge--Kutta discretization in time provides second-order accuracy for smooth solutions of the incompressible {Euler, MHD and GPR} systems; for the weakly compressible equations the implicit pressure step at each stage acts as a backward Euler step on the acoustic terms, so that second order in time holds for the convective terms only.

	\section{Numerical results} 
	\label{sec.results}
	In this section, we validate the new numerical method proposed in this paper using several well-known benchmark problems. 
	
	\subsection{Taylor-Green vortex}
	
	\begin{table}[!tb]
		\centering
		\caption{Numerical convergence results of the \textit{first} order version of the new structure-preserving semi-implicit scheme for the \textit{incompressible} Euler equations applied to the 2D Taylor-Green vortex problem. The $L^2$ error norms refer to the velocity components $u$ and $v$ and the pressure $p$ at the final time $t_f = 0.25$ together with the observed order of decay $\mathcal{O}$.}
		\begin{tabular}{ccccccc}
			$N_x = N_y$ & $L^2(u)$ & $\mathcal{O}(u)$ & $L^2(v)$ & $\mathcal{O}(v)$ & $L^2(p)$ & $\mathcal{O}(p)$ \\
			\hline
			20 & 1.40E-1 & - & 1.40E-1 & - & 3.43E-1 & - \\
			40 & 7.24E-2 & 0.95 & 7.13E-2 & 0.97 & 1.60E-1 & 1.10 \\
			60 & 4.80E-2 & 1.01 & 4.71E-2 & 1.02 & 1.02E-1 & 1.11 \\
			80 & 3.63E-2 & 0.97 & 3.56E-2 & 0.98 & 7.40E-2 & 1.12 \\
			100 & 2.89E-2 & 1.01 & 2.84E-2 & 1.02 & 5.79E-2 & 1.10 \\
			\hline
		\end{tabular}	
		\label{convergence table euler, 1 order}
	\end{table}
	\begin{table}[!tb]
		\centering
		\caption{Numerical convergence results of the \textit{second} order version of the new structure-preserving semi-implicit scheme for the \textit{incompressible} Euler equations applied to the 2D Taylor-Green vortex problem. The $L^2$ error norms refer to the velocity components $u$ and $v$ and the pressure $p$ at the final time $t_f = 0.25$ together with the observed order of decay $\mathcal{O}$.}
		\begin{tabular}{ccccccc}
			$N_x = N_y$ & $L^2(u)$ & $\mathcal{O}(u)$ & $L^2(v)$ & $\mathcal{O}(v)$ & $L^2(p)$ & $\mathcal{O}(p)$ \\
			\hline
			20 & 5.59E-2 & - & 6.14E-2 & - & 5.11E-2 & - \\
			40 & 1.55E-2 & 1.85 & 1.73E-2 & 1.83 & 1.31E-2 & 1.96 \\
			60 & 7.08E-3 & 1.94 & 7.94E-3 & 1.92 & 5.30E-3 & 2.23 \\
			80 & 4.12E-3 & 1.88 & 4.71E-3 & 1.82 & 2.93E-3 & 2.06 \\
			100 & 2.63E-3 & 2.00 & 3.01E-3 & 2.00 & 1.77E-3 & 2.28 \\
			\hline
		\end{tabular}
		\label{convergence table euler, 2 order}
	\end{table}
	\begin{figure}[!tb]
		\centering
		\begin{minipage}{0.45\textwidth}
			\includegraphics[width=0.9847\textwidth,trim=10.4 12 160 188,clip]{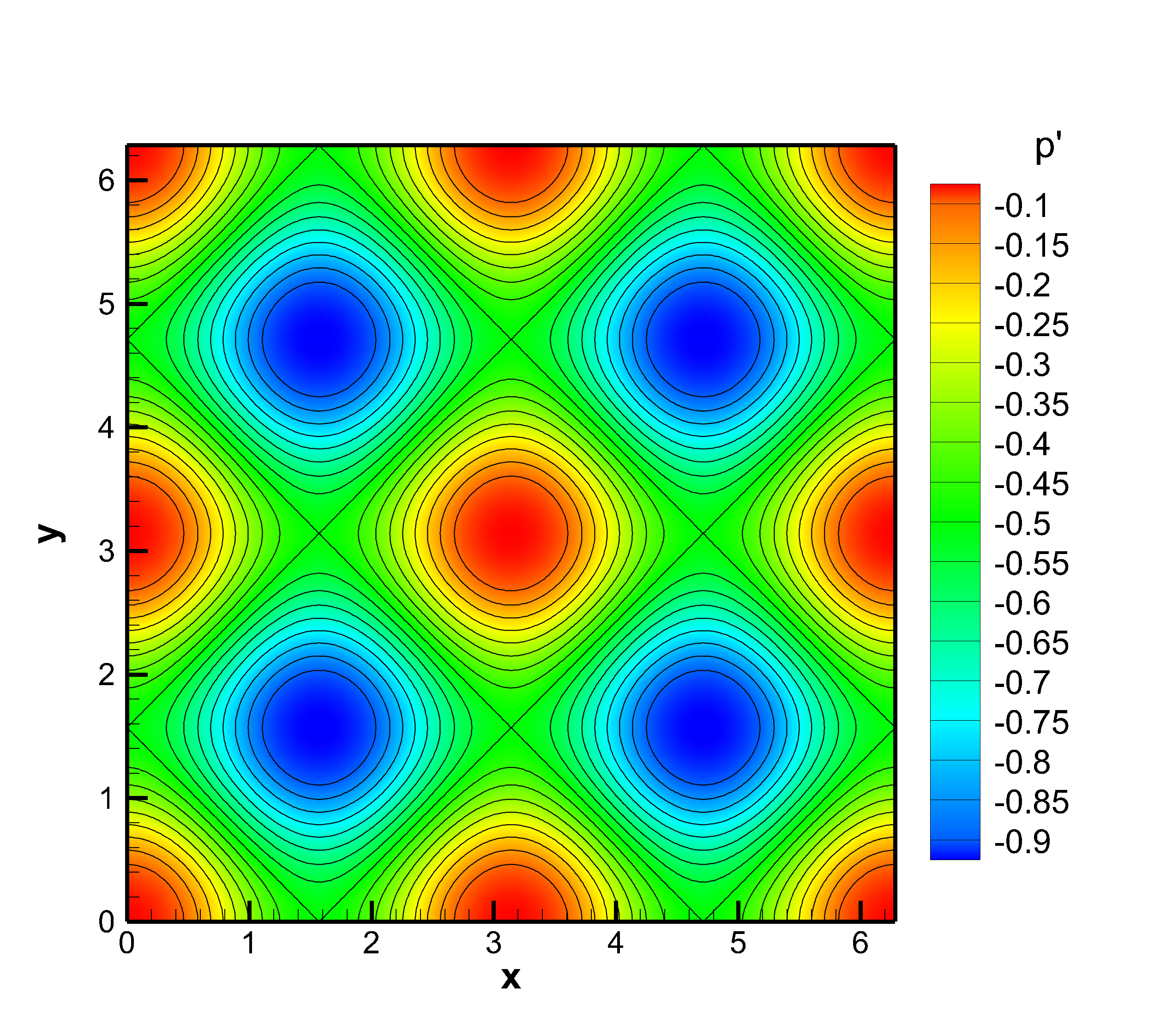}
		\end{minipage}
		\hfill
		\begin{minipage}{0.45\textwidth}
			\centering
			\includegraphics[width=\textwidth,trim=15 12 160 185,clip]{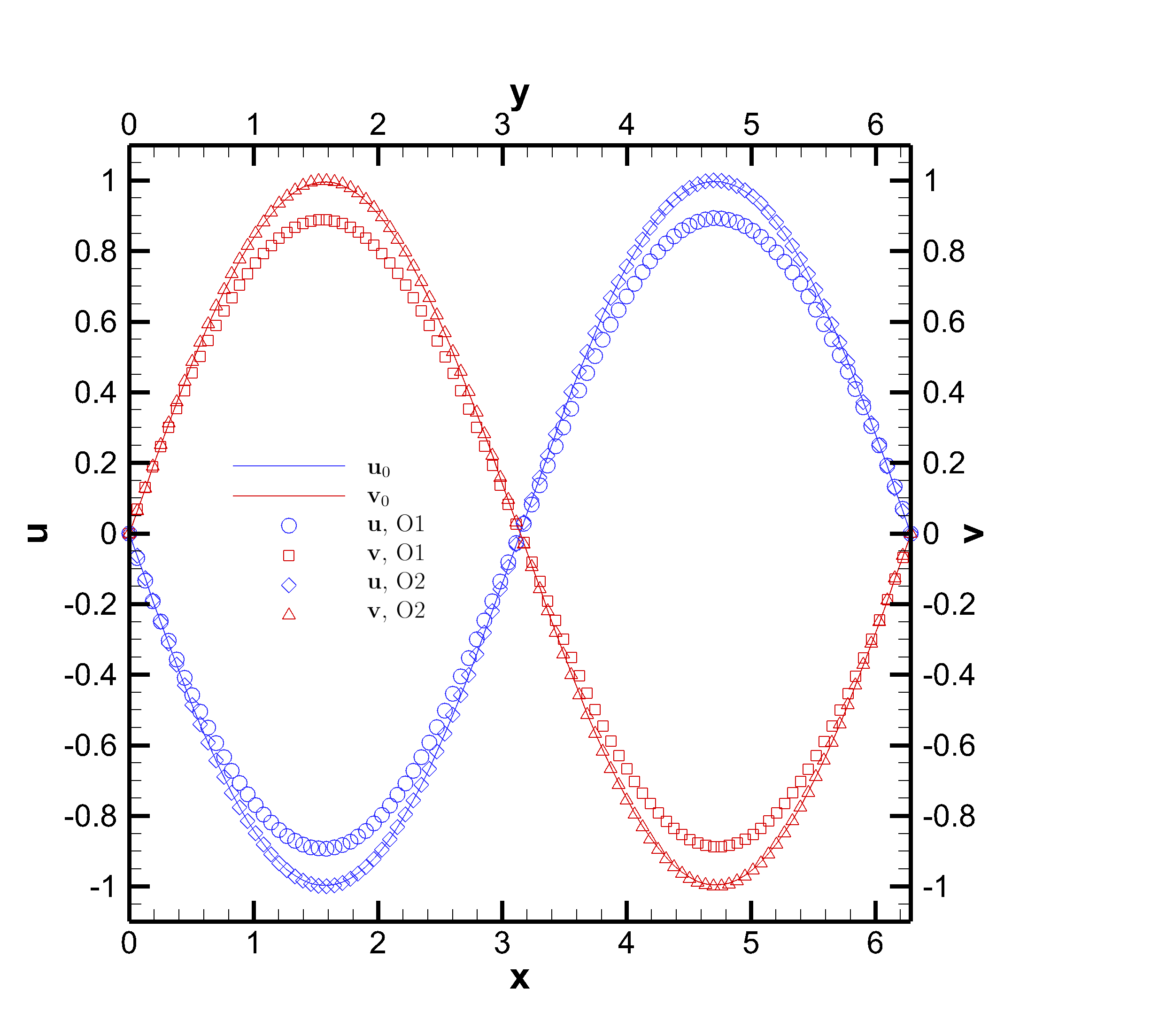}
		\end{minipage}
		
		\begin{minipage}{0.45\textwidth}
			\centering
			\includegraphics[width=\textwidth,trim=15 12 160 185,clip]{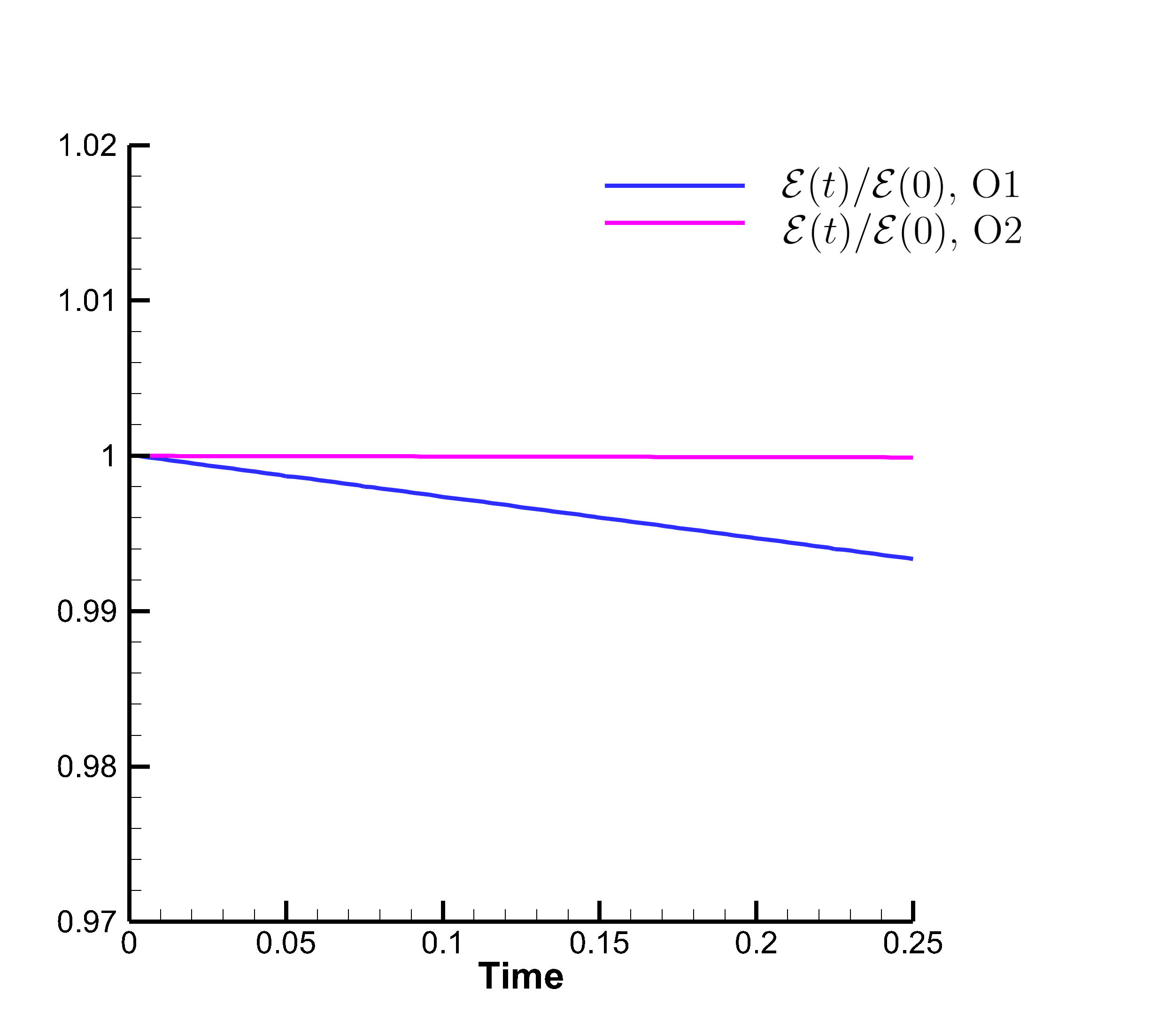}
		\end{minipage}
		\hfill
		\begin{minipage}{0.45\textwidth}
			\includegraphics[width=\textwidth,trim=15 12 160 185,clip]{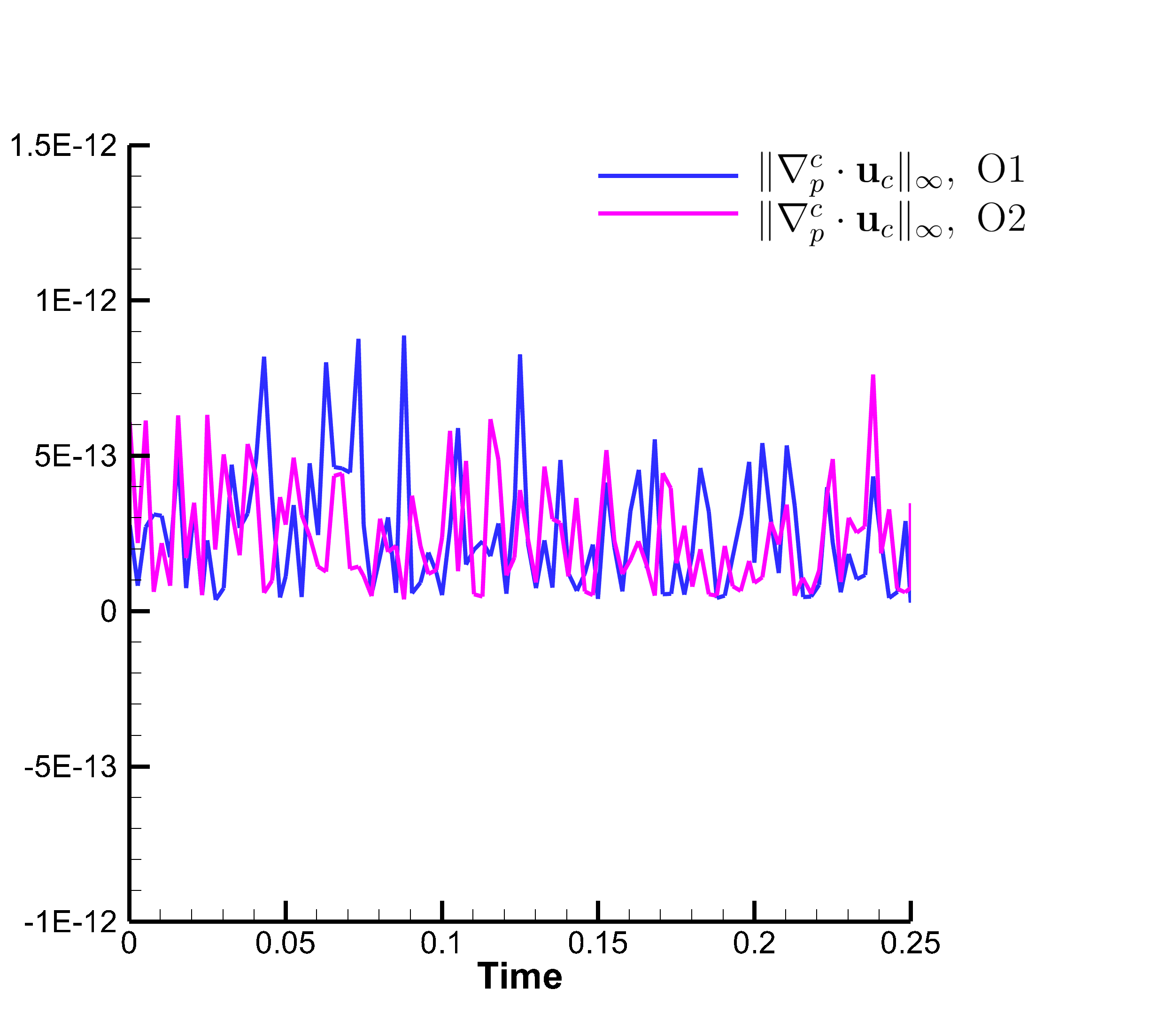}
		\end{minipage}
		
		\caption{Numerical solution of the Taylor-Green vortex at the final time $t_f = 0.25$ obtained using the new structure-preserving semi-implicit scheme on a mesh with \(N_x=N_y=200\). 
			Contours of the pressure deviation $p' = p - p_0$ using a second order reconstruction; colour flood with contour lines at the levels of the colour bar (top left), 
			1D cut along the line $y = \pi$ for the velocity component $u$ and along the line $x = \pi$ for $v$
			and comparison with the exact solution (top right). 
			Time evolution of the kinetic energy $\mathcal{E}$ integrated over $\Omega$   normalized by its initial value (bottom left) and time evolution of the $L^\infty$ norm of the discrete velocity divergence $\nabla_p^c \cdot \mathbf{u}_c$ (bottom right). In the top right, bottom left and bottom right panels, the results are presented for both first- (O1) and second-order scheme (O2).}
		\label{fig:TaylorGreen incompressible Euler}
	\end{figure}
	
	Here we solve the well-known Taylor-Green vortex problem of the incompressible Euler and Navier-Stokes equations, see e.g. \cite{Tavelli2015} for the setup used here. We also solve the problem in the low Mach number limit for the weakly compressible Euler equations. 
	We consider a periodic square domain $\Omega = [0, 2\pi]^2$ and set CFL = 0.9. In the incompressible limit, the exact solution for pressure, density and velocity is given as follows:
	\begin{equation}
		\begin{aligned}
			p(\mathbf{x}, t) &= P_0 + \frac{1}{4}\left(\cos(2x) + \cos(2y) \right) e^{-4\nu t},\\
			\rho(\mathbf{x}) &= 1,
		\end{aligned}
		\qquad
		\begin{aligned}
			u(\mathbf{x}) &= \phantom{-} \sin(x)  \cos(y) e^{-2\nu t},\\
			v(\mathbf{x}) &= -\cos(x) \sin(y) e^{-2\nu t},
		\end{aligned}	
		\qquad	
	\end{equation}
	where $\nu = \mu / \rho$ is the kinematic viscosity of the fluid and the constant $P_0$ can be chosen arbitrarily. For the first series of numerical experiments, carried out in the incompressible limit, we set $P_0 = -\frac{1}{2}$ and impose the exact pressure at one node strongly in the linear algebraic system for the pressure in order to fix the pressure constant inside the numerical scheme. 
	We first solve the problem with $\nu=0$ (inviscid incompressible Euler equations) on a triangular mesh with a number of $N_x = N_y = 200$ points along each boundary edge of the square domain. The simulation is run until a final time of $t_f = 0.25$ to verify that the discrete divergence of the velocity field remains zero up to machine precision for both first and second order of accuracy as shown in Figure~\ref{fig:TaylorGreen incompressible Euler}.
	We then run the same test again but on a sequence of successively refined unstructured triangular meshes using $N_x=N_y=20,40,60,80,$ and 100 in order to verify the orders of convergence, see Table~\ref{convergence table euler, 1 order} and Table~\ref{convergence table euler, 2 order}. Here and in all following convergence tables, the $L^2$ errors are computed from the cell values of the velocity and of the other cell quantities, compared with the exact solution at the barycenters, and from the nodal values of the pressure, compared with the exact solution at the nodes, while the orders of convergence are computed with the mesh size $h \propto 1/N_x$.

	\begin{table}[!tb]
		\centering
		\caption{Taylor-Green vortex for the weakly compressible Euler equations at $t_f = 0.25$: $L^\infty$ norm of the discrete velocity divergence at the interior nodes for increasing isothermal sound speed $c_0$, i.e. decreasing Mach number $M$, for the first- and the second-order scheme on two meshes, together with the observed order of decay $\mathcal{O}_M$ with respect to $M$. The divergence decreases approximately as $M^2$, which provides also clear numerical evidence of the asymptotic-preserving (AP) property of our new structure-preserving scheme.}
		\begin{tabular}{cccccccccc}
			& & \multicolumn{4}{c}{first-order scheme} & \multicolumn{4}{c}{second-order scheme} \\
			\hline 		
			& & \multicolumn{2}{c}{$N_x = N_y = 100$} & \multicolumn{2}{c}{$N_x = N_y = 200$} & \multicolumn{2}{c}{$N_x = N_y = 100$} & \multicolumn{2}{c}{$N_x = N_y = 200$} \\
			\hline 		
			$c_0$ & $M$ & $\| \nabla_p^c \cdot \mathbf{u}_c \|_{\infty}$ & $\mathcal{O}_M$ & $\| \nabla_p^c \cdot \mathbf{u}_c \|_{\infty}$ & $\mathcal{O}_M$ & $\| \nabla_p^c \cdot \mathbf{u}_c \|_{\infty}$ & $\mathcal{O}_M$ & $\| \nabla_p^c \cdot \mathbf{u}_c \|_{\infty}$ & $\mathcal{O}_M$ \\
			\hline
			$10$   & $10^{-1}$ & 2.91E-3\phantom{0}  & --   & 3.42E-3\phantom{0}  & --   & 3.83E-3\phantom{0}  & --   & 4.00E-3\phantom{0}  & --   \\
			$10^2$ & $10^{-2}$ & 3.96E-5\phantom{0}  & 1.87 & 3.90E-5\phantom{0}  & 1.94 & 3.93E-5\phantom{0}  & 1.99 & 3.95E-5\phantom{0}  & 2.01 \\
			$10^3$ & $10^{-3}$ & 3.98E-7\phantom{0}  & 2.00 & 3.93E-7\phantom{0}  & 2.00 & 4.12E-7\phantom{0}  & 1.98 & 3.94E-7\phantom{0}  & 2.00 \\
			$10^4$ & $10^{-4}$ & 4.07E-9\phantom{0}  & 1.99 & 4.56E-9\phantom{0}  & 1.94 & 4.12E-9\phantom{0}  & 2.00 & 4.22E-9\phantom{0}  & 1.97 \\
			$10^5$ & $10^{-5}$ & 3.98E-11 & 2.01 & 4.01E-11 & 2.06 & 4.12E-11 & 2.00 & 3.93E-11 & 2.03 \\
			\hline
		\end{tabular}
		\label{tab:ap}
	\end{table}
	
	In the case of \textit{compressible} Euler equations, the density is computed as $\rho(\mathbf{x}) = p(\mathbf{x})/c_0^2$, where $c_0$ is the sound speed of the fluid.
	In Table~\ref{tab:ap} we assess the asymptotic-preserving (AP) behavior of the scheme by progressively increasing the isothermal sound speed for the first- and the second-order scheme on two different meshes. Since the characteristic velocity of the Taylor-Green vortex is of order one, $c_0 = 10, \ldots, 10^5$ correspond to Mach numbers $M \simeq 10^{-1}, \ldots, 10^{-5}$, respectively. The table also reports the observed order of decay $\mathcal{O}_M$ of the discrete divergence with respect to the Mach number, computed from two subsequent values of $c_0$. The results show that the divergence of the discrete velocity decreases as $M^2$, consistent with the incompressible limit, as already shown in \cite{ boscheri2021structure,zampa2025asymptotic}.
	
	
	
	\begin{table}[!tb]
		\centering	
		\caption{Numerical convergence results of the first order version of the new structure-preserving semi-implicit scheme for the \textit{incompressible} Navier-Stokes equations with \textit{explicit} viscosity applied to the 2D Taylor-Green vortex problem with $\nu = 0.1$. The $L^2$ error norms refer to the velocity components $u$ and $v$ and the pressure $p$ at the final time $t_f = 0.25$ together with the observed order of decay~$\mathcal{O}$.}
		\begin{tabular}{ccccccc}
			$N_x = N_y$ & $L^2(u)$ & $\mathcal{O}(u)$ & $L^2(v)$ & $\mathcal{O}(v)$ & $L^2(p)$ & $\mathcal{O}(p)$ \\
			\hline
			20  & 1.281E-1 & -    & 1.286E-1 & -    & 3.065E-1 & -    \\
			40  & 6.657E-2 & 0.94 & 6.594E-2 & 0.96 & 1.429E-1 & 1.10 \\
			60  & 4.416E-2 & 1.01 & 4.363E-2 & 1.02 & 9.156E-2 & 1.10 \\
			80  & 3.339E-2 & 0.97 & 3.297E-2 & 0.97 & 6.682E-2 & 1.10 \\
			100 & 2.664E-2 & 1.01 & 2.629E-2 & 1.01 & 5.254E-2 & 1.08 \\
			\hline
		\end{tabular}
		\label{convergence table Navier-Stokes with explicit viscosity}
	\end{table}
	
	\begin{table}[!tb]
		\centering
		\caption{Numerical convergence results of the first order version of the new structure-preserving semi-implicit scheme for the \textit{incompressible} Navier-Stokes equations with \textit{implicit} viscosity applied to the 2D Taylor-Green vortex problem with $\nu = 0.1$. The $L^2$ error norms refer to the velocity components $u$ and $v$ and the pressure $p$ at the final time $t_f = 0.25$ together with the observed order of decay~$\mathcal{O}$.}	
		\begin{tabular}{ccccccc}
			$N_x = N_y$ & $L^2(u)$ & $\mathcal{O}(u)$ & $L^2(v)$ & $\mathcal{O}(v)$ & $L^2(p)$ & $\mathcal{O}(p)$ \\
			\hline
			20  & 1.278E-1 & -    & 1.283E-1 & -    & 3.088E-1 & -    \\
			40  & 6.637E-2 & 0.95 & 6.578E-2 & 0.96 & 1.455E-1 & 1.09 \\
			60  & 4.404E-2 & 1.01 & 4.353E-2 & 1.02 & 9.383E-2 & 1.08 \\
			80  & 3.331E-2 & 0.97 & 3.290E-2 & 0.97 & 6.843E-2 & 1.10 \\
			100 & 2.658E-2 & 1.01 & 2.623E-2 & 1.01 & 5.378E-2 & 1.08 \\
			\hline
		\end{tabular}
		\label{convergence table Navier-Stokes with implicit viscosity}
	\end{table}
	
	\begin{figure}[!tb]
		\centering
		\begin{minipage}{0.31\textwidth}
			\centering
			\includegraphics[width=\textwidth,trim=15 12 100 185,clip]{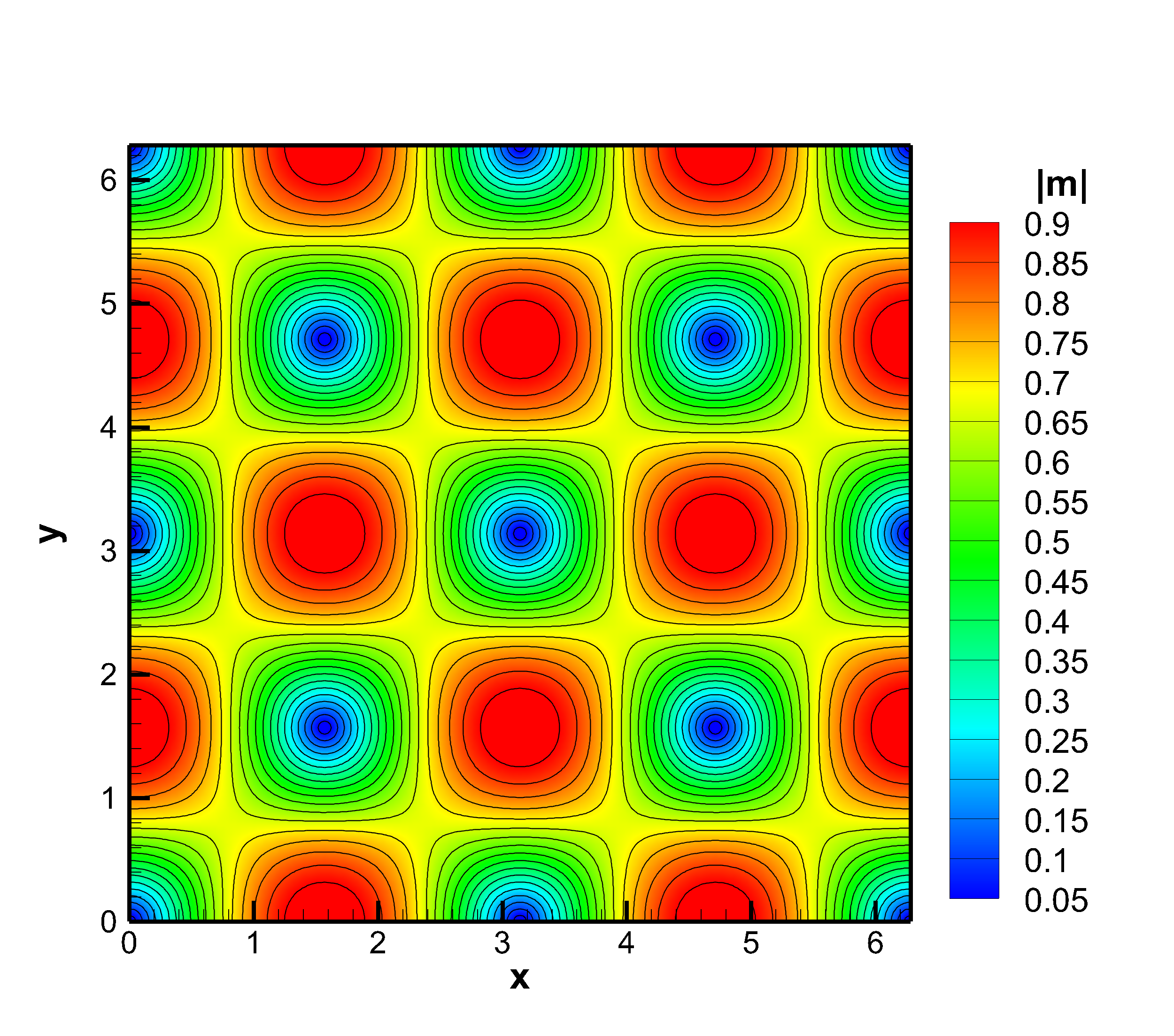}
		\end{minipage}
		\hfill
		\begin{minipage}{0.31\textwidth}
			\centering
			\includegraphics[width=\textwidth,trim=15 12 100 185,clip]{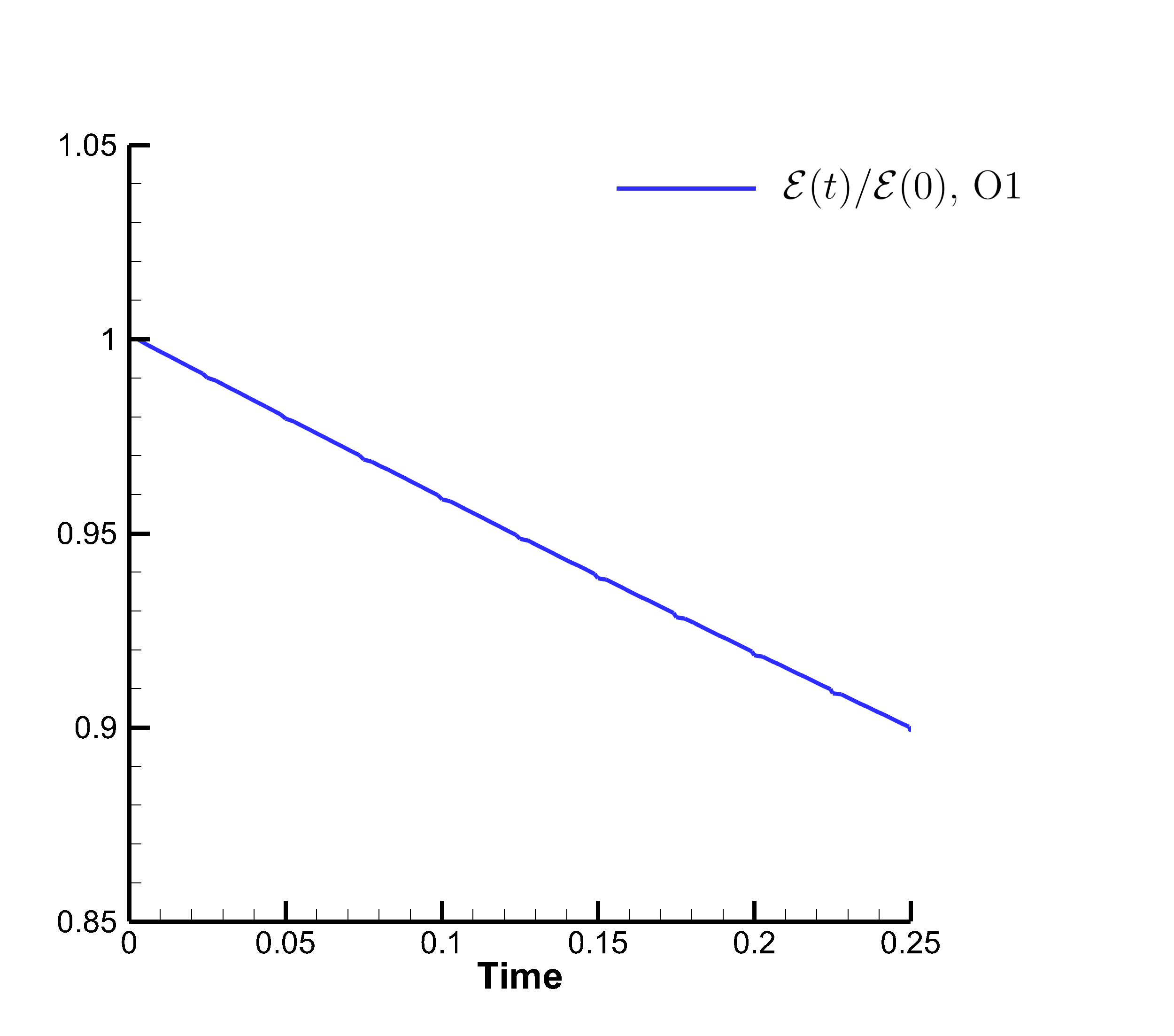}
		\end{minipage}
		\hfill
		\begin{minipage}{0.31\textwidth}
			\centering
			\includegraphics[width=\textwidth,trim=15 12 100 185,clip]{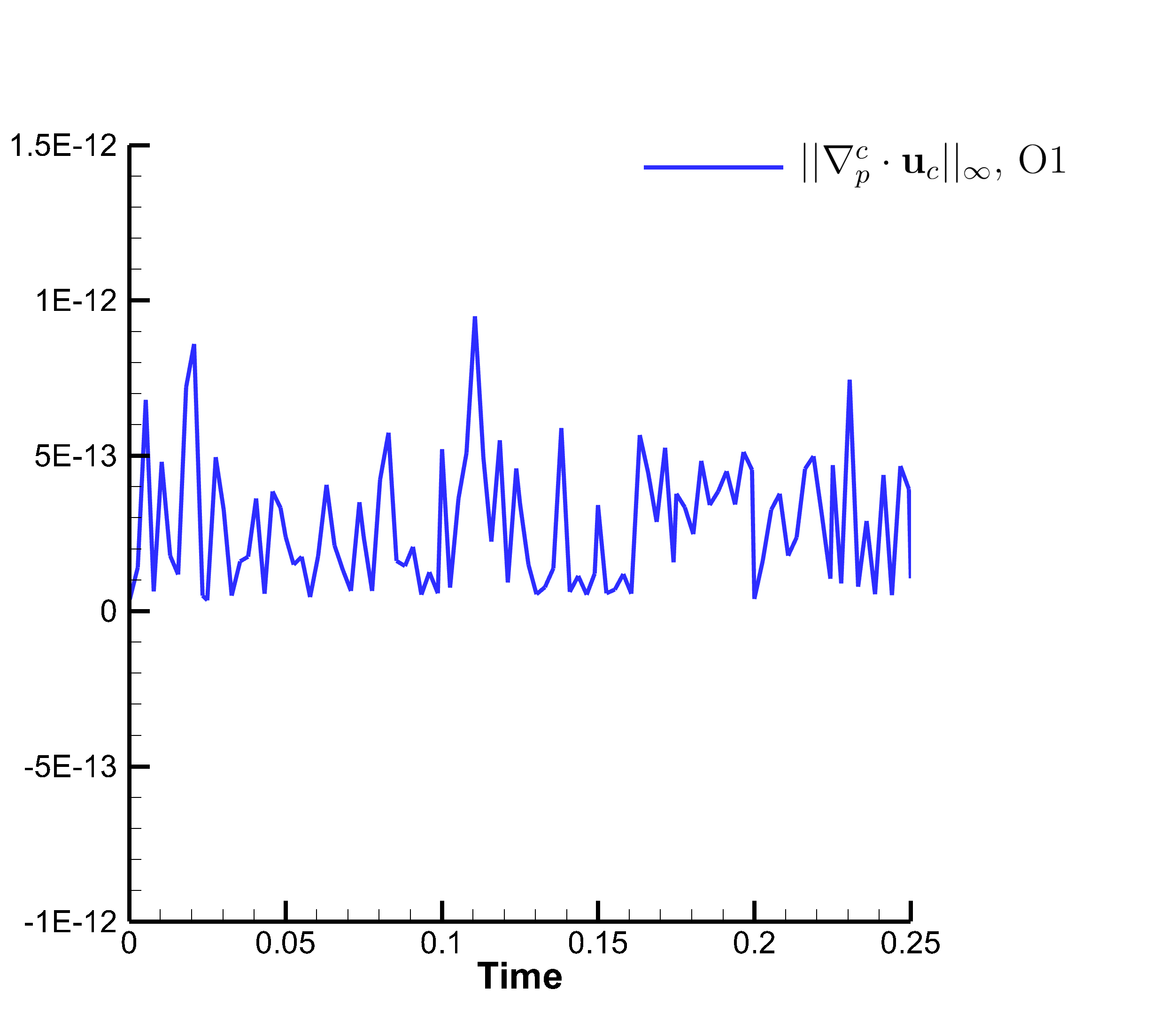}
		\end{minipage}	
		\caption{Numerical solution of the Taylor-Green vortex at the final time $t_f = 0.25$ obtained using the new structure-preserving semi-implicit scheme for the incompressible Navier-Stokes equations with implicit viscosity (first order scheme, $\nu = 0.1$) on a mesh with $N_x = N_y = 200$. 
			Magnitude of the momentum $|\mathbf{m}|$; colour flood with contour lines at the levels of the colour bar (left), time evolution of the kinetic energy $\mathcal{E}$ normalized by its initial value (center) and time evolution of the $L^\infty$ norm of the discrete velocity divergence $\nabla_p^c \cdot \mathbf{u}_c$.}
		\label{fig:TaylorGreen incompressible Navier-Stokes}
	\end{figure}
	In the case of the incompressible Navier-Stokes equations, we treat the viscous term both explicitly and implicitly. 
	As before, we run the problem on a triangular mesh with a number of $N_x = N_y = 200$ points along each boundary edge of the periodic square domain. The simulation is run with $\nu=0.1$ and $\rho=1$ until $t_f = 0.25$ to verify that the discrete divergence of the velocity remains zero up to machine precision as shown in Figure \ref{fig:TaylorGreen incompressible Navier-Stokes}. 
	We then run the same test again on a sequence of successively refined unstructured triangular meshes using \(N_x=N_y=20,40,60,80,\) and 100 and $\nu = 0.1$ in order to verify the order of convergence, see Tables \ref{convergence table Navier-Stokes with explicit viscosity} and \ref{convergence table Navier-Stokes with implicit viscosity}.

	\begin{figure}[!b]
		\centering
		\begin{minipage}{0.48\textwidth}
			\centering
			\includegraphics[width=\textwidth,trim=15 12 200 185,clip]{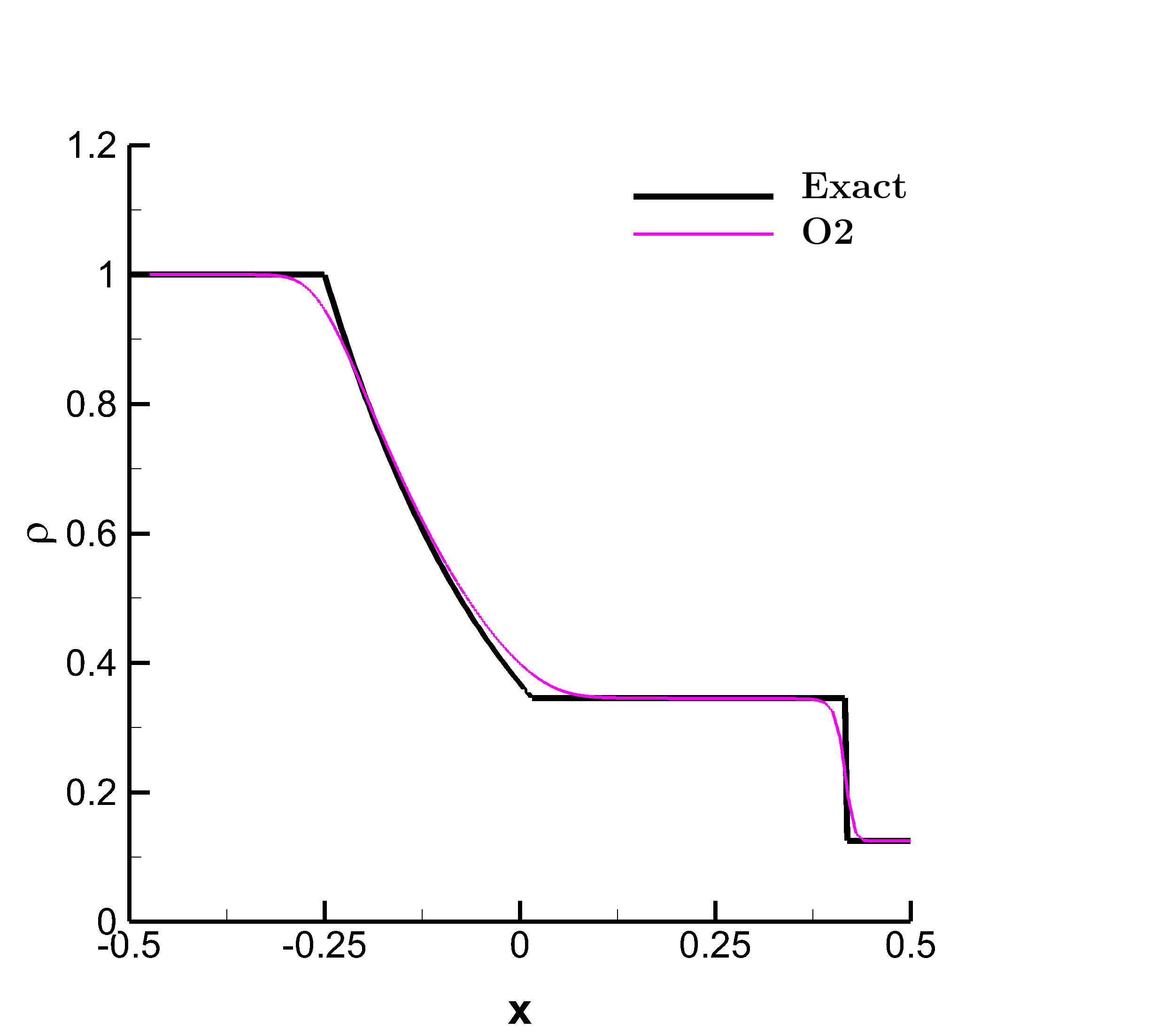}
		\end{minipage}
		\hfill
		\begin{minipage}{0.48\textwidth}
			\centering
			\includegraphics[width=\textwidth,trim=15 12 200 185,clip]{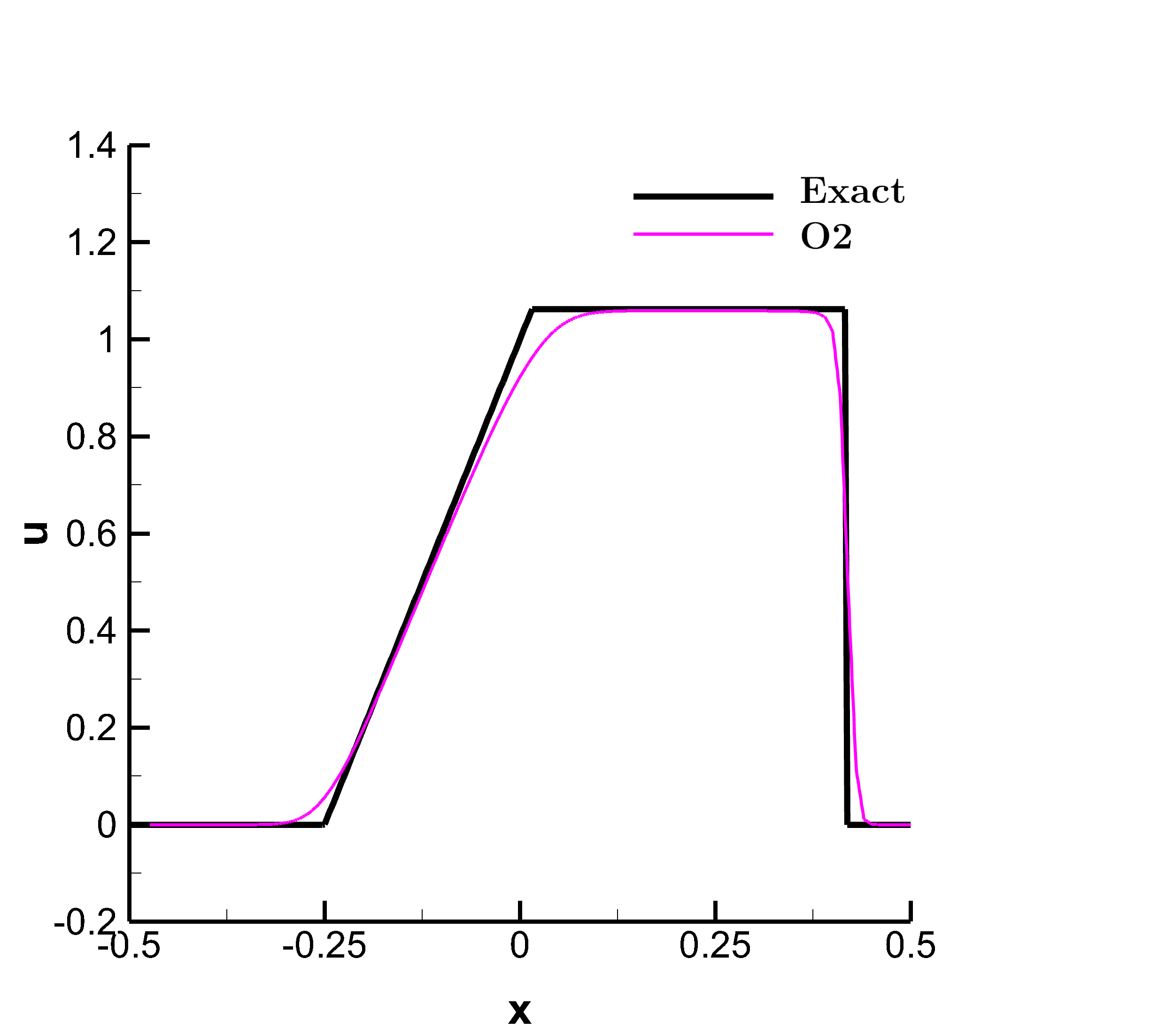}
		\end{minipage}
		\caption{Riemann problem for the isothermal Euler equations. 1D cut along the line $y = 0$ of the density (left) and of the velocity (right) at $t_f = 0.25$ obtained with the second-order scheme (O2), compared with the exact solution of the Riemann problem.}
		\label{Sod test case}
	\end{figure}
	
	\subsection{Riemann problem}
	
	{ For the Sod-type shock tube problem considered here, the computational domain is the rectangle $\Omega = [-0.5, 0.5] \times [-0.05, 0.05]$, with periodic boundary conditions in the $y$ direction and Dirichlet boundaries in the $x$ direction, discretized with a regular triangular mesh with 100 intervals in the $x$ direction and composed of 2412 nearly equilateral triangles. 
		The simulation is performed with the second-order scheme with CFL~=~0.9 until the final time $t_f = 0.25$, which takes 220 time steps. The initial condition reads
		\begin{equation}
			(\rho, \mathbf{u}, p) =
			\begin{cases}
				(1, 0, 0, \rho \, c_0^2) & \text{if } x \leq 0, \\[6pt]
				(0.125, 0, 0, \rho \, c_0^2 ) & \text{if } x > 0,
			\end{cases}
		\end{equation}
		where $c_0 = 1$ is the isothermal sound speed of the fluid. The exact solution of this Riemann problem consists of a left-moving rarefaction wave and a right-moving shock wave, connecting the initial high-density/high-pressure state on the left with the low-density/low-pressure state on the right; there is no contact discontinuity in the density, since in the isothermal case the pressure is a function of the density alone. The aim of this test is to assess the robustness of the proposed semi-implicit scheme in the presence of discontinuous solutions. Figure~\ref{Sod test case} compares the density and the velocity along the line $y = 0$ with the exact solution: the scheme captures the shock and the rarefaction at the correct positions, the shock is resolved within a few cells, while the kinks at the head and at the tail of the rarefaction are smoothed by the limiter. Overall the numerical solution is in good agreement with the exact one and no spurious oscillations are visible. 
	}

	\subsection{Gresho vortex}
	\label{sec.gresho}
	\begin{figure}[!tb]
		\centering
		\begin{minipage}{0.48\textwidth}
			\centering
			\includegraphics[width=\textwidth,trim=15 12 160 185,clip]{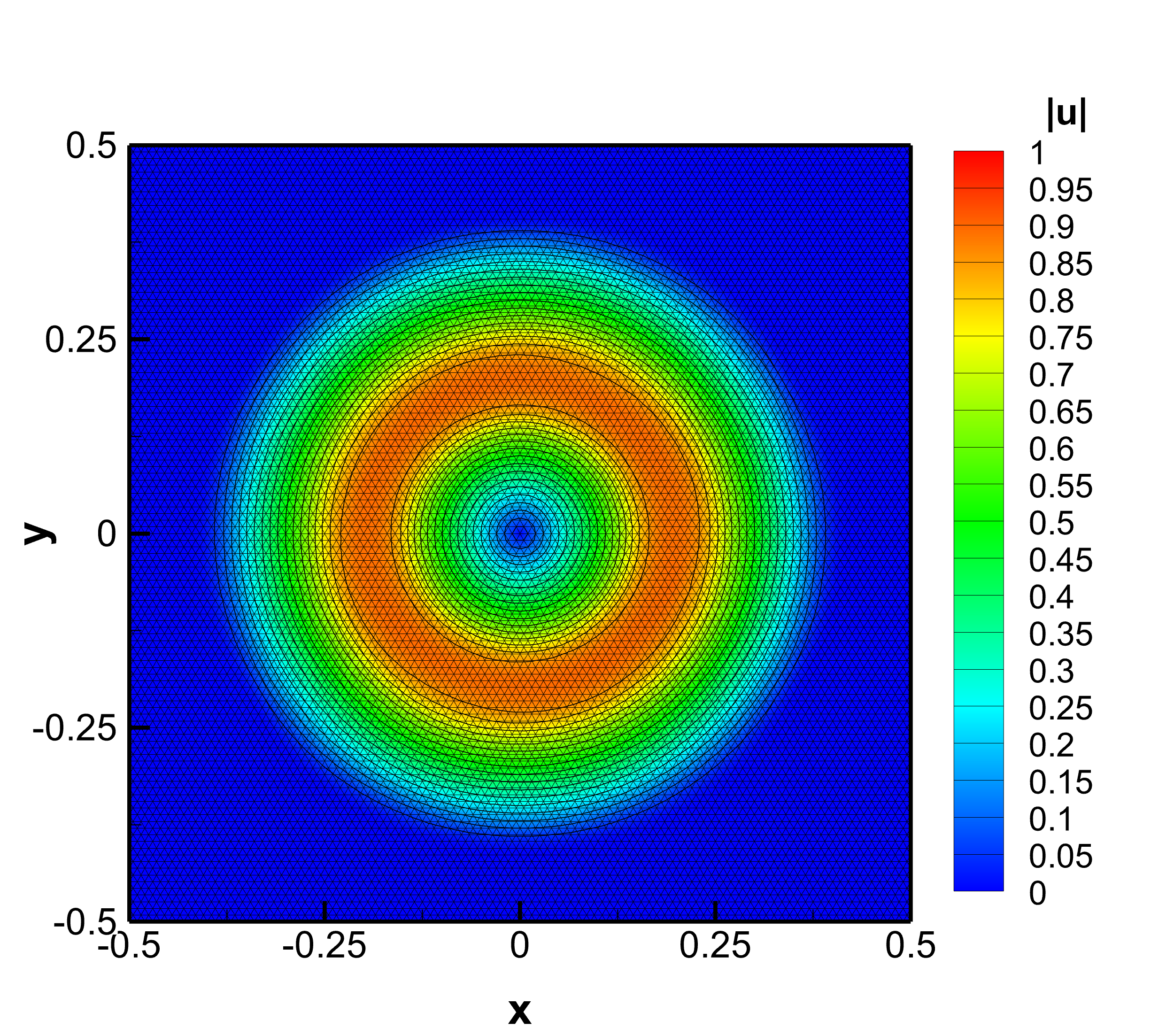}
		\end{minipage}
		\hfill
		\begin{minipage}{0.48\textwidth}
			\centering
			\includegraphics[width=\textwidth,trim=15 12 160 185,clip]{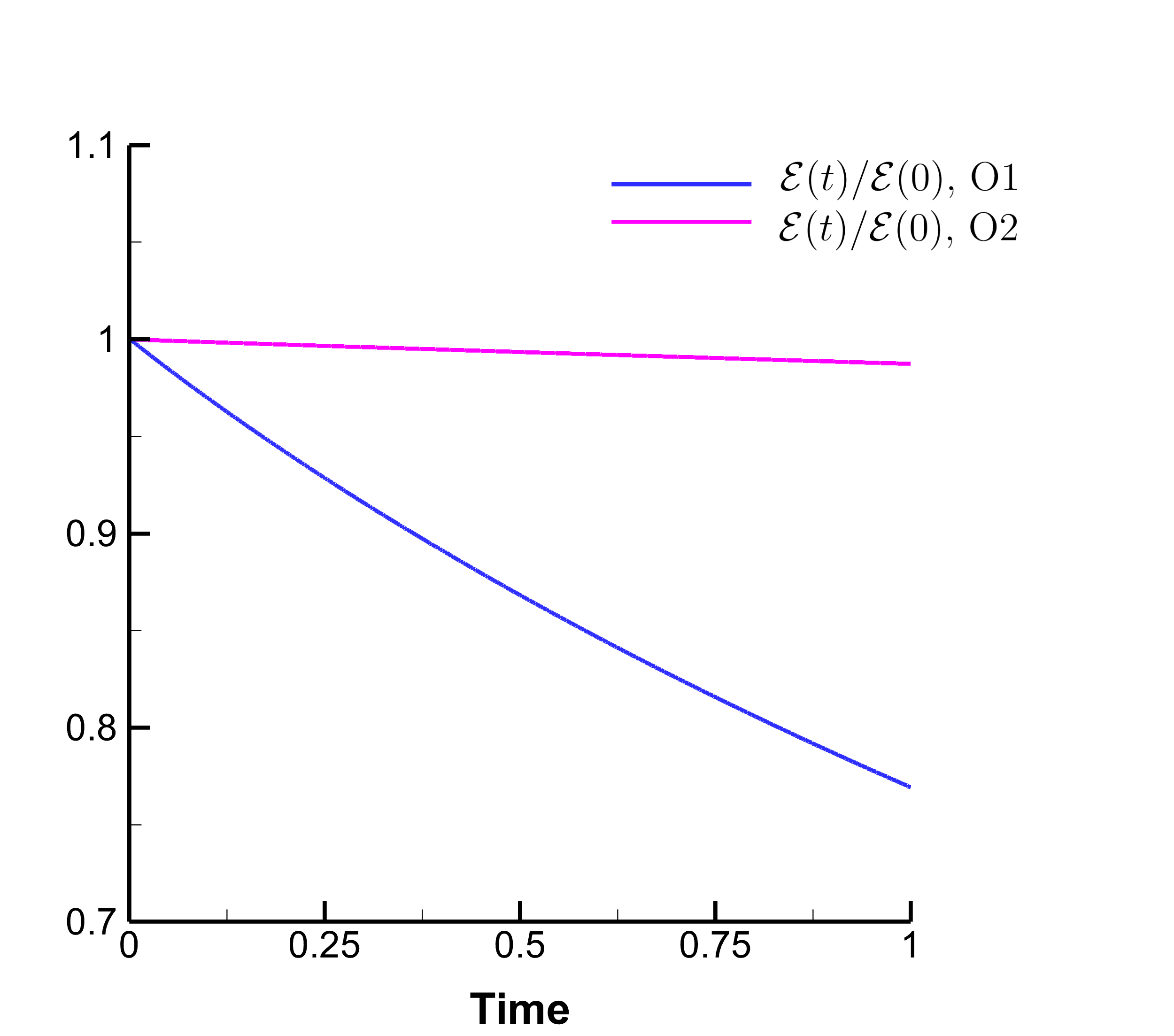}
		\end{minipage}
		\\[2mm]
		\begin{minipage}{0.48\textwidth}
			\centering
			\includegraphics[width=\textwidth,trim=15 12 160 185,clip]{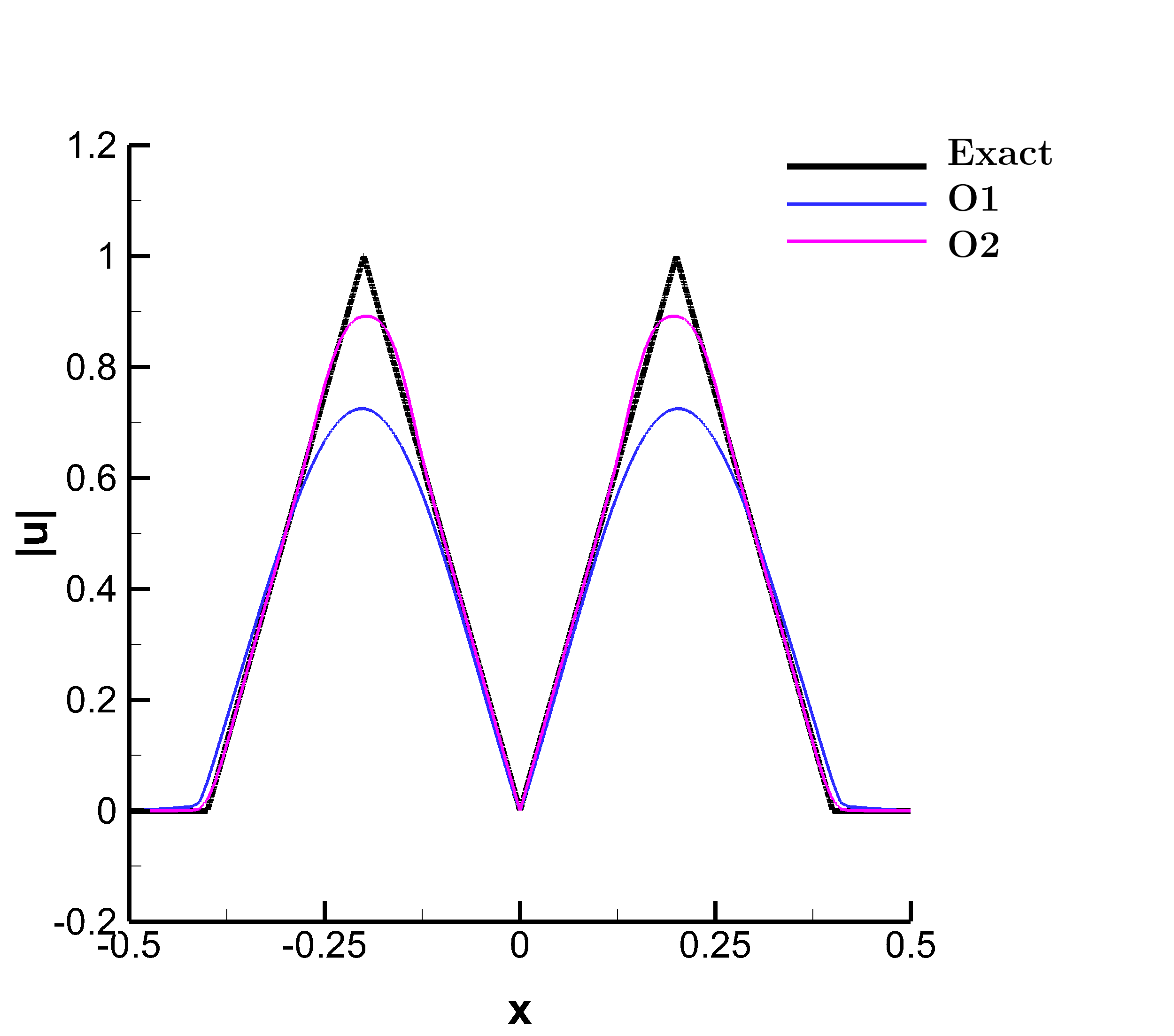}
		\end{minipage}
		\hfill
		\begin{minipage}{0.48\textwidth}
			\centering
			\includegraphics[width=\textwidth,trim=15 12 160 185,clip]{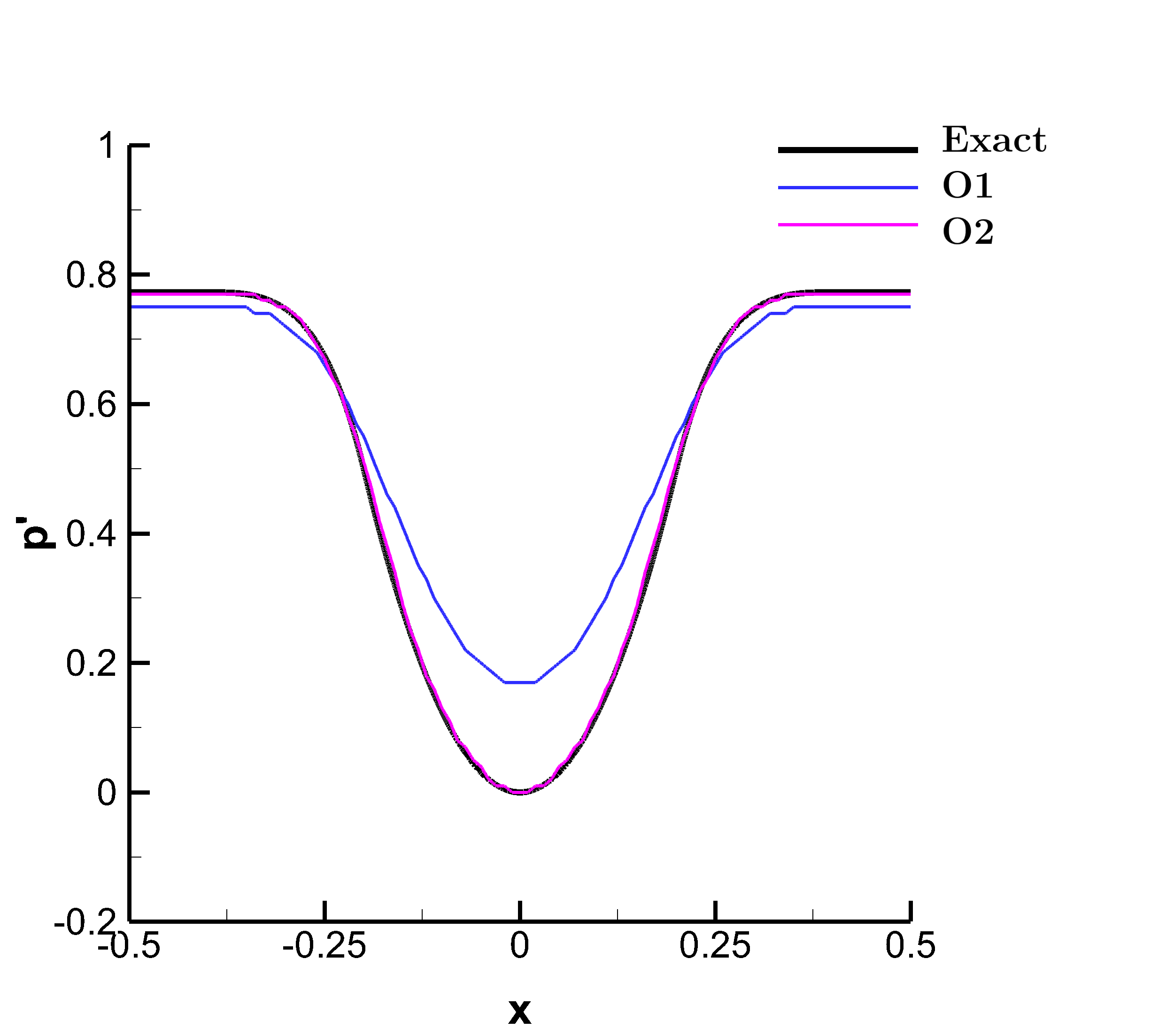}
		\end{minipage}
		\caption{Gresho vortex at $M = 10^{-5}$ on the mesh with $N_x = N_y = 100$ at the final time $t_f = 1$. Top left: velocity magnitude $|\mathbf{u}|$ of the second-order scheme (O2); colour flood with contour lines at the levels of the colour bar. Top right: time evolution of the total kinetic energy, normalized by its initial value, for the first- (O1) and the second-order scheme (O2). Bottom: cuts along the $x$ axis ($y = 0$) of the velocity magnitude (left) and of the pressure deviation $p' = p - p_0$ (right) at $t_f$ for O1 and O2, compared with the exact solution.}
		\label{fig:Gresho vortex}
	\end{figure}
	{ The Gresho vortex \cite{gresho1990theory,liska2003comparison} is a stationary solution of the incompressible Euler equations in which the centrifugal force of a rotating flow is balanced exactly by the pressure gradient. Since the exact solution does not change in time, any deviation of the numerical solution from the initial condition is caused by the numerical dissipation and dispersion of the scheme, which makes this test a sensitive measure of the quality of a method for vortical flows. We employ this benchmark to assess the ability of our weakly compressible scheme of Section~\ref{sec.wce} to preserve vortical structures in the low-Mach regime.
		We solve the problem at a low Mach number of $M = 10^{-5}$, i.e. with the isothermal sound speed $c_0 = 10^5$ and the background pressure $p_0 = \rho_0 c_0^2$ for $\rho_0 = 1$, where the constant $p_0$ is split off from the pressure system as described in Section~\ref{sec.wce}. The computational domain is the periodic square $\Omega = [-0.5, 0.5]^2$, discretized with a regular triangular mesh with $N_x = N_y = 100$ and 23316 nearly equilateral triangles.
		The initial condition is the vortex
		\begin{equation}
			\mathbf{u}(\mathbf{x}) = \mathbf{e}_{\phi} \,
			\begin{cases}
				5r & \text{if } r < 0.2, \\[6pt]
				2-5r & \text{if } 0.2 \le r < 0.4, \\[6pt]
				0 & \text{if } r \ge 0.4,
			\end{cases}
			\qquad
			p'(\mathbf{x}) =
			\begin{cases}
				\frac{25}{2} r^2 & \text{if } r < 0.2, \\[6pt]
				\frac{25}{2} r^2 + 4(1-5r) + 4 \log(5r) & \text{if } 0.2 \le r < 0.4, \\[6pt]
				4 \log 2 - 2 & \text{if } r \ge 0.4,
			\end{cases}
			\label{eqn.gresho.ic}
		\end{equation}
		with $r = \sqrt{x^2+y^2}$, the azimuthal unit vector $\mathbf{e}_{\phi}$ of polar coordinates and the pressure $p = p_0 + p'$. The density is initialized with the constant value $\rho_0$, which differs from the density $\rho(p)$ given by the equation of state only by $p'/c_0^2 = \mathcal{O}(M^2)$, hence the initial data are well prepared. The velocity profile is continuous but has kinks at $r = 0.2$ and $r = 0.4$, and the pressure balances the centrifugal force so that the vortex does not evolve in time.
		
		We run the first-order scheme (O1) and the second-order scheme (O2) with CFL = 0.9 until the final time $t_f = 1$, which takes 659 and 735 time steps, respectively; the time step is restricted only by the explicit discretization of the convective terms and does not depend on $c_0$. Figure~\ref{fig:Gresho vortex} summarizes the results. The velocity magnitude obtained with the second-order scheme at $t_f$, plotted together with the mesh (top left), keeps the circular shape of the vortex and the maximum at $r = 0.2$. The time evolution of the total kinetic energy $\mathcal{E}(t) = \int_\Omega \frac{1}{2} \rho \left| \mathbf{u} \right|^2 \, d\mathbf{x}$, normalized by its initial value (top right), quantifies the numerical dissipation.
		The cuts along the $x$ axis (bottom) compare the velocity magnitude and the pressure deviation of both schemes with the exact solution. The first order scheme smears the velocity and pressure profiles visibly, whereas the second order scheme stays much closer to the exact profiles. The density deviation from $\rho_0$ remains below $7.73 \cdot 10^{-11}$ during the entire simulation, consistent with the expected second order scaling in terms of the Mach number.
	}
	
	\subsection{Low Mach number flow around a circular cylinder}
	\label{sec.cylinder}
	
	{ Next we consider the inviscid weakly compressible flow around a circular cylinder in the low Mach number regime. This test allows us to verify the asymptotic preserving property of our new scheme in the weakly compressible regime in the presence of wall boundaries. A circular cylinder of radius $R = 1$ is centered at the origin of the computational domain $\Omega = [-10, 10]^2$. The computational grid is a boundary-fitted triangular mesh with $160$ nodes on the cylinder wall and $60$ rings of elements whose size grows geometrically from the wall to the outer boundary, see the left panel of Figure~\ref{fig.cylinder}. As initial condition we take the incompressible potential flow around the cylinder with free stream velocity field $\mathbf{u}_\infty = (u_\infty, 0)$ and $u_\infty = 1$,
		\begin{eqnarray}
			u(\mathbf{x}) &=& u_\infty \left( 1 - R^2 \, \frac{x^2 - y^2}{r(\mathbf{x})^4} \right), \qquad
			v(\mathbf{x}) \,=\, - 2 u_\infty R^2 \, \frac{x y}{r(\mathbf{x})^4}, \nonumber \\
			p(\mathbf{x}) &=& p_0 + \frac{1}{2} \rho_\infty \left( u_\infty^2 - u(\mathbf{x})^2 - v(\mathbf{x})^2 \right), \qquad
			\rho(\mathbf{x}) \,=\, \frac{p(\mathbf{x})}{c_0^2},
			\label{eqn.cylinder.ic}
		\end{eqnarray}
		with $r(\mathbf{x})^2 = x^2 + y^2$, the constant free stream density $\rho_\infty = 1$ and the background pressure $p_0 = \rho_\infty c_0^2$. The Mach number of the flow is $M = u_\infty / c_0$. On the cylinder wall the exterior state of the Rusanov flux \eqref{eqn.rusanov} is the mirror image of the interior state, i.e. the normal velocity component changes its sign while the tangential velocity and the density are copied, and the nodal pressure on the wall is not prescribed, so that the natural no-flux condition of the pressure system \eqref{pressure correction for the compressible case} applies there. On the outer boundary of the domain the potential flow \eqref{eqn.cylinder.ic} is prescribed both as exterior state of the Rusanov flux and as Dirichlet condition for the nodal pressure. As explained in Section~\ref{sec.wce}, the pressure system is solved for the deviation $p'_p = p_p - p_0$ from the constant background pressure, which is also used in the discrete pressure gradient of the momentum update \eqref{eqn.mom}, so that the large constant $p_0$ does not affect the velocity field at very small Mach numbers. For details the reader is referred to the well-known multiple pressure variables (MPV) method of Klein and Munz and collaborators, see \cite{klein1995semi,munz2003extension,park2005multiple}. 
		
		We run the simulations with CFL = 0.9 until the final time $t_f = 2$ for the Mach numbers $M = 10^{-1}, 10^{-2}, 10^{-3}, 10^{-4}$ and $10^{-5}$, i.e. for $c_0 = 10, \ldots, 10^5$, with the first- and second-order version of our scheme. Table~\ref{tab.cylinder} reports the $L^\infty$ norms of the discrete divergence of the velocity field at the interior nodes and of the density deviation $\rho_p - \rho_\infty$ at the final time, together with the observed order of decay with respect to the Mach number computed from two subsequent values of $M$. Both quantities decrease proportionally to $M^2$, which clearly confirms again the asymptotic preserving property of the method also in the presence of a wall boundary. Figure~\ref{fig.cylinder} shows the pressure deviation $p' = p - p_0$ together with the mesh and the velocity component $u$ obtained with the second-order scheme at $M = 10^{-2}$ at time $t = 20$, when the flow has reached a steady state. The symmetry of the potential flow is properly reproduced up to the effect of the numerical dissipation of the explicit discretization of the convective terms. The second row of Figure~\ref{fig.cylinder} compares the pressure deviation $p'$ and the velocity magnitude $|\mathbf{u}|$ along the circle $r = R + c$ with $c = 0.1$, obtained by linear interpolation of the nodal values, with the potential flow \eqref{eqn.cylinder.ic}.
	}
	
	\begin{table}[!tb]
		\centering
		\caption{Low Mach number flow around a circular cylinder at $t_f = 2$: $L^\infty$ norm of the discrete velocity divergence $\nabla_p^c \cdot \mathbf{u}_c$ at the interior nodes and of the density deviation $\rho_p - \rho_\infty$ for decreasing Mach number $M$, for the first- and the second-order scheme, together with the observed order of decay $\mathcal{O}_M$ with respect to $M$. Both quantities decrease as $M^2$.}
		\setlength{\tabcolsep}{4pt}
		\begin{tabular}{ccccccccc}
			& \multicolumn{4}{c}{first order scheme} & \multicolumn{4}{c}{second order scheme} \\
			\hline
			$M$ & $\| \nabla_p^c \cdot \mathbf{u}_c \|_\infty$ & $\mathcal{O}_M$ & $\| \rho_p - \rho_\infty \|_\infty$ & $\mathcal{O}_M$ & $\| \nabla_p^c \cdot \mathbf{u}_c \|_\infty$ & $\mathcal{O}_M$ & $\| \rho_p - \rho_\infty \|_\infty$ & $\mathcal{O}_M$ \\
			\hline
			$10^{-1}$ & 2.82E-2\phantom{0}  & --   & 1.44E-2\phantom{0}  & --   & 3.15E-2\phantom{0}  & --   & 1.44E-2\phantom{0}  & --   \\
			$10^{-2}$ & 2.80E-4\phantom{0}  & 2.00 & 1.44E-4\phantom{0}  & 2.00 & 3.78E-4\phantom{0} & 1.92 & 1.50E-4\phantom{0}  & 2.00 \\
			$10^{-3}$ & 2.96E-6\phantom{0}  & 1.98 & 1.44E-6\phantom{0}  & 2.00 & 6.18E-6\phantom{0}  & 1.79 & 1.50E-6\phantom{0}  & 2.00 \\
			$10^{-4}$ & 2.95E-8\phantom{0}  & 2.00 & 1.44E-8\phantom{0}  & 2.00 & 6.21E-8\phantom{0}  & 2.00 & 1.49E-8\phantom{0}  & 2.00 \\
			$10^{-5}$ & 2.96E-10 & 2.00 & 1.44E-10 & 2.00 & 6.20E-10 & 2.00 & 1.50E-10 & 2.00 \\
			\hline
		\end{tabular}
		\label{tab.cylinder}
	\end{table}
	
	\begin{figure}[!tb]
		\centering
		\begin{minipage}{0.48\textwidth}
			\centering
			\includegraphics[width=\textwidth,trim=15 12 200 185,clip]{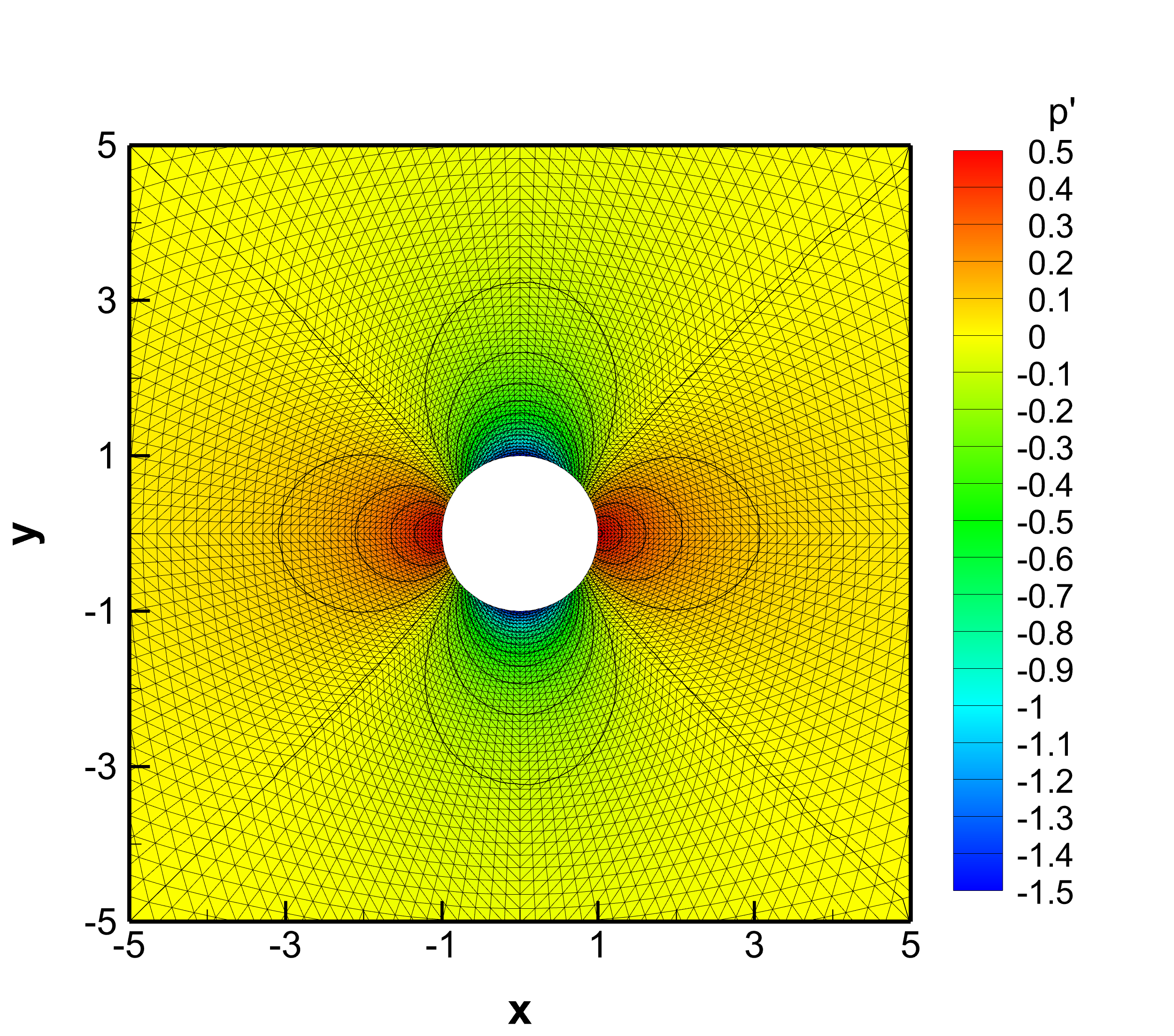}
		\end{minipage}
		\hfill
		\begin{minipage}{0.48\textwidth}
			\centering
			\includegraphics[width=\textwidth,trim=15 12 200 185,clip]{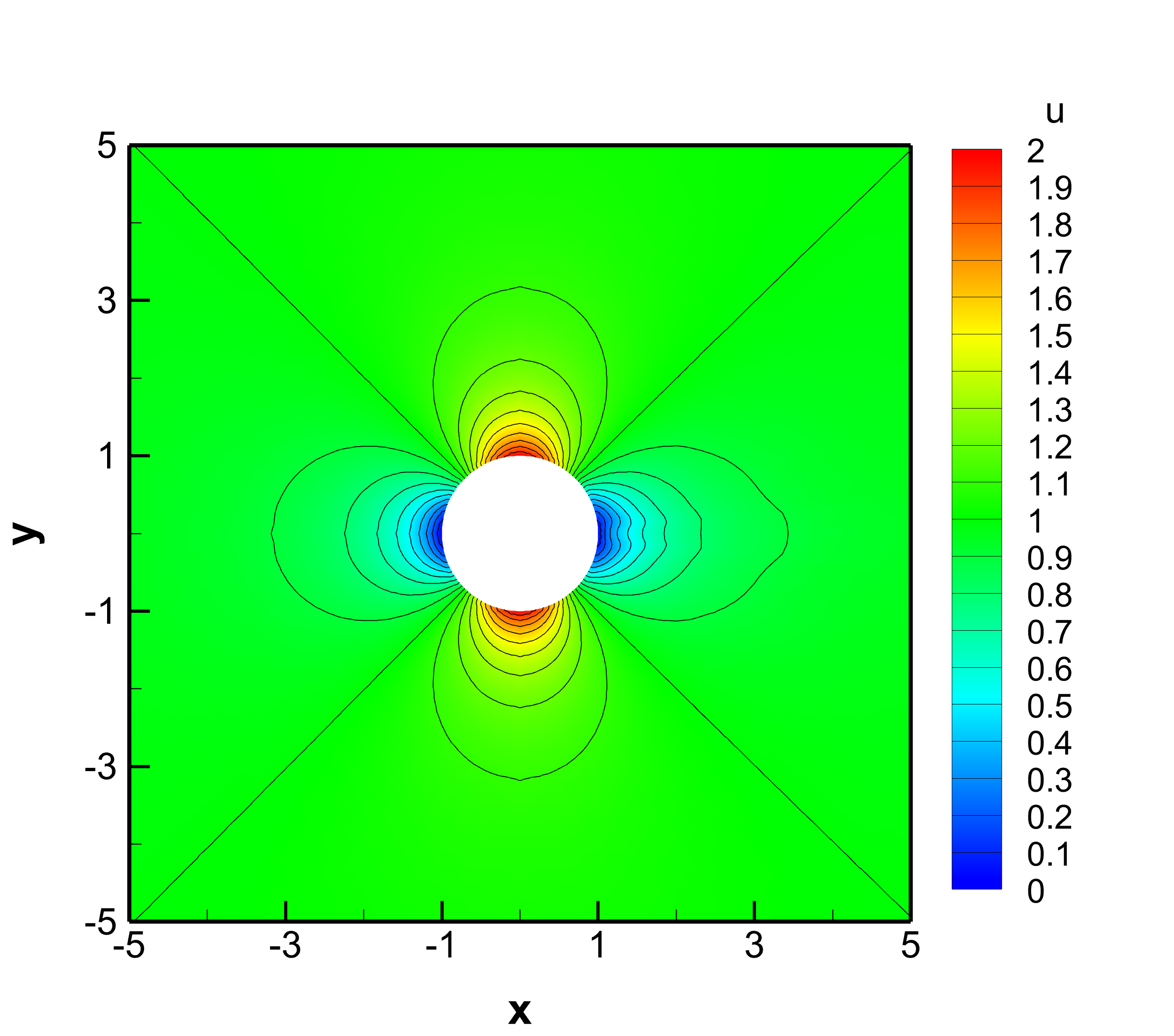}
		\end{minipage}
		\\[2mm]
		\begin{minipage}{0.48\textwidth}
			\centering
			\includegraphics[width=\textwidth,trim=15 12 200 815,clip]{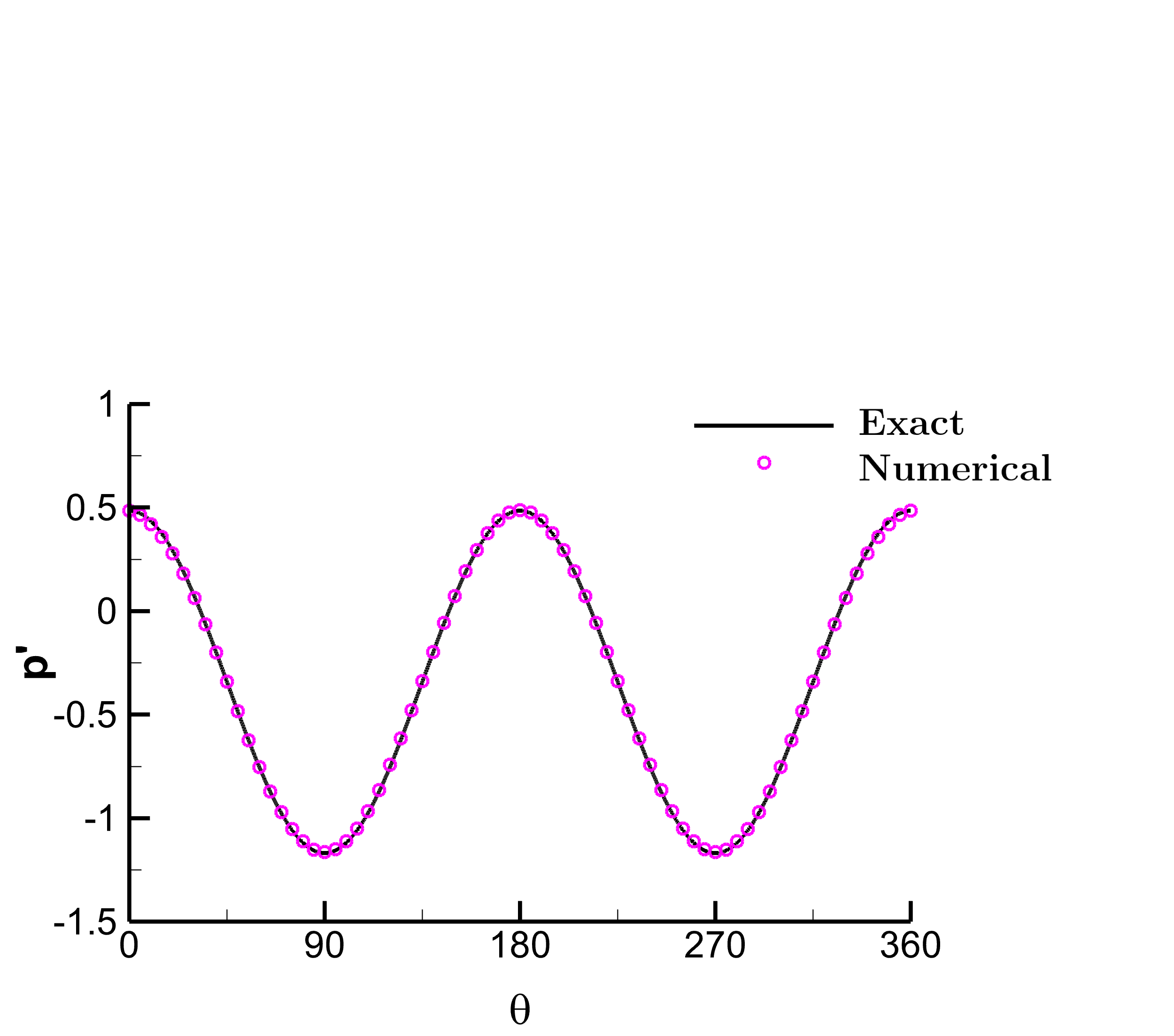}
		\end{minipage}
		\hfill
		\begin{minipage}{0.48\textwidth}
			\centering
			\includegraphics[width=\textwidth,trim=15 12 200 815,clip]{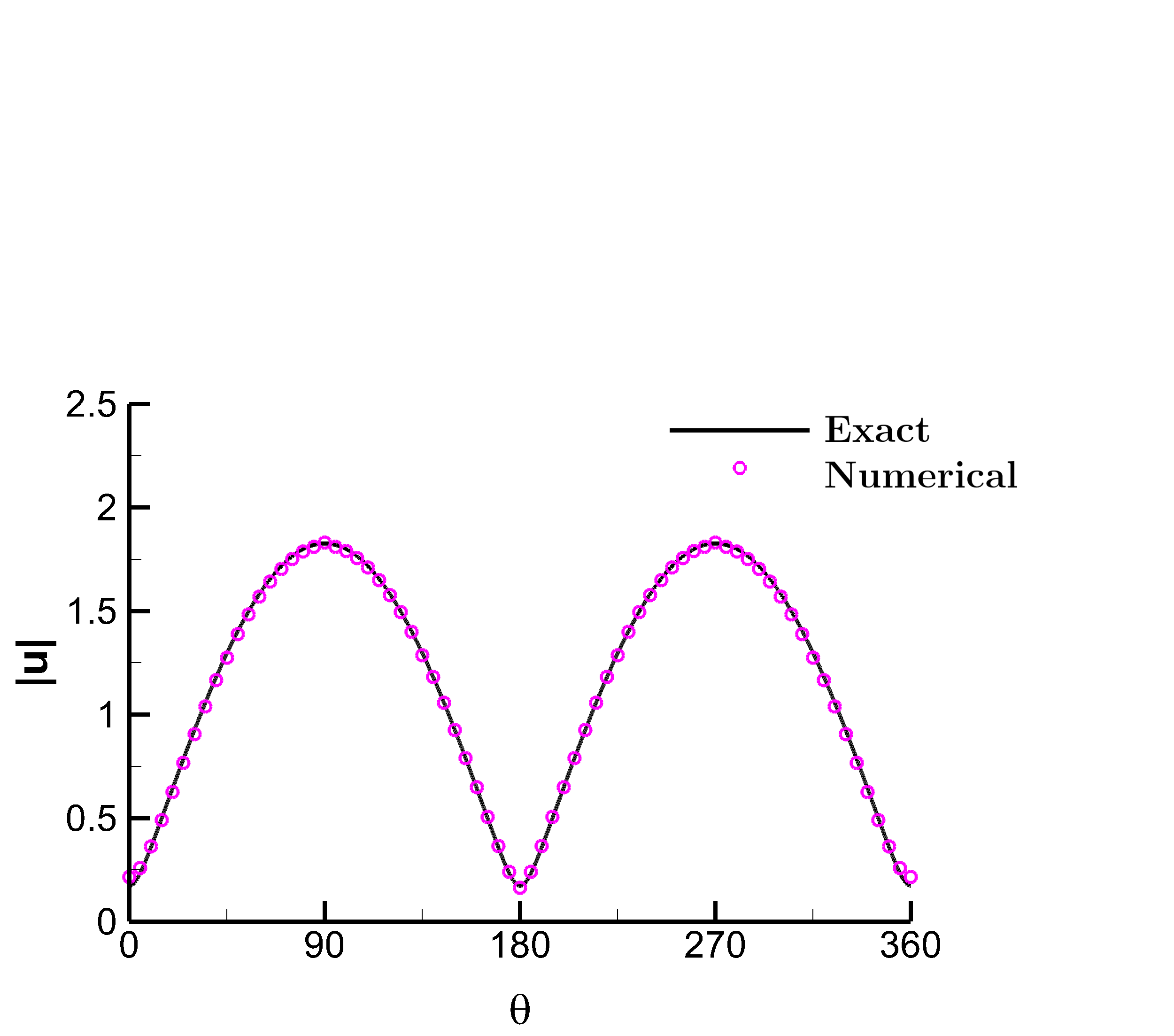}
		\end{minipage}
		\caption{Low Mach number flow around a circular cylinder at $M = 10^{-2}$ obtained with the second-order scheme at time $t_f = 20$. Top row: contours of the pressure deviation $p' = p - p_0$ with the computational mesh (left) and velocity component $u$ (right) in the region $[-5,5]^2$ around the cylinder. Bottom row: pressure deviation (left) and velocity magnitude $|\mathbf{u}|$ (right) along the circle $r = R + c$ with $c = 0.1$ as functions of the angle $\theta$ in degrees, measured counter-clockwise from the positive $x$ axis; numerical solution (symbols) and potential flow reference solution \eqref{eqn.cylinder.ic} (solid line).}
		\label{fig.cylinder}
	\end{figure}
	
	\subsection{Incompressible MHD vortex}
	\begin{figure}[!tb]
		\centering
		\begin{minipage}{0.48\textwidth}
			\centering
			\includegraphics[width=\textwidth,trim=15 12 100 185,clip]{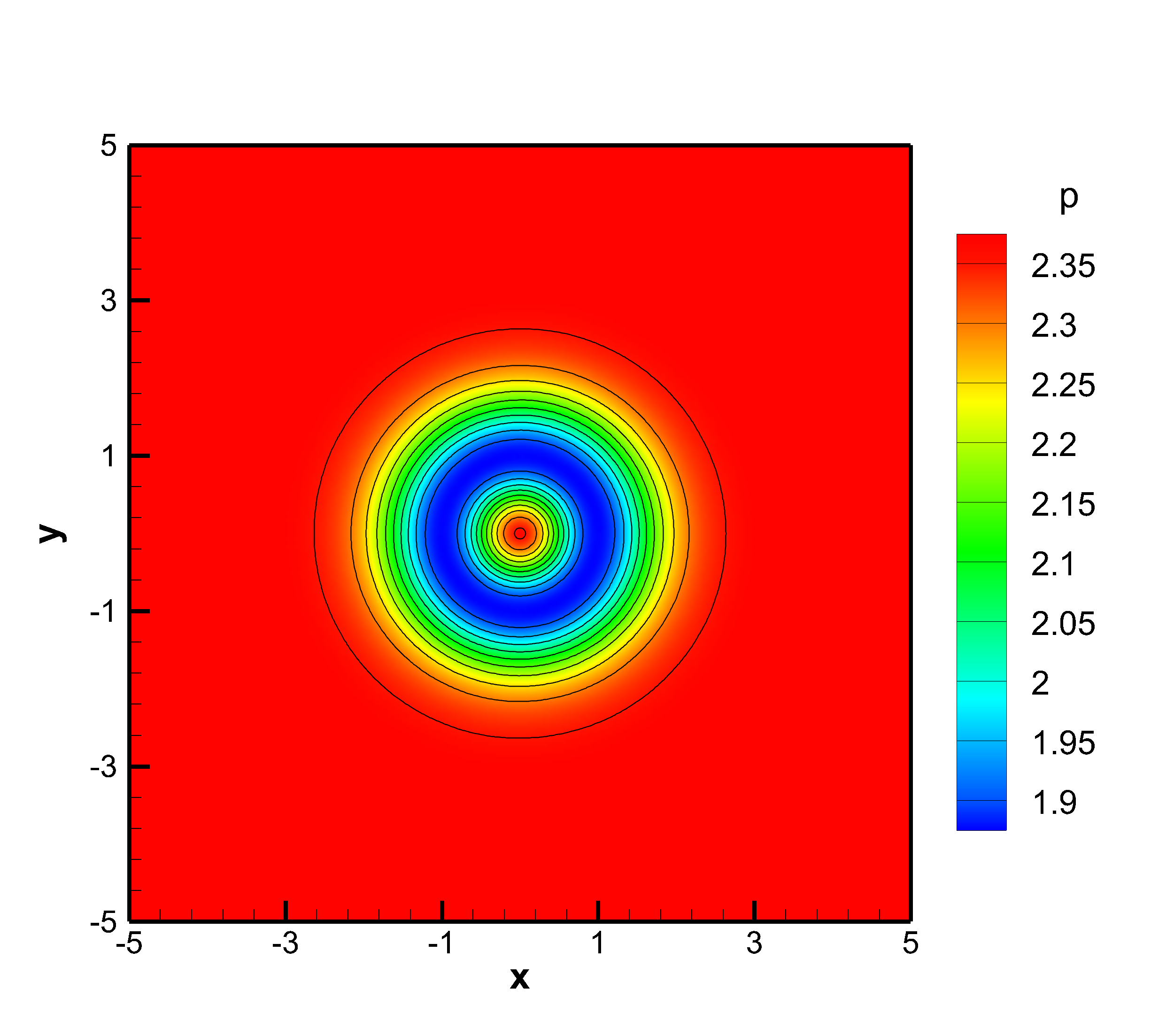}
		\end{minipage}
		\hfill
		\begin{minipage}{0.48\textwidth}
			\centering
			\includegraphics[width=\textwidth,trim=15 12 100 185,clip]{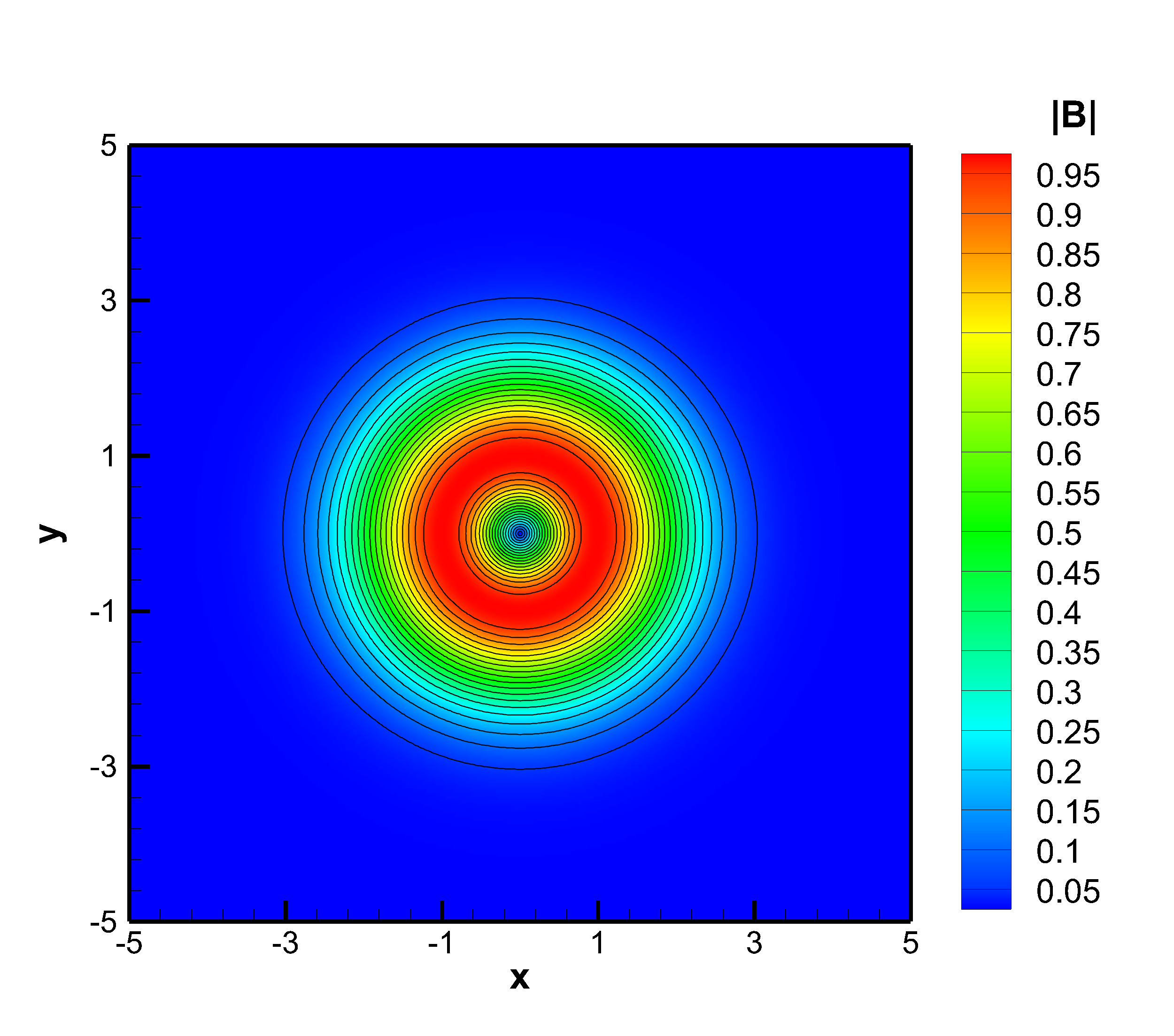}
		\end{minipage}
		
		\begin{minipage}{0.48\textwidth}
			\centering
			\includegraphics[width=\textwidth,trim=15 12 100 185,clip]{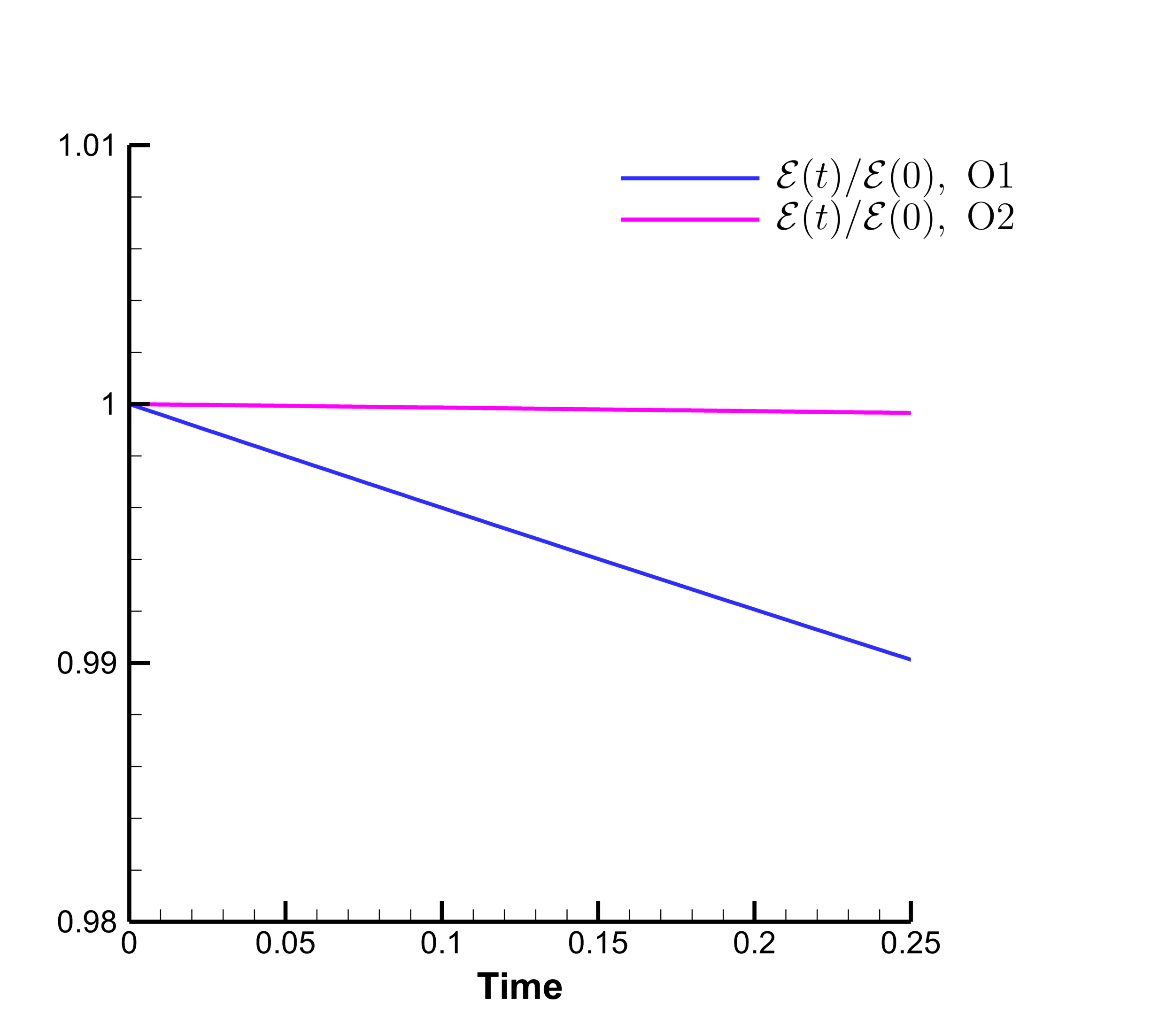}
		\end{minipage}
		\hfill
		\begin{minipage}{0.48\textwidth}
			\centering
			\includegraphics[width=\textwidth,trim=15 12 100 185,clip]{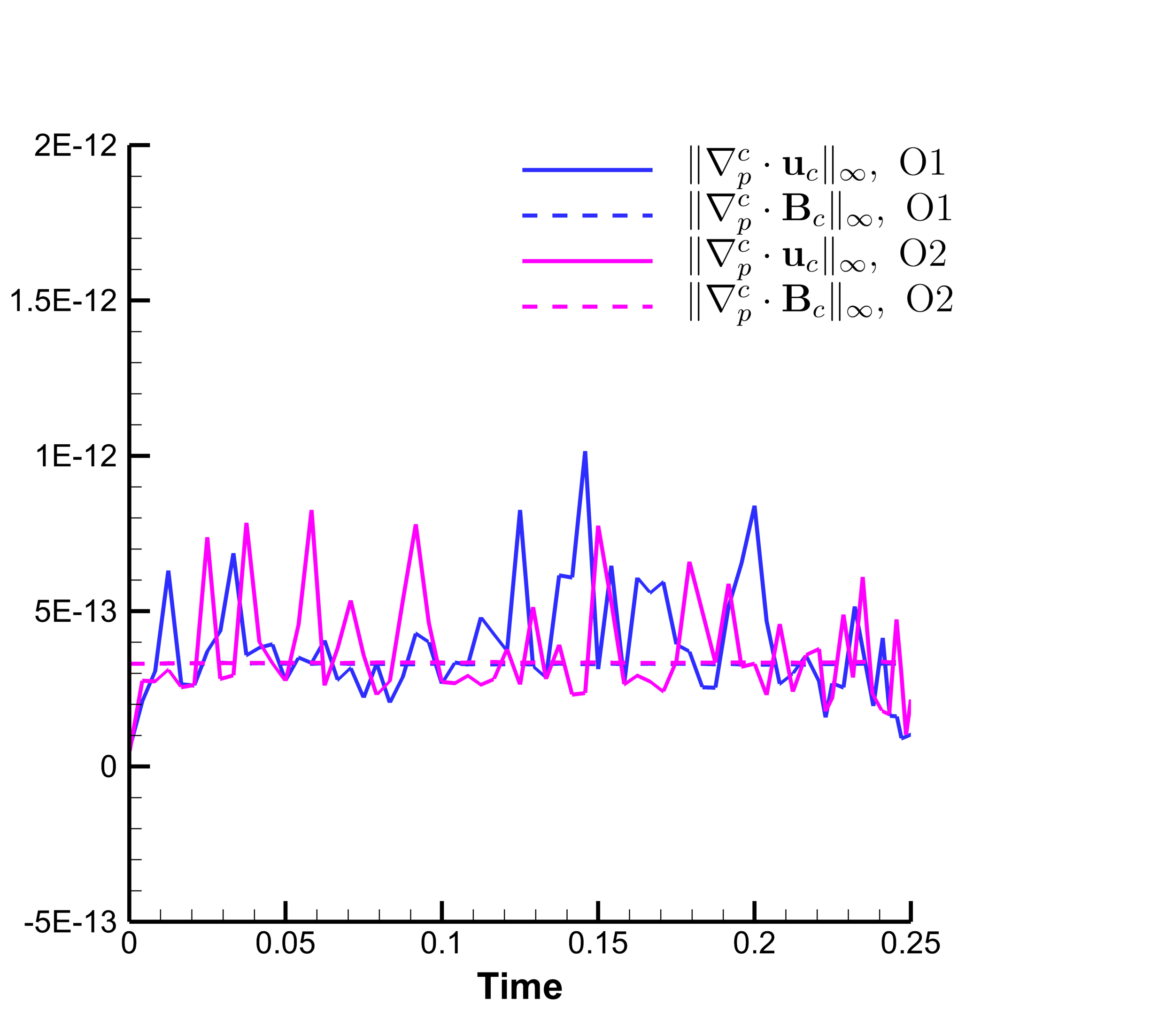}
		\end{minipage}
		\caption{Numerical solution of the MHD vortex at $t_f = 0.25$ using the new structure-preserving semi-implicit scheme on a mesh with $N_x = N_y = 200$. Pressure contours (top left) and contours of the magnitude of the magnetic field (top right) using the second-order scheme; colour flood with contour lines at the levels of the colour bar. Time evolution of the total energy $\mathcal{E}$, i.e. of the sum of the kinetic and the magnetic energy, normalized by its initial value (bottom left). Time evolution of the $L^\infty$ norm of the discrete velocity divergence $\nabla_p^c \cdot \mathbf{u}_c$ and of the discrete magnetic field divergence $\nabla_p^c \cdot \mathbf{B}_c$ (bottom right). In the bottom panels, the results are presented for both first- (O1) and second-order (O2) scheme.}
		\label{fig:CentralMHDVortec MHD system}
	\end{figure}
	\begin{table}[!tb]
		\centering
		\caption{Numerical convergence results of the \textit{first} order version of the new structure-preserving semi-implicit scheme for the \textit{incompressible} MHD system applied to a vortex problem. The $L^2$~error norms refer to the velocity components $u$ and $v$, the pressure $p$, and the magnetic field components $B_x$ and $B_y$ at $t_f = 0.25$ together with the observed order of decay $\mathcal{O}$.}
		\setlength{\tabcolsep}{3pt}   
		\begin{tabular}{ccccccccccc}
			$N_x = N_y$ & $L^2(u)$ & $\mathcal{O}(u)$ & $L^2(v)$ & $\mathcal{O}(v)$ & $L^2(p)$ & $\mathcal{O}(p)$ & $L^2(B_x)$ & $\mathcal{O}(B_x)$ & $L^2(B_y)$ & $\mathcal{O}(B_y)$ \\
			\hline
			20  & 2.42E-1 & -    & 2.48E-1 & -    & 3.20E-1 & -    & 1.23E-1 & -    & 1.22E-1 & -    \\
			40  & 1.36E-1 & 0.83 & 1.35E-1 & 0.88 & 1.81E-1 & 0.82 & 3.59E-2 & 1.78 & 3.57E-2 & 1.77 \\
			60  & 9.26E-2 & 0.95 & 9.30E-2 & 0.92 & 1.26E-1 & 0.88 & 1.68E-2 & 1.88 & 1.67E-2 & 1.87 \\
			80  & 7.14E-2 & 0.90 & 7.12E-2 & 0.93 & 9.84E-2 & 0.87 & 9.90E-3 & 1.83 & 9.90E-3 & 1.82 \\
			100 & 5.73E-2 & 0.98 & 5.75E-2 & 0.96 & 8.03E-2 & 0.91 & 6.33E-3 & 2.00 & 6.34E-3 & 2.00 \\
			\hline
		\end{tabular}
		\label{tab.MHD 1 order}
	\end{table}
	\begin{table}[!tb]
		\centering
		\caption{Numerical convergence results of the \textit{second} order version of the new structure-preserving semi-implicit scheme for the \textit{incompressible} MHD system applied to a vortex problem. The $L^2$~error norms refer to the velocity components $u$ and $v$, the pressure $p$, and the magnetic field components $B_x$ and $B_y$ at $t_f = 0.25$ together with the observed order of decay $\mathcal{O}$.}
		\setlength{\tabcolsep}{3pt}   
		\begin{tabular}{ccccccccccc}
			$N_x = N_y$ & $L^2(u)$ & $\mathcal{O}(u)$ & $L^2(v)$ & $\mathcal{O}(v)$ & $L^2(p)$ & $\mathcal{O}(p)$ & $L^2(B_x)$ & $\mathcal{O}(B_x)$ & $L^2(B_y)$ & $\mathcal{O}(B_y)$ \\
			\hline
			20  & 1.32E-1 & -    & 1.37E-1 & -    & 2.07E-1 & -    & 5.90E-2 & -    & 5.41E-2 & -    \\
			40  & 4.39E-2 & 1.59 & 4.55E-2 & 1.59 & 5.80E-2 & 1.83 & 1.21E-2 & 2.29 & 1.10E-2 & 2.29 \\
			60  & 2.26E-2 & 1.63 & 2.29E-2 & 1.70 & 2.57E-2 & 2.01 & 4.27E-3 & 2.57 & 4.04E-3 & 2.48 \\
			80  & 1.42E-2 & 1.61 & 1.35E-2 & 1.84 & 1.43E-2 & 2.04 & 2.24E-3 & 2.25 & 1.99E-3 & 2.46 \\
			100 & 9.34E-3 & 1.89 & 9.14E-3 & 1.74 & 8.98E-3 & 2.08 & 1.34E-3 & 2.30 & 1.17E-3 & 2.38 \\
			\hline
		\end{tabular}
		\label{tab.MHD 2 order}
	\end{table}

	The test case of the MHD vortex is used to solve the incompressible MHD system. The computational domain is the periodic square $\Omega = [-5, 5]^2$ discretized with a triangular mesh with $N_x = N_y = 200$ points along each boundary edge, the simulation is performed with CFL = 0.9 until the final time $t_f = 0.25$. The pressure is set to $p(\mathbf{x}) = 1 + \frac{1}{2}e -\frac{1}{2}r(\mathbf{x})^2e^{1-r(\mathbf{x})^2}$  where $r(\mathbf{x}) = \sqrt{x^2 + y^2}$ and the density is $\rho(\mathbf{x}) = 1$. The velocity field is defined as $\mathbf{u}(\mathbf{x}) = e^{\frac{1}{2}(1-r(\mathbf{x})^2)}[- y, x, 0]^T$ and the vector potential is set to $\mathbf{A}(\mathbf{x}) = e^{\frac{1}{2}(1-r(\mathbf{x})^2)}[0, 0, -1]^T$. The initial magnetic field is obtained as the discrete curl of the discrete vector potential, i.e. $\mathbf{B}^0_c = \nabla_c^p \times \mathbf{A}_p$ and the electric field is computed as in \eqref{eqn.mhd.induction}. In the second-order scheme, the nodal average in \eqref{eqn.mhd.induction} is computed with the weights $|l_{pc} \mathbf{n}_{pc}| / \sum_{c' \in \mathcal{C}(p)} |l_{pc'} \mathbf{n}_{pc'}|$ and the resistive term is evaluated with the node-extrapolated states of the piecewise linear reconstruction. The dissipation coefficient of the Rusanov flux \eqref{eqn.rusanov} is $s_{ac} = 2 \max \left( |\mathbf{u}_a|, |\mathbf{u}_c| \right)$. 
	Figure \ref{fig:CentralMHDVortec MHD system} reports the pressure and the magnetic field distributions, together with the time evolution of the total energy and of the discrete divergence of the velocity and the magnetic fields. The divergence constraints remain satisfied up to machine precision throughout the entire simulation for both first and second order of accuracy, confirming the structure-preserving property of our proposed method.
	We then run the same test again but on a sequence of successively refined unstructured triangular meshes using $N_x = N_y = 20, 40, 60, 80$ and 100 points on each boundary edge to verify the order of convergence for both the first and second order version of our scheme. We report the obtained results in Table~\ref{tab.MHD 1 order} and Table~\ref{tab.MHD 2 order}, confirming that both variants of the method achieve their designed order of accuracy.


	\begin{figure}[!tb]
		\centering
		\begin{minipage}{0.32\textwidth}
			\centering
			\includegraphics[width=\textwidth,trim=15 12 70 185,clip]{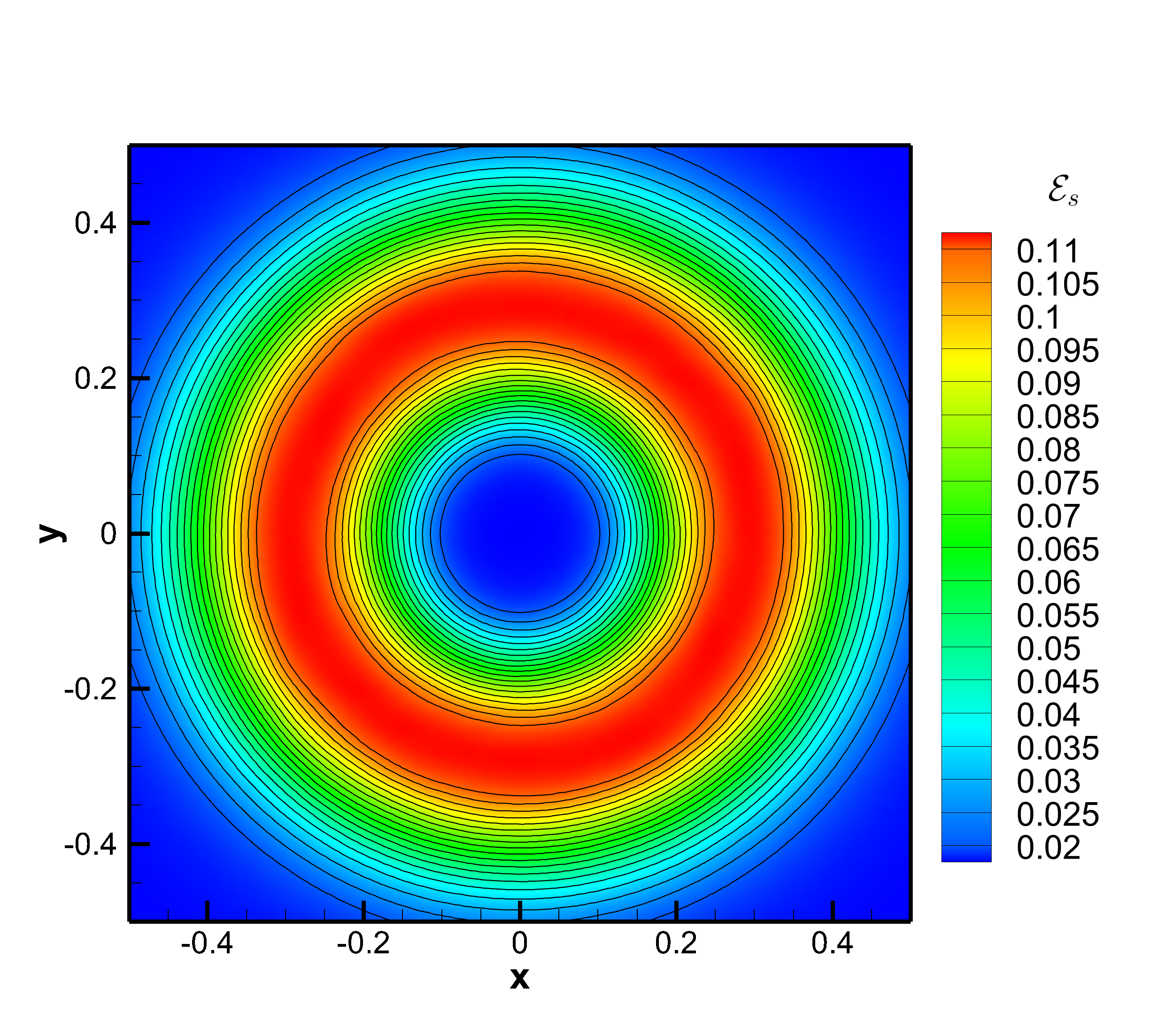}
		\end{minipage}
		\hfill
		\begin{minipage}{0.32\textwidth}
			\centering
			\includegraphics[width=\textwidth,trim=15 12 70 185,clip]{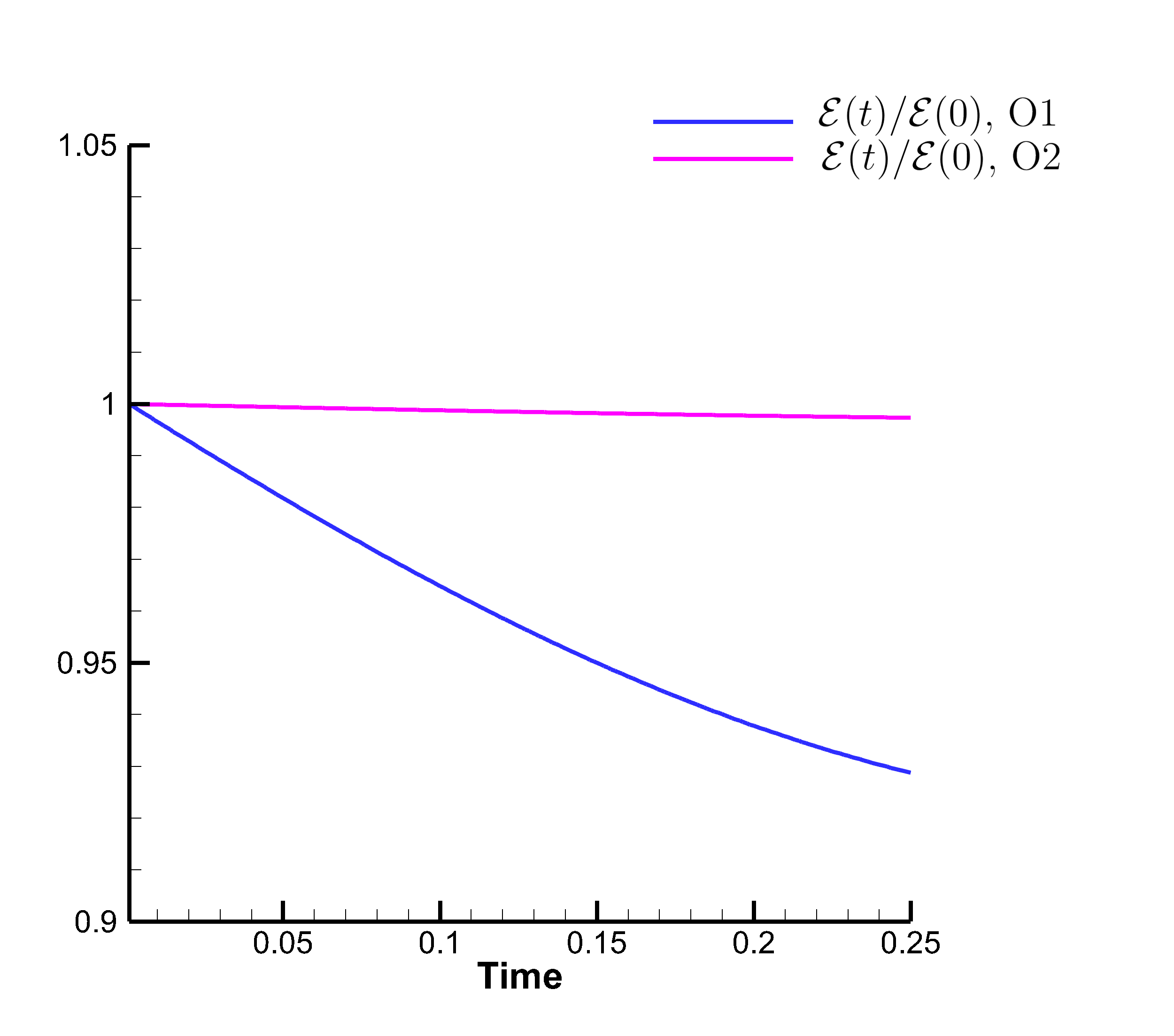}
		\end{minipage}
		\hfill
		\begin{minipage}{0.32\textwidth}
			\centering
			\includegraphics[width=\textwidth,trim=15 12 70 185,clip]{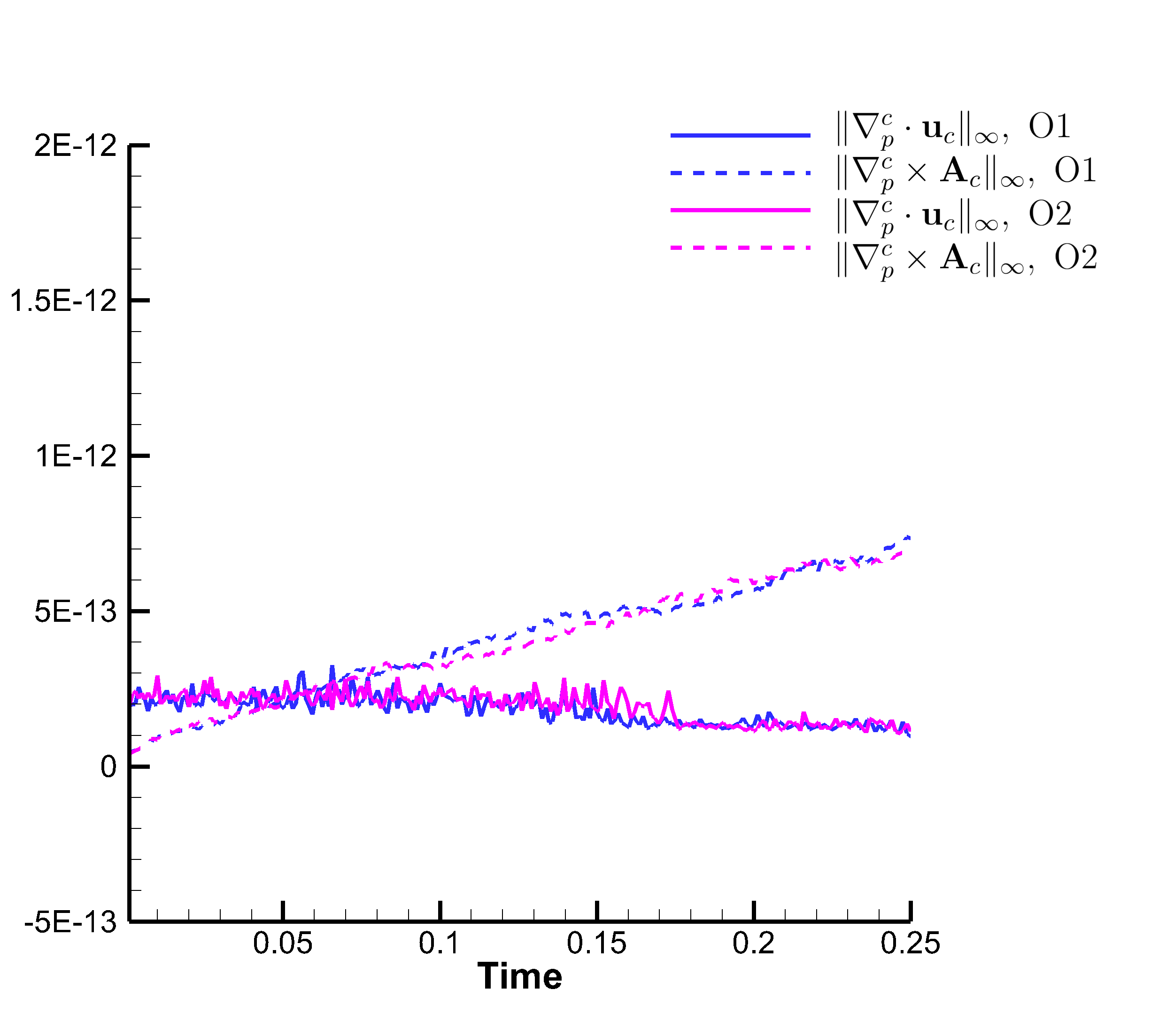}
		\end{minipage}
		\caption{Numerical solution of the smooth solid rotor problem obtained using the new structure-preserving semi-implicit scheme on a mesh with $N_x = N_y = 200$. 
			Elastic energy density $\mathcal{E}_s$ at $t_f = 0.25$ using the second-order scheme, shown in the subdomain $[-0.5, 0.5]^2$; colour flood with contour lines at the levels of the colour bar (left). Time evolution of the total energy $\mathcal{E}$, i.e. of the sum of the kinetic and the elastic energy, normalized by its initial value (center), and of the $L^\infty$ norm of the discrete velocity divergence $\nabla_p^c \cdot \mathbf{u}_c$ and of the discrete curl of the distortion field $\nabla_p^c \times \mathbf{A}_c$ (right). In the center and right panels, the results are presented for both first- (O1) and second-order (O2) scheme.}
		\label{fig:Solid rotor for GPR}
	\end{figure}
	\subsection{Smooth solid rotor}
	In the smooth variant of the solid rotor test case, see e.g. \cite{boscheri2021structure}, the computational domain is the periodic square $\Omega = [-1, 1]^2$ discretized with a triangular mesh with $N_x = N_y = 200$ points on each boundary edge, the final simulation time is set to $t_f = 0.25$ and we choose CFL = 0.9. Except for the velocity field, all variables are initially constant throughout the domain. Specifically we set $\rho(\mathbf{x}) = 1$, $p(\mathbf{x}) = 1$, $\mathbf{A}(\mathbf{x}) = \mathbf{I}$, while the velocity field is $\mathbf{u}(\mathbf{x}) = \frac{1}{R} \, e^{\frac{1}{2}\left(1 - r^2(\mathbf{x})/R^2\right)} [-y, x, 0]^T$ with $r(\mathbf{x}) = \sqrt{x^2 + y^2}$ and $R = 0.2$. The characteristic speed of shear waves is $c_s = 0.25$. The dissipation coefficient $s_{ac}$ of the Rusanov flux is the same as for the MHD vortex. 
	The purpose of this test case is to assess the ability of our compatible semi-implicit scheme to impose the divergence-free constraint of the velocity field maintaining at the same time also the curl-free involution associated with the distortion field in the incompressible GPR model. The corresponding numerical results are shown in Figure~\ref{fig:Solid rotor for GPR}. The figure shows the elastic energy density at the final time, together with the time evolution of the total energy normalized by its initial value and of the discrete divergence of the velocity field and of the discrete curl of the distortion field. As expected, both differential constraints are satisfied up to machine precision throughout the simulation for both first and second order of accuracy, confirming the structure-preserving properties of our discretization.
	
	\section{Conclusions}
	\label{conclusion}
	
	In this paper, we have designed a new compatible, structure-preserving semi-implicit finite volume method on vertex-staggered unstructured meshes for the solution of a broad spectrum of nonlinear systems of time-dependent partial differential equations. 
	The semi-implicit nature of the scheme, which discretizes only the pressure term and, optionally, the viscous terms implicitly and all other terms explicitly, allows one to obtain an asymptotic-preserving method that is consistent with the low Mach number limit of the equations. The resulting linear algebraic system for the pressure is symmetric and positive definite. 
	Moreover, since the time step restriction concerns only the explicit subsystem and does not depend on the sound speed, we can ensure robust performance across all Mach number regimes. Thanks to the use of compatible discrete differential nabla operators on the primal and dual meshes, the method is exactly divergence-free and curl-free. 
	We have shown that the first order version of the scheme is energy stable at the fully discrete level for the incompressible Euler, Navier-Stokes, MHD and GPR systems, provided that the numerical dissipation and the time step satisfy suitable conditions, while the energy stability analysis of the second order scheme and of the other PDE systems contained in this paper is left to future research. 
	
	Since our new SP method is up to now at most second-order accurate in space and time, future work will focus on its extension to higher orders of accuracy. Possible directions include the use of better than piecewise linear polynomial reconstructions in space in combination with suitable IMEX Runge--Kutta time integrators \cite{PareschiRusso2000,BPR2017}. We also consider the use of higher-order structure-preserving space discretization operators  such as those recently introduced in \cite{abgrall2025simple}, where the cell-centered quantities are represented via discontinuous piecewise polynomials of degree $N$ and the nodal quantities are represented by piecewise continuous polynomials of degree $N+1$.
	
	We will also investigate the extension of the proposed approach to the full compressible Euler equations along the ideas outlined in \cite{park2005multiple,dumbser2016conservative,boscheri2021structure,TavelliDumbser2017}.
	Finally, we want to exploit some of the knowledge acquired in this work to develop novel multidimensional Riemann solvers, such as those of~\cite{gallice2022entropy,GaburroMultiDRS} with structure-preserving capabilities, to be inserted in high-order explicit ALE schemes on moving polygonal or polyhedral meshes such as those of~\cite{gaburro2020high,gaburro2021high,gaburro2021unified,gaburro2025high,bonafini2026stability,gaburro2027treatment}. 
	
	\section{Acknowledgments}
	
	All authors are members of the INdAM GNCS group in Italy and are very grateful for the  thorough review and the constructive comments of the two anonymous referees which greatly helped to improve the quality and clarity of our manuscript. 
	
	E.~Gaburro and E.~Bernardelli gratefully acknowledge the support received from the European Union 
	with the ERC Starting Grant \textit{ALcHyMiA} (grant agreement No. 101114995).
	
	M.~Dumbser was financially supported by the Italian Ministry of University
	and Research (MUR) in the framework of the Departments of Excellence  Initiative 2018--2027 attributed to DICAM of the University of Trento (grant L. 232/2016), by the Fondazione Caritro via the project SOPHOS and by the European Research Council (ERC) under the European Union's Horizon Europe research and innovation programme via the project SOPHOS, grant agreement no. ERC-ADG-2025-101265878-SOPHOS.
	
	Views and opinions expressed are however those of the authors only and do not necessarily reflect those of the European Union or the European Research Council Executive Agency. Neither the European Union nor the granting authority can be held responsible for them.

	

	\bibliographystyle{plain}
	\bibliography{./referencesFuCHSiA-1}

\end{document}